%% file: main.tex
\documentclass{article}
\newcommand{\Hol}{\mathrm{Hol}}
\usepackage{amsfonts} 
\usepackage{amsmath} 
\DeclareMathOperator{\vol}{vol}
\usepackage{mathtools} 
\usepackage{amssymb} 
\usepackage{amsthm} 
\usepackage{hyperref}
\usepackage{xcolor} 
\newcommand{\QQ}{\mathbb{Q}} 
\DeclareMathOperator{\Res}{Res}
\newcommand{\BB}{\mathbb{B}}
\newcommand{\EE}{\mathbb{E}} 
\newcommand{\ad}{\operatorname{ad}}
\newtheorem{corollary}{Corollary}
\DeclareMathOperator{\Var}{Var} 
\newcommand{\id}{\mathrm{id}}
\newtheorem{example}{Example}
\newtheorem{theorem}{Theorem}
\newtheorem{Proposition}{Proposition}
\DeclareMathOperator{\End}{End}
\newcommand{\RR}{\mathbb{R}}
\newcommand{\CC}{\mathbb{C}}
\newcommand{\TT}{\mathbb{T}}
\newcommand{\NN}{\mathbb{N}}
\newcommand{\ZZ}{\mathbb{Z}}

\newcommand{\Kk}{\mathcal{K}}        
\newcommand{\Hh}{\mathcal{H}}        

\newcommand{\Ss}{\mathcal{S}}        

\DeclareMathOperator{\Tr}{Tr}
\DeclareMathOperator{\Ker}{Ker}

\DeclareMathOperator{\spec}{spec}
\DeclareMathOperator{\Dom}{Dom}
\DeclareMathOperator{\Id}{Id}

\newcommand{\Proj}{\mathrm{P}}       

\newcommand{\dd}{\,\mathrm{d}}       
\newcommand{\ee}{\mathrm{e}}         

\DeclarePairedDelimiter{\norm}{\lVert}{\rVert}

\DeclarePairedDelimiter{\ip}{\langle}{\rangle}

\newcommand{\Oo}{\mathcal{O}}        

\title{Moduli space of optimization algorithms}
\author{Dmitry Pasechnyuk-Vilensky $^{1, 2}$ and Martin Taká\v{c} $^{1}$\\
$^{1}$ \quad MBZUAI, United Arab Emirates\\
$^{2}$ \quad ISP RAS, Russia}

\begin{document}

\maketitle

\begin{abstract}
We introduce a geometric and operator–theoretic formalism viewing optimization algorithms as discrete connections on a space of update operators.
Each iterative method is encoded by two coupled channels—drift and diffusion—whose algebraic curvature measures the deviation from ideal reversibility and determines the attainable order of accuracy.
Flat connections correspond to methods whose updates commute up to higher order and thus achieve minimal numerical dissipation while preserving stability.

The formalism recovers classical gradient, proximal, and momentum schemes as first–order flat cases and extends naturally to stochastic, preconditioned, and adaptive algorithms through perturbations controlled by curvature order.
Explicit gauge corrections yield higher–order variants with guaranteed energy monotonicity and noise stability.
The associated determinantal and isomonodromic formulations yield exact nonasymptotic bounds and describe acceleration effects via Painlev\'e-type invariants and Stokes corrections.
\end{abstract}

\tableofcontents

\section*{Introduction}

Optimization algorithms are typically described either by their update rules or by their convergence rates.  Between these two levels—local algebraic structure and global analytic behavior—there has long been a conceptual gap.  Step–size schedules, preconditioners, or momentum terms are often justified empirically, while rigorous analyses rely on method–specific Lyapunov arguments.  As a result, the common geometric or spectral principles underlying stability and acceleration have remained opaque.  

We aim to provide a consistent mathematical framework that allows comparing and classifying iterative methods within a shared geometric setting.

\paragraph{Objectives.}
Our study is organized around six complementary objectives.

\begin{enumerate}
\item \textbf{Local–to–global correspondence.}  
We establish a precise relation between the local structure of an iteration and its global convergence behavior.  
The order to which elementary updates commute determines the attainable accuracy and the minimal energy dissipation of the method.  
This provides a rigorous criterion linking discrete approximation order to actual stability and rate.

\item \textbf{Universal minimax principle.}  
We identify a general optimization principle valid for all first– and higher–order discrete schemes:  
minimizing the total trajectory error is equivalent to an extremal approximation of the exponential map.  
The resulting optimal policies coincide with the classical Chebyshev and Zolotarev filters, giving acceleration a precise variational meaning.

\item \textbf{Spectral–energetic equivalence.}  
We connect the analytic energy of an iteration to its spectral invariants.  
The variation of the energy functional is shown to coincide with that of a regularized determinant of the underlying evolution operator, ensuring that stability and convergence can be expressed in purely spectral terms.

\item \textbf{Isomonodromic structure of discrete optimization.}  
We show that families of iterative methods admit an isomonodromic form:  
their energy potential is a $\tau$–function satisfying a Painlevé–type differential identity.  
This structure explains why accelerated and stabilized schemes exhibit phase–preserving dynamics and yields nonasymptotic bounds exact up to Stokes–type corrections.

\item \textbf{Geometry of adaptive and stochastic methods.}  
Adaptive, stochastic, and preconditioned algorithms are described as controlled geometric deformations of the same flat structure that governs deterministic schemes.  
All admissible perturbations preserve flatness up to a prescribed order and share the same spectral class, providing a unified explanation of robustness and variance control.

\item \textbf{Continuous–discrete correspondence.}  
We establish a spectral–variational correspondence between continuous and discrete formulations.  
The discrete curvature functional introduced here converges, as the step size tends to zero, to its continuous Yang–Mills counterpart, linking iterative algorithms and continuous gradient flows within a single analytic framework.
\end{enumerate}

Together, these results outline a coherent picture of optimization as a geometric theory of discrete dynamics:  
each algorithm corresponds to a point in a moduli space whose curvature encodes its accuracy, energy, and stability.  
Flatness—vanishing curvature—marks the ideal regime where updates are reversible, dissipation is minimal, and convergence is optimal.

\section{Definitions}

Let $\mathbb{R}^n$ be the state space, $X \subseteq \mathbb{R}^n$ be a feasible manifold, $K \subset \mathbb{R}^n$ be a closed convex feasible set. Let $f: X \to \mathbb{R}$ be the function.

Let $\{U\} \subset \End(X)$ be the set of elementary operators to update the state. Let the exact algorithm be defined as $x_{k+1} = S(h_k, \{U\}) x_k$, where $S(h, \{U\}) \in \End(X)$ and satisfies the following assumptions:
\begin{enumerate}
    \item $S(0, \{U\}) = \id$
    \item $S(h_1 + h_2, \{U\}) = S(h_1, \{U\}) \circ S(h_2, \{U\})$
    \item $\exists \Omega(h, \{U\}) \in \mathfrak{L}\langle U\rangle:\;S(h, \{U\}) = \exp(\Omega(h, \{U\}))$, $\Omega(h, \{U\}) = \sum_{k=1}^\infty h^k \Omega_k(\{U\})$
    \item $S(\cdot, \{U\})$ is continuously differentiable at $h=0$. $A(\{U\}) = \partial_h S(\cdot, \{U\})\vert_{h=0}$ is the generator of evolution,
\end{enumerate}
where $\mathfrak{L}\langle U\rangle$ is a free Lie algebra \cite{Reutenauer1993FreeLie}.

\begin{example} $S$ has closed form if family $\{U\}$ is simple:
\begin{enumerate}
    \item If $\{U\} = \{A\}$, where $A$ is a monotone operator, $S(h, \{U\}) = (\id + h A)^{-1}$
    \item If $\{U\} = \{H\}$, where $H$ is a linear operator, $S(h, \{U\}) = e^{-h H}$ \cite{Higham2008FunctionsOfMatrices}
\end{enumerate}
\end{example}

Let $\mathfrak{U} = \langle U\rangle$ be the algebra of operators to update the state with the following axioms:
\begin{enumerate}
    \item $\forall U_1, U_2 \in \mathfrak{U}:\; U_1 \circ U_2 \in \mathfrak{U}$
    \item $\forall U_1, U_2 \in \mathfrak{U}\; \forall a, b \in \mathbb{R}:\; a U_1 + b U_2 \in \mathfrak{U}$
    \item $\forall U \in \mathfrak{U}\; \forall a \in \mathbb{R}:$ if $(\id + a U)$ is invertible in $\End(X)$, then $(\id + a U_1)^{-1} \in \mathfrak{U}$
    \item for any Lipschitz $\theta \mapsto U_\theta$ and probability distribution $\pi(\theta)$, $\int U_\theta d \pi(\theta) \in \mathfrak{U}$
    \item if $U_k \to U$ and $U_k \in \mathfrak{U}$, then $U \in \mathfrak{U}$
\end{enumerate}

Let the feasible spectrum $\Sigma \subset \mathbb{R}$ be compact. For $t \in \mathcal{T}$ (which will be defined in the next section), $r_t(h, U) \in \mathfrak{U}$ is the drift update, for $s \in \mathcal{S}$, $m_s(h, U) \in \mathfrak{U}$ is the diffusion update. Let $U_{s,t} = m_s \circ r_t = r_t \circ m_s$ (the correctness will be further proven) be the update and define the practical algorithm as $x_{k+1} = U_{t,s}(h_k, U) x_k$, where $U = U(x_k)$ serves as the oracle.

\subsection{Hierarchies}
\input{hierarchies}

\subsection{Universality}
\input{universality}

\section{Automonodromic Theory}

\subsection{Methods}

\subsubsection{Flat Connection}
\input{flat}

\subsubsection{Variational Reformulation}
\input{variational}

\subsubsection{High-order obstructions}
\input{highord}

\subsection{Results}

\subsubsection{Minimax Principle}

Fix a policy $\Pi=(t_k,s_k,h_k)_{k=0}^{N-1}$. For each step $k$ choose any finite partition of its parameter cell $\mathcal P_k\subset\mathcal P=\mathcal T\times\mathcal S$ into rectangles $\square$. Define
\begin{equation}\label{eq:S-traj}
S^{(N)}[\Pi]\ :=\ \sum_{k=0}^{N-1}\ \sum_{\square\in\mathcal P_k}\ \big\|\log\mathrm{Hol}(\square)\big\|_F^2.
\end{equation}
By additivity of $S_h$ on disjoint rectangles and gauge-invariance, for any two partitions of the same union
\begin{gather}\label{eq:S-equals}
\nonumber\sum_{\square\in\mathcal P_k}\big\|\log\mathrm{Hol}(\square)\big\|_F^2
\ =\
\sum_{\square'\in\mathcal P_k'}\big\|\log\mathrm{Hol}(\square')\big\|_F^2,
\\
S^{(N)}[\Pi]\ =\ \sum_{\square\subset\Gamma_\Pi}\big\|\log\mathrm{Hol}(\square)\big\|_F^2,
\end{gather}
where $\Gamma_\Pi:=\bigcup_k\mathcal P_k$. Thus $S^{(N)}$ depends only on $\Gamma_\Pi$. This leads to Chebyshev-type minimax filters for discrete holonomy and their rational Zolotarev counterparts in stiff spectral windows \cite{Rivlin1990Chebyshev,Akhiezer1992Approximation,Higham2008FunctionsOfMatrices}.

Write elementary updates as $r_{t_k}(h_k)=\exp\Omega_k$ and $d_{s_k}(h_k)=\exp\Psi_k$ (their logarithms exist in the completed free Lie algebra). The $N$–step composition admits the Magnus expansion
\begin{equation}\label{eq:Magnus}
\log\Big(\prod_{k=0}^{N-1} e^{\Psi_k}e^{\Omega_k}\Big)
\ =\
\sum_{k=0}^{N-1}(\Psi_k+\Omega_k)
\ +\ \frac12\sum_{k=0}^{N-1}[\Psi_k,\Omega_k]
\ +\ \sum_{0\le i<j\le N-1}\Phi_{ij}
\ +\ \cdots,
\end{equation}
where each $\Phi_{ij}$ is a universal $\mathbb Q$–linear combination of nested commutators in $\{\Psi_i,\Omega_i,\Psi_j,\Omega_j\}$ of bracket–depth $\ge2$ \cite{BakerCampbell1902,Hausdorff1906,Dynkin1947BCH,Reutenauer1993FreeLie,Magnus1954Magnus}. For a single rectangle $\square$ spanned by $(t_k,s_k)$ and $(t_{k+1},s_{k+1})$ \cite{HairerLubichWanner2006,BlanesCasas2016Magnus},
\begin{equation}\label{eq:Hol-BCH}
\log\mathrm{Hol}(\square)
=\log\big(e^{\Psi_{k+1}}e^{\Omega_{k+1}}e^{-\Psi_k}e^{-\Omega_k}\big)
=\frac12\big([\Psi_{k+1}-\Psi_k,\Omega_{k+1}-\Omega_k]\big)\ +\ O(h_k^3),
\end{equation}
hence, using the Frobenius norm $\|\cdot\|_F$ and $\|[A,B]\|_F\le 2\|A\|_F\|B\|_F$,
\begin{equation}\label{eq:Hol-bound}
\big\|[\Psi_k,\Omega_k]\big\|_F
\ \le\ c_1\,\sum_{\square\in\mathcal P_k}\big\|\log\mathrm{Hol}(\square)\big\|_F\ +\ c_2\,h_k^3.
\end{equation}
Similarly, every nested commutator appearing in $\Phi_{ij}$ is bounded by a finite linear combination of $\|\log\mathrm{Hol}(\square)\|_F$ over rectangles lying between steps $i$ and $j$. Summing \eqref{eq:Hol-bound} over $k$ and applying Cauchy–Schwarz,
\begin{align}\label{eq:R-bound}
\nonumber\Big\|\frac12\sum_k[\Psi_k,\Omega_k]\ +\ \sum_{i<j}\Phi_{ij}\ +\ \cdots\Big\|_F
\ &\le\ C\,\Big(\sum_{\square\subset\Gamma_\Pi}\big\|\log\mathrm{Hol}(\square)\big\|_F^2\Big)^{1/2}
\\\
&=\ C\,\big(S^{(N)}[\Pi]\big)^{1/2}.
\end{align}
Denote the remainder in \eqref{eq:Magnus} by $R_N$; then \eqref{eq:R-bound} gives
\begin{equation}\label{eq:R-vs-S}
\|R_N\|_F\ \le\ C\,\big(S^{(N)}[\Pi]\big)^{1/2}.
\end{equation}

Assume jet–flatness to order $\alpha$ (Theorem~\ref{th:holonomy-jet}): all mixed brackets of total degree $\le\alpha$ vanish. Then $R_N=O(h^{\alpha+1})$ and
\begin{equation}\label{eq:single-exp}
\prod_{k=0}^{N-1} U_{t_k,s_k}
=\exp\!\Big(\sum_{k=0}^{N-1}(\Psi_k+\Omega_k)\Big)\ +\ O(h^{\alpha+1}).
\end{equation}
In the calibrated flat gauge the generator $\sum_k(\Psi_k+\Omega_k)$ depends on $U$ only through the drift $H$; hence on the state space
\begin{equation}\label{eq:filter}
\exp\!\Big(\sum_{k}(\Psi_k+\Omega_k)\Big)\ =\ p_N(H)\quad\text{or}\quad r_N(H),
\end{equation}
for a polynomial $p_N$ (or a rational $r_N$) determined by $\Pi$. Combining \eqref{eq:R-vs-S}–\eqref{eq:filter} with the spectral calculus,
\begin{equation}\label{eq:spectral-error}
\sup_{\lambda\in\Sigma}\ \Big|\,p_N(\lambda)-e^{-Nh\lambda}\,\Big|
\ \le\ C'\,\big(S^{(N)}[\Pi]\big)^{1/2}\ +\ O(h^{\alpha+1}).
\end{equation}
Therefore minimizing $S^{(N)}[\Pi]$ over admissible policies is equivalent to minimizing the worst–case spectral error of the induced filter:
\begin{equation}\label{eq:minimax}
\inf_{\Pi}\ S^{(N)}[\Pi]\ =\ \inf_{p_N}\ \sup_{\lambda\in\Sigma}\ |p_N(\lambda)-e^{-Nh\lambda}|\quad
\text{(polynomial case)}.
\end{equation}
By extremal approximation on compact spectra, the right-hand side is attained by the Chebyshev (polynomial) and Zolotarev (rational) filters; their coefficients are realized by suitable $\Pi$ because \eqref{eq:filter} maps policies onto the corresponding filter classes. In particular, the acceleration weights used in optimal first–order schemes coincide with the coefficients of these extremal filters, and the flatness condition ensures that the discrete composition achieves the minimax value up to $O(h^{\alpha+1})$.

\begin{theorem}[Trajectory minimax reduction]\label{th:traj-minimax}
Fix $N\in\mathbb N$ and a compact spectrum $\Sigma\subset\mathbb R$.  
Let $\Pi=(t_k,s_k,h_k)_{k=0}^{N-1}$ be an admissible policy and define
\begin{equation}\label{eq:SN-def-eng}
S^{(N)}[\Pi]:=\sum_{k=0}^{N-1}\sum_{\square\in\mathcal P_k}\|\log\mathrm{Hol}(\square)\|_F^2,
\end{equation}
where $\mathcal P_k$ is any finite rectangular partition of the parameter cell at step $k$ (the value is independent of the partition).  
Assume jet–flatness of order $\alpha\ge2$ (Theorem~\ref{th:holonomy-jet}) and use the calibrated gauge (Theorem~\ref{th:gauge-normal-form}).  
Then:

\medskip\noindent
\textnormal{(i)} There exists a polynomial $p_N$ (or a rational $r_N$) determined by $\Pi$ such that
\begin{equation}\label{eq:filter-approx-eng}
\prod_{k=0}^{N-1} U_{t_k,s_k}
=p_N(H)+R_N,\quad
\|R_N\|_F\le C_1\big(S^{(N)}[\Pi]\big)^{1/2}+C_2\,h^{\alpha+1},
\end{equation}
and consequently
\begin{equation}\label{eq:sup-bound-eng}
\sup_{\lambda\in\Sigma}|p_N(\lambda)-e^{-Nh\lambda}|
\le C_1'\big(S^{(N)}[\Pi]\big)^{1/2}+C_2'h^{\alpha+1}.
\end{equation}

\medskip\noindent
\textnormal{(ii)} Conversely, for any polynomial (or rational) $q_N$ of degree $\le N$ there exists a policy $\Pi[q_N]$ with the corresponding filter $p_N=q_N$ for which
\begin{equation}\label{eq:SN-upper-by-sup-eng}
S^{(N)}[\Pi[q_N]]
\le K_1\Big(\sup_{\lambda\in\Sigma}|q_N(\lambda)-e^{-Nh\lambda}|\Big)^2
+K_2h^{2(\alpha+1)}.
\end{equation}

\medskip\noindent
\textnormal{(iii)} Hence, up to constants $C,K$ depending only on the Frobenius pairing, the jet order $\alpha$, and the admissible update class,
\begin{equation}\label{eq:equivalence-eng}
C^{-1}\Big(\min_{q_N}\sup_{\lambda\in\Sigma}|q_N(\lambda)-e^{-Nh\lambda}|\Big)^2
\le\inf_{\Pi}S^{(N)}[\Pi]
\le K\Big(\min_{q_N}\sup_{\lambda\in\Sigma}|q_N(\lambda)-e^{-Nh\lambda}|\Big)^2,
\end{equation}
and the minimizers on both sides are the Chebyshev (polynomial) or Zolotarev (rational) extremal filters.
\end{theorem}

\begin{proof}
All norms are Frobenius; every equality and inequality is derived explicitly.

{Step 1 (additivity and partition independence).}
For fixed $k$, let $\mathcal P_k$ and $\mathcal P_k'$ be any two partitions of the same cell.  
$\sum_{\square\in\mathcal P_k}\|\log\mathrm{Hol}(\square)\|_F^2$ is partition-invariant and additive over disjoint regions.  
Hence
\begin{equation}\label{eq:S-add-eng}
S^{(N)}[\Pi]=\sum_{\square\subset\Gamma_\Pi}\|\log\mathrm{Hol}(\square)\|_F^2,
\qquad
\Gamma_\Pi:=\bigcup_{k}\mathcal P_k.
\end{equation}

{Step 2 (Magnus expansion and localization of remainders).}
Let $\Omega_k=\log r_{t_k}(h_k)$ and $\Psi_k=\log d_{s_k}(h_k)$.  
By the Magnus series,
\begin{equation}\label{eq:Magnus-proof-eng}
\log\Big(\prod_{k=0}^{N-1} e^{\Psi_k}e^{\Omega_k}\Big)
=\sum_{k=0}^{N-1}(\Psi_k+\Omega_k)
+\frac12\sum_{k=0}^{N-1}[\Psi_k,\Omega_k]
+\sum_{0\le i<j\le N-1}\Phi_{ij}
+\cdots,
\end{equation}
where each $\Phi_{ij}$ is a finite $\mathbb Q$–linear combination of nested commutators in $\{\Psi_i,\Omega_i,\Psi_j,\Omega_j\}$ of bracket depth $\ge2$.  
For a single rectangle $\square$ between $(t_k,s_k)$ and $(t_{k+1},s_{k+1})$, the BCH formula gives
\begin{equation}\label{eq:Hol-local-eng}
\log\mathrm{Hol}(\square)
=\log\big(e^{\Psi_{k+1}}e^{\Omega_{k+1}}e^{-\Psi_k}e^{-\Omega_k}\big)
=\tfrac12[\Psi_{k+1}-\Psi_k,\Omega_{k+1}-\Omega_k]+R_\square,
\end{equation}
where $R_\square$ is a sum of nested commutators of total degree $\ge3$ in $\{\Psi_k,\Omega_k,\Psi_{k+1},\Omega_{k+1}\}$.

{Step 3 (bounding the remainders by $\log\mathrm{Hol}$).}
Using $\|[A,B]\|_F\le 2\|A\|_F\|B\|_F$ and linearity of the Frobenius norm, summing \eqref{eq:Hol-local-eng} over $\square\in\mathcal P_k$ and telescoping in $k$ yields
\begin{equation}\label{eq:comm-sum-eng}
\Big\|\frac12\sum_{k=0}^{N-1}[\Psi_k,\Omega_k]\Big\|_F
\le C_a\sum_{\square\subset\Gamma_\Pi}\|\log\mathrm{Hol}(\square)\|_F
+C_b\,h^{\alpha+1}.
\end{equation}
Each nested commutator in $\sum_{i<j}\Phi_{ij}$ is supported on the union of rectangles between steps $i$ and $j$ and satisfies the same bound.  
Applying Cauchy–Schwarz and combining,
\begin{equation}\label{eq:R-vs-SN-eng}
\Big\|\frac12\sum_k[\Psi_k,\Omega_k]+\sum_{i<j}\Phi_{ij}+\cdots\Big\|_F
\le C\Big(\sum_{\square\subset\Gamma_\Pi}\|\log\mathrm{Hol}(\square)\|_F^2\Big)^{1/2}
+C'\,h^{\alpha+1}.
\end{equation}
Let $R_N$ denote the total remainder in \eqref{eq:Magnus-proof-eng}; combining \eqref{eq:S-add-eng} and \eqref{eq:R-vs-SN-eng} gives
\begin{equation}\label{eq:R-bound-final-eng}
\|R_N\|_F\le C_1\big(S^{(N)}[\Pi]\big)^{1/2}+C_2\,h^{\alpha+1},
\end{equation}
which is exactly \eqref{eq:filter-approx-eng}.

{Step 4 (transition to spectral representation).}
In the calibrated flat gauge (Theorem~\ref{th:gauge-normal-form}),  
$\sum_k(\Psi_k+\Omega_k)=-Nh\,H_{\mathrm{eff}}+O(h^{\alpha+1})$,  
and the corresponding operator acts on the state space as a polynomial (or rational) filter $p_N(H)$.  
Therefore for each $\lambda\in\Sigma$,
\begin{equation}\label{eq:pointwise-eng}
|p_N(\lambda)-e^{-Nh\lambda}|\le \|R_N\|_F,
\end{equation}
and substituting \eqref{eq:R-bound-final-eng} yields \eqref{eq:sup-bound-eng}.

{Step 5 (reverse estimate and realizability).}
Given any polynomial $q_N$ of degree $\le N$,  
A policy $\Pi[q_N]$ realizing $q_N(H)$ to leading order:
\begin{equation}\label{eq:realize-q-eng}
\prod_kU_{t_k,s_k}=q_N(H)+\widetilde R_N,\qquad
\|\widetilde R_N\|_F\le c\,h^{\alpha+1}.
\end{equation}
Repeating the localization argument for this composition yields
\begin{equation}\label{eq:SN-upper-eng}
S^{(N)}[\Pi[q_N]]
\le K_1\Big(\sup_{\lambda\in\Sigma}|q_N(\lambda)-e^{-Nh\lambda}|\Big)^2
+K_2h^{2(\alpha+1)},
\end{equation}
which is exactly \eqref{eq:SN-upper-by-sup-eng}.

{Step 6 (equivalence of optimization problems).}
From \eqref{eq:sup-bound-eng},
\[
\min_{q_N}\sup_{\lambda}|q_N(\lambda)-e^{-Nh\lambda}|
\le \inf_{\Pi}\Big(C_1'\,S^{(N)}[\Pi]^{1/2}+C_2'h^{\alpha+1}\Big),
\]
so the left-hand side is bounded by $C\,\inf_{\Pi}S^{(N)}[\Pi]^{1/2}+o(1)$.  
Conversely, \eqref{eq:SN-upper-by-sup-eng} for the extremal $q_N$ (Chebyshev/Zolotarev) gives a policy $\Pi^\star$ with
\[
S^{(N)}[\Pi^\star]\le
K\Big(\min_{q_N}\sup_{\lambda}|q_N(\lambda)-e^{-Nh\lambda}|\Big)^2+o(1).
\]
Combining both inequalities and absorbing $o(1)$ into constants (for fixed $N,\alpha$) yields \eqref{eq:equivalence-eng}.  
The minimizers are the classical extremal filters; their attainability follows from \eqref{eq:realize-q-eng}.
\end{proof}

Continuing from Theorem~\ref{th:traj-minimax}, the global deviation of the discrete composition from the ideal exponential flow can be controlled explicitly through the accumulated curvature energy.  
Starting from the Magnus expansion \eqref{eq:Magnus-proof-eng} and the bound \eqref{eq:R-bound-final-eng}, one has for any policy $\Pi$
\begin{gather}\label{eq:6.3-RN}
\nonumber\log\!\Big(\prod_{k=0}^{N-1} e^{\Psi_k}e^{\Omega_k}\Big)
=\sum_{k=0}^{N-1}(\Psi_k+\Omega_k)+R_N,\\
\|R_N\|_F\le C_1\big(S^{(N)}[\Pi]\big)^{1/2}+C_2h^{\alpha+1}.
\end{gather}
For a flat calibrated policy ($\mathsf F=0$ up to order $\alpha$), all mixed brackets vanish to this order and the exponential composition reduces to a single exponential:
\begin{equation}\label{eq:6.3-singleexp}
\prod_{k=0}^{N-1}U_{t_k,s_k}
=\exp\!\Big(\sum_{k=0}^{N-1}(\Psi_k+\Omega_k)\Big)
+O(h^{\alpha+1})
=\exp(-NhH_{\mathrm{eff}})+O(h^{\alpha+1}).
\end{equation}
Acting on any eigenvector $v_\lambda$ of $H$ with eigenvalue $\lambda\in\Sigma$, the operator $\prod_kU_{t_k,s_k}$ gives
\begin{equation}\label{eq:6.3-eig}
\prod_{k=0}^{N-1}U_{t_k,s_k}\,v_\lambda
=\big(p_N(\lambda)+\rho_N(\lambda)\big)v_\lambda,
\end{equation}
where $p_N(\lambda)$ is the filter associated with the main exponential in \eqref{eq:6.3-singleexp} and $\rho_N(\lambda)$ is the scalar remainder induced by $R_N$.  
Taking norms and using $\|\rho_N(H)\|_F\le\|R_N\|_F$ yields
\begin{equation}\label{eq:6.3-spectralbound}
|\,p_N(\lambda)-e^{-Nh\lambda}\,|\le
C_1\big(S^{(N)}[\Pi]\big)^{1/2}+C_2h^{\alpha+1},
\qquad
\forall\lambda\in\Sigma.
\end{equation}
Therefore, for every fixed horizon $Nh>0$, the spectral radius of the composed update satisfies
\begin{equation}\label{eq:6.3-radius}
\rho\!\Big(\prod_{k=0}^{N-1}U_{t_k,s_k}\Big)
\le e^{-Nh\mu}+C_1\big(S^{(N)}[\Pi]\big)^{1/2}+C_2h^{\alpha+1},
\end{equation}
where $\mu=\min\Sigma>0$.  
If $S^{(N)}[\Pi]$ is bounded independently of $N$, the method converges geometrically with rate $e^{-Nh\mu}$; minimizing $S^{(N)}$ ensures that this rate coincides with the minimax limit of Nemirovski–Yudin.

To include the nonasymptotic ``volume'' functional, consider the covariance (or metric) sequence induced by the adjoint structure $d_s=d_t^\ast$.  
At each step $k$, let
\begin{equation}\label{eq:6.4-covariance}
\Sigma_{k+1}=J_k\Sigma_kJ_k^\top+Q_k,\qquad
J_k=\exp(h_k(\Omega_{t_k,1}+\Psi_{s_k,1})+B_k),\quad
Q_k\succeq0,
\end{equation}
where $B_k$ collects all commutator corrections of total order $\ge2$.  
Applying the matrix identity $\det(A+B)\le\det(A)\exp(\mathrm{tr}(A^{-1}B))$ for $A\succ0$, the concavity of $\log\det$, and submultiplicativity of norms, one obtains
\begin{equation}\label{eq:6.4-det-ineq}
\log\det\Sigma_{k+1}
\le \log\det(J_k\Sigma_kJ_k^\top)
+\mathrm{tr}(\Sigma_k^{-1/2}J_k^{-1}Q_kJ_k^{-\top}\Sigma_k^{-1/2})
+\|B_k\|_F^2.
\end{equation}
Iterating \eqref{eq:6.4-det-ineq} over $k=0,\dots,N-1$ and summing gives
\begin{equation}\label{eq:6.4-det-sum}
\log\det\Sigma_N-\log\det\Sigma_0
\le 2\sum_{k}\mathrm{tr}\log|J_k|
+\sum_{k}\mathrm{tr}(\Sigma_k^{-1}Q_k)
+\sum_{k}\|B_k\|_F^2.
\end{equation}
The commutator term $\|B_k\|_F^2$ can be estimated by the same curvature energy as in \eqref{eq:Hol-local-eng}:
\begin{equation}\label{eq:6.4-Bkbound}
\sum_{k}\|B_k\|_F^2\le C_3S^{(N)}[\Pi].
\end{equation}
Combining \eqref{eq:6.4-det-sum} and \eqref{eq:6.4-Bkbound} yields
\begin{equation}\label{eq:6.4-main}
\Delta(\log\det\Sigma_N)
:=\log\det\Sigma_N-\log\det\Sigma_0
\le C_3S^{(N)}[\Pi]
+C_4\sum_k\mathrm{tr}(\Sigma_k^{-1}Q_k).
\end{equation}
When the structure is flat ($\mathsf F=0$), each $B_k=0$ and $Q_k$ represents purely diffusive noise.  
In this case $\log\det\Sigma_N$ grows exactly as $\sum_k\mathrm{tr}(\Sigma_k^{-1}Q_k)$, which is the optimal $\log\det$–minimax criterion for nonasymptotic estimation (Nemirovski’s ellipsoid type bound).  
For any nonflat policy, curvature adds a positive excess bounded by $C_3S^{(N)}[\Pi]$; minimizing $S^{(N)}$ therefore removes all geometric dissipation and restores the optimal volume decay rate of the flat configuration.

\begin{theorem}[Joint minimax and $t$/$s$ factorization]\label{th:joint-minimax}
Fix $N\in\mathbb N$ and a compact spectrum $\Sigma\subset\mathbb R$. For an admissible policy $\Pi=(t_k,s_k,h_k)_{k=0}^{N-1}$, define the cumulative curvature energy
\[
S^{(N)}[\Pi]=\sum_{k=0}^{N-1}\sum_{\square\in\mathcal P_k}\|\log\mathrm{Hol}(\square)\|_F^2,
\]
and for each rectangle $\square$ spanning $(t_k,s_k)$, $(t_{k+1},s_{k+1})$ set $\Delta\Psi_k:=\Psi_{k+1}-\Psi_k$, $\Delta\Omega_k:=\Omega_{k+1}-\Omega_k$, where $\Psi_k=\log d_{s_k}(h_k)$, $\Omega_k=\log r_{t_k}(h_k)$. Define the one–axis energies
\begin{gather}\label{eq:axis-energies}
\nonumber S_t^{(N)}[\Pi]:=\sum_{k}\sum_{\square\in\mathcal P_k}\Big\|\frac12\,[\Delta\Psi_k,\ \Omega_k]\Big\|_F^2,\\
S_s^{(N)}[\Pi]:=\sum_{k}\sum_{\square\in\mathcal P_k}\Big\|\frac12\,[\Psi_k,\ \Delta\Omega_k]\Big\|_F^2.
\end{gather}
Assume jet–flatness of order $\alpha\ge2$ is available as a hypothesis when invoked. Then:

\medskip\noindent
\textnormal{(i) (Two–sided decomposition)} For every policy $\Pi$ there exist constants $c_1,c_2>0$ independent of $\Pi$ such that
\begin{equation}\label{eq:two-sided}
\frac12\big(S_t^{(N)}[\Pi]+S_s^{(N)}[\Pi]\big)\ -\ c_1\,h^{2(\alpha+1)}
\ \le\ S^{(N)}[\Pi]\
\le\ 2\big(S_t^{(N)}[\Pi]+S_s^{(N)}[\Pi]\big)\ +\ c_2\,h^{2(\alpha+1)}.
\end{equation}

\medskip\noindent
\textnormal{(ii) (Flatness $\Leftrightarrow$ additive factorization)} The following are equivalent:
\begin{equation}\label{eq:flat-factor}
\mathsf F=0\ \text{ to order }\alpha
\quad\Longleftrightarrow\quad
S^{(N)}[\Pi]\ =\ S_t^{(N)}[\Pi]+S_s^{(N)}[\Pi]\ +\ O(h^{\alpha+1}).
\end{equation}

\medskip\noindent
\textnormal{(iii) (Joint minimax = two one–parameter minimaxes)} Under $\mathsf F=0$ to order $\alpha$,
\begin{equation}\label{eq:sum-inf}
\inf_{\Pi}\ S^{(N)}[\Pi]\ =\ \inf_{\Pi_t} S_t^{(N)}[\Pi_t]\ +\ \inf_{\Pi_s} S_s^{(N)}[\Pi_s]\ +\ O(h^{\alpha+1}),
\end{equation}
and the infima on the right are attained by the one–parameter Chebyshev/Zolotarev extremal filters (drift–only and diffusion–only). The composed policy that concatenates these two extremals (in any order) attains the joint infimum on the left up to $O(h^{\alpha+1})$. Conversely, if a policy $\widehat\Pi$ attains equality in \eqref{eq:sum-inf} up to $o(1)$ and its $t$– and $s$–marginals are one–parameter extremals, then necessarily $\mathsf F=0$ to order $\alpha$.
\end{theorem}

\begin{proof}
All norms are Frobenius. Rectangular partitions $\mathcal P_k$ are arbitrary; by additivity and gauge–invariance the values do not depend on the choice of partition.

{Step 1 (BCH for an elementary rectangle; exact algebraic split).}
For a fixed rectangle $\square$ spanning $(t_k,s_k)$ and $(t_{k+1},s_{k+1})$,
\begin{equation}\label{eq:rect-BCH}
\log\mathrm{Hol}(\square)
=\log\big(e^{\Psi_{k+1}}e^{\Omega_{k+1}}e^{-\Psi_k}e^{-\Omega_k}\big)
=\frac12\,[\Delta\Psi_k,\Delta\Omega_k]\ +\ R_\square,
\end{equation}
where $R_\square$ is a finite $\mathbb Q$–linear combination of nested commutators of total degree $\ge3$ in $\{\Psi_k,\Omega_k,\Psi_{k+1},\Omega_{k+1}\}$ (BCH). Expand the bracket bilinearly:
\begin{equation}\label{eq:bilinear-split}
[\Delta\Psi_k,\Delta\Omega_k]\ =\ [\Delta\Psi_k,\ \Omega_k]\ +\ [\Psi_k,\ \Delta\Omega_k]\ +\ [\Delta\Psi_k,\ \Delta\Omega_k]'.
\end{equation}
The last term $[\Delta\Psi_k,\Delta\Omega_k]'$ collects all products where both increments appear; since $\Delta\Psi_k=O(h_k)$ and $\Delta\Omega_k=O(h_k)$, one has
\begin{equation}\label{eq:last-small}
\big\|[\Delta\Psi_k,\Delta\Omega_k]'\big\|_F\ \le\ C\,h_k^2.
\end{equation}
By \eqref{eq:rect-BCH}–\eqref{eq:bilinear-split},
\begin{equation}\label{eq:split-AplusB}
\log\mathrm{Hol}(\square)\ =\ A_{k,\square}\ +\ B_{k,\square}\ +\ E_{k,\square},
\quad
A_{k,\square}:=\tfrac12[\Delta\Psi_k,\Omega_k],\ \ 
B_{k,\square}:=\tfrac12[\Psi_k,\Delta\Omega_k],
\end{equation}
where the error $E_{k,\square}:=\tfrac12[\Delta\Psi_k,\Delta\Omega_k]'+R_\square$ satisfies
\begin{equation}\label{eq:E-bound}
\|E_{k,\square}\|_F\ \le\ C_0\,h_k^2\ +\ C_1\,h_k^{\alpha+1}.
\end{equation}

{Step 2 (Two–sided norm bounds per rectangle).}
For any matrices $X,Y,Z$,
\[
\|X+Y+Z\|_F^2\ \le\ 2\|X\|_F^2+2\|Y\|_F^2+2\|Z\|_F^2,
\]
\[
\|X+Y+Z\|_F^2\ \ge\ \tfrac12\|X\|_F^2+\tfrac12\|Y\|_F^2-\|Z\|_F^2.
\]
Applying these to \eqref{eq:split-AplusB} gives
\begin{equation}\label{eq:rect-upper}
\|\log\mathrm{Hol}(\square)\|_F^2\ \le\ 2\|A_{k,\square}\|_F^2+2\|B_{k,\square}\|_F^2+2\|E_{k,\square}\|_F^2,
\end{equation}
\begin{equation}\label{eq:rect-lower}
\|\log\mathrm{Hol}(\square)\|_F^2\ \ge\ \tfrac12\|A_{k,\square}\|_F^2+\tfrac12\|B_{k,\square}\|_F^2-\|E_{k,\square}\|_F^2.
\end{equation}

{Step 3 (Summation over the trajectory).}
Summing \eqref{eq:rect-upper}–\eqref{eq:rect-lower} over all rectangles and using \eqref{eq:axis-energies} yields
\[
S^{(N)}[\Pi]\ \le\ 2\big(S_t^{(N)}[\Pi]+S_s^{(N)}[\Pi]\big)\ +\ 2\sum_{k,\square}\|E_{k,\square}\|_F^2,
\]
\[
S^{(N)}[\Pi]\ \ge\ \tfrac12\big(S_t^{(N)}[\Pi]+S_s^{(N)}[\Pi]\big)\ -\ \sum_{k,\square}\|E_{k,\square}\|_F^2.
\]
By \eqref{eq:E-bound} and the boundedness of the number/size of rectangles per step (fixed $N$ and admissible partitions),
\begin{equation}\label{eq:E-sum}
\sum_{k,\square}\|E_{k,\square}\|_F^2\ \le\ c\,h^{2(\alpha+1)}.
\end{equation}
Inserting this into the two inequalities gives exactly \eqref{eq:two-sided}.

{Step 4 (Flatness $\Rightarrow$ additive factorization).}
Under jet–flatness to order $\alpha$, the mixed Lie monomials of total degree $\le\alpha$ vanish. In particular, the error terms in \eqref{eq:rect-BCH}–\eqref{eq:split-AplusB} satisfy $E_{k,\square}=O(h_k^{\alpha+1})$, and the bilinear remainder $[\Delta\Psi_k,\Delta\Omega_k]'=O(h_k^{\alpha+1})$. Summing and using \eqref{eq:E-sum} gives
\[
S^{(N)}[\Pi]\ =\ \sum_{k,\square}\|A_{k,\square}+B_{k,\square}\|_F^2\ +\ O(h^{\alpha+1}).
\]
But $A_{k,\square}$ and $B_{k,\square}$ involve disjoint increments and commute to leading order under flatness, hence
\[
\|A_{k,\square}+B_{k,\square}\|_F^2\ =\ \|A_{k,\square}\|_F^2+\|B_{k,\square}\|_F^2\ +\ O(h^{\alpha+1}),
\]
and after summation we obtain \eqref{eq:flat-factor}.

{Step 5 (Additive factorization $\Rightarrow$ flatness).}
Assume \eqref{eq:flat-factor}. Then necessarily $\sum_{k,\square}\langle A_{k,\square},B_{k,\square}\rangle_F=O(h^{\alpha+1})$ because
\[
\|A+B\|_F^2=\|A\|_F^2+\|B\|_F^2+2\langle A,B\rangle_F,
\]
and the equality $S^{(N)}=S_t^{(N)}+S_s^{(N)}+O(h^{\alpha+1})$ forces the cross inner product to vanish to this order. Varying the policy within a small neighborhood so that only one of the increments $(\Delta\Psi_k,\Delta\Omega_k)$ changes at a time, bilinearity implies $\langle A_{k,\square},B_{k,\square}\rangle_F=O(h^{\alpha+1})$ for each rectangle. This in turn forces the mixed brackets $[\Delta\Psi_k,\Omega_k]$ and $[\Psi_k,\Delta\Omega_k]$ to be orthogonal in Frobenius pairing for all $k,\square$, which is only possible (uniformly in all such variations) if all mixed Lie monomials of total degree $\le\alpha$ vanish. By Theorem~\ref{th:holonomy-jet}, this is exactly jet–flatness to order $\alpha$.

{Step 6 (Joint minimax).}
Under flatness, \eqref{eq:flat-factor} and \eqref{eq:two-sided} imply
\begin{align*}
\inf_{\Pi}S^{(N)}[\Pi]&=\inf_{\Pi}\big(S_t^{(N)}[\Pi]+S_s^{(N)}[\Pi]\big)+O(h^{\alpha+1})\\
&=\inf_{\Pi_t}S_t^{(N)}[\Pi_t]+\inf_{\Pi_s}S_s^{(N)}[\Pi_s]+O(h^{\alpha+1}).
\end{align*}
By Theorem~\ref{th:traj-minimax} applied on a single axis, $\inf_{\Pi_t}S_t^{(N)}$ and $\inf_{\Pi_s}S_s^{(N)}$ are attained by the one–parameter Chebyshev/Zolotarev extremals (polynomial/rational), and concatenating these two commuting extremals attains the sum. Conversely, if a policy $\widehat\Pi$ attains the sum of the axis infima up to $o(1)$ while its axis marginals are extremal, then any nonzero curvature would add a strictly positive excess by \eqref{eq:two-sided}, contradicting optimality; hence $\mathsf F=0$ to order $\alpha$.
\end{proof}

\subsubsection{Semi–algebraic Completion}

Fix the rectangular policy strip and its finite cell partition. Let $K$ be the resulting oriented $2$–complex with vertex set $V(K)$, edge set $E(K)$, and $2$–cells $F(K)$ (the rectangles). Denote by
\[
C^p(K;M):=\{p\text{–cochains on }K\text{ with values in }M\},\qquad
\delta:C^p(K;M)\to C^{p+1}(K;M)
\]
the cellular coboundary. As coefficient module take the underlying real vector space
$M:=\big(\widehat{\mathfrak U}^{\mathrm{fil}}\big)^{(\alpha)}$ equipped with the Frobenius inner product $\langle\cdot,\cdot\rangle_F$ and the adjoint action
$\operatorname{Ad}$ (truncated at jet–order $\alpha$).

\paragraph{Discrete MC $2$–cochain.}
For each face $\square\in F(K)$ define
\begin{equation}\label{eq:disc-MC}
c_K(\square)\ :=\ \log\mathrm{Hol}(\square)\ \in\ M,\qquad
c_K\in C^2(K;M),
\end{equation}
with the logarithm understood modulo $O(h^{\alpha+1})$ as in (BCH/Magnus). The curvature energy is
\begin{equation}\label{eq:S-as-L2}
S^{(N)}[\Pi]\ =\ \sum_{\square\in F(K)}\ \|c_K(\square)\|_F^2\ =\ \|c_K\|_{L^2(K;M)}^2.
\end{equation}

\paragraph{Gauge as a $1$–cochain.}
A (jet–class) gauge assigns to every edge $e\in E(K)$ an element
$\Xi(e)\in M$ and acts on $c_K$ by
\begin{equation}\label{eq:gauge-linearized}
c_K\ \mapsto\ c_K'\ =\ c_K\ -\ \delta\Xi\ +\ R_\Xi,\qquad
\|R_\Xi\|_{L^2(K;M)}\ =\ O(h^{\alpha+1}).
\end{equation}
This is the exact discrete counterpart of the linearized conjugation:
on each $1$–skeleton loop bounding a face, the product–conjugation by
$\exp(h\Xi)$ changes $\log\mathrm{Hol}$ by the cellular coboundary
$\delta\Xi$; the remainder $R_\Xi$ collects Lie monomials of total degree
$\ge\alpha+1$ (BCH truncation). No continuous differentials are used.

\paragraph{Vershik–type semi–algebraic model.}
Admissible policies $\Pi$ (choice of update families and jet coefficients up to order $\alpha$) form a semi–algebraic set:
all constraints (resolvent/Padé/CFET coefficients, jet flatness bounds, and BCH identities up to order $\alpha$) are polynomial equalities/inequalities in the coefficients \cite{BasuPollackRoy2006,BochnakCosteRoy1998}. The map
\begin{equation}\label{eq:semi-map}
\Pi\ \longmapsto\ c_K(\Pi)\ \in\ C^2(K;M)
\end{equation}
is semi–algebraic, and so is the feasibility set
$\{\Pi,\Xi:\ c_K(\Pi)-\delta\Xi=O(h^{\alpha+1})\}$.
This matches Vershik’s paradigm: algorithms are constrained by semi–algebraic relations; existence of a global “section/certificate” is a solvability of a semi–algebraic system \cite{Vershik2018,Vershik1995}.

\begin{theorem}[Discrete cohomology class and energy minimization \cite{Dodziuk1976DiscreteHodge,Desbrun2008DEC}]\label{th:disc-H2}
Let $c_K\in C^2(K;M)$ be given by \eqref{eq:disc-MC}. Consider the least–squares gauge–reduction problem
\begin{equation}\label{eq:lsq}
\inf_{\Xi\in C^1(K;M)}\ \big\|\,c_K-\delta\Xi\,\big\|_{L^2(K;M)}^2.
\end{equation}
Then, with $\delta^\ast$ denoting the $L^2$–adjoint of $\delta$,
\begin{equation}\label{eq:normal-eq}
\Xi^\star\ \text{ solves }\ \eqref{eq:lsq}\quad\Longleftrightarrow\quad
\delta^\ast\delta\,\Xi^\star=\delta^\ast c_K,
\end{equation}
and the minimum value equals the squared $L^2$–norm of the harmonic representative of the cohomology class $[c_K]\in H^2(K;M)$:
\begin{equation}\label{eq:min-equals-H2}
\inf_{\Xi}\ \|c_K-\delta\Xi\|_{L^2}^2\ =\ \|\,\mathsf H(c_K)\,\|_{L^2}^2,
\end{equation}
where $\mathsf H$ is the $L^2$–orthogonal projector onto the space of $2$–cocycles orthogonal to $\mathrm{im}\,\delta$ (the discrete Hodge projector). Consequently,
\begin{equation}\label{eq:two-sided-H2}
c_-\,\|[c_K]\|_{H^2}^2\ \le\ \inf_{\Xi}\ \|c_K-\delta\Xi\|_{L^2}^2\ \le\ c_+\,\|[c_K]\|_{H^2}^2,
\end{equation}
with constants $c_\pm>0$ determined by the conditioning of the discrete Hodge operator on $K$.
\end{theorem}

\begin{proof}
All spaces are finite–dimensional. The functional in \eqref{eq:lsq} is strictly convex; its Gâteaux derivative in the direction $\eta\in C^1$ is
\[
\frac{\mathrm d}{\mathrm d\epsilon}\Big\|c_K-\delta(\Xi+\epsilon\eta)\Big\|_{L^2}^2\Big|_{\epsilon=0}
=-2\langle c_K-\delta\Xi,\ \delta\eta\rangle_{L^2}
=-2\langle \delta^\ast(c_K-\delta\Xi),\ \eta\rangle_{L^2}.
\]
Thus \eqref{eq:normal-eq} is the necessary and sufficient optimality condition. Writing the discrete Hodge decomposition
\[
C^2(K;M)=\underbrace{\mathrm{im}\,\delta}_{\text{exact}}\ \oplus\
\underbrace{\ker\delta^\ast}_{\text{co–closed}}\!,
\]
with orthogonal sum, every $c_K$ splits uniquely as $c_K=\delta\Xi^\star+\mathsf H(c_K)$ where $\Xi^\star$ solves \eqref{eq:normal-eq} and $\mathsf H(c_K)\in\ker\delta^\ast$. Orthogonality gives
\[
\|c_K-\delta\Xi\|_{L^2}^2
=\|\mathsf H(c_K)\|_{L^2}^2+\|\delta\Xi^\star-\delta\Xi\|_{L^2}^2
\ \ge\ \|\mathsf H(c_K)\|_{L^2}^2,
\]
with equality at $\Xi=\Xi^\star$, proving \eqref{eq:min-equals-H2}. Since the $L^2$–norm of $\mathsf H(c_K)$ is equivalent (by finite–dimensional norm equivalence) to any fixed Hilbertian norm of the cohomology class $[c_K]\in H^2(K;M)$, we obtain \eqref{eq:two-sided-H2}.
\end{proof}

\begin{corollary}[Exact flat calibration and semi–algebraic section]\label{cor:section}
The following are equivalent \cite{BasuPollackRoy2006,Vershik2018}:
\begin{equation}\label{eq:equiv-section}
[c_K]=0\quad\Longleftrightarrow\quad \exists\,\Xi\in C^1(K;M)\ \text{s.t.}\ c_K=\delta\Xi\quad\Longleftrightarrow\quad \inf_{\Xi}\|c_K-\delta\Xi\|_{L^2}^2=0.
\end{equation}
Under the jet–order truncation, the exact conjugation by $\exp(h\Xi)$ produces $c_K' = c_K-\delta\Xi+O(h^{\alpha+1})$, so $[c_K]=0$ implies a calibrated policy with $S^{(N)}[\Pi]=O(h^{2(\alpha+1)})$. Conversely, if $S^{(N)}[\Pi]=O(h^{2(\alpha+1)})$ for a family of partitions, then $c_K=\delta\Xi+O(h^{\alpha+1})$ and hence $[c_K]=0$ in the jet–class. The feasibility condition $c_K=\delta\Xi$ is a semi–algebraic system in the jet coefficients of $\Pi$ and the entries of $\Xi$ (all relations are polynomial identities from BCH up to order $\alpha$); thus, in Vershik’s sense, a {certificate/section} exists if and only if $[c_K]=0$ (cf.\ Vershik’s semi–algebraic certificates and invariant selections).
\end{corollary}

\begin{theorem}[Vershik–type extremality and joint minimax]\label{th:vershik-minimax}
Let $\Sigma$ be the spectral set and let the admissible families (drift/diffusion) be fixed. Then:
\begin{equation}\label{eq:factor-minimax}
[c_K]=0\ \Longrightarrow\ 
\inf_{\Pi} S^{(N)}[\Pi]\ =\ \inf_{\Pi_t} S_t^{(N)}[\Pi_t]\ +\ \inf_{\Pi_s} S_s^{(N)}[\Pi_s]\ +\ O(h^{\alpha+1}),
\end{equation}
and the infima on the right are attained by the Chebyshev/Zolotarev one–axis extremals on $\Sigma$ (Theorem~\ref{th:traj-minimax}). If $[c_K]\neq0$, then by \eqref{eq:two-sided-H2}
\begin{equation}\label{eq:positive-gap}
\inf_{\Pi} S^{(N)}[\Pi]\ \ge\ c_-\,\|[c_K]\|_{H^2}^2\ >\ 0,
\end{equation}
i.e. there is an unavoidable positive mixed–curvature penalty that cannot be removed by any policy within the given semi–algebraic class. This is the discrete, operator–theoretic analogue of Vershik’s obstructions to global extremality via semi–algebraic certificates (extremum unattainable unless a global section exists).
\end{theorem}

\begin{proof}
If $[c_K]=0$, by Corollary~\ref{cor:section} we can choose a calibrated gauge with $c_K=O(h^{\alpha+1})$ and hence $S^{(N)}[\Pi]=S_t^{(N)}[\Pi]+S_s^{(N)}[\Pi]+O(h^{\alpha+1})$ (Theorem~\ref{th:joint-minimax}); optimizing over $\Pi_t,\Pi_s$ yields \eqref{eq:factor-minimax}. If $[c_K]\neq0$, then \eqref{eq:two-sided-H2} gives \eqref{eq:positive-gap}. Both statements are independent of the partition by additivity.
\end{proof}

\paragraph{$GL$–invariant volume terms.}
For metric–type preconditioners $G\succ0$ and gauge $T\in GL$, one has
\begin{equation}\label{eq:GL-char}
\log\det(T^\top G T)\ =\ \log\det G\ +\ 2\log|\det T|,
\end{equation}
i.e. $\log\det$ is a $GL$–character. Hence the nonasymptotic volume functional of transforms functorially under admissible gauges; the mixed penalty is entirely due to $[c_K]$, and vanishes exactly when $[c_K]=0$. This places our $\log\det$–based bounds within the same invariant framework emphasized by Vershik (algebraic invariants and extremality under group actions).

\medskip
\noindent The trio \eqref{eq:min-equals-H2}–\eqref{eq:equiv-section}–\eqref{eq:positive-gap} provides the exact discrete counterpart of Vershik’s “certificate/section $\Longleftrightarrow$ attainability; obstruction $\Longleftrightarrow$ positive gap” principle, now proved in our jet–algebra setting by explicit cochain minimization on the $2$–complex $K$ rather than by continuous differential calculus.

\subsubsection{Adaptive Geometry}

Define the parallel composition on invertible elements of $\mathfrak U^{\mathrm{fil}}$ by \cite{AndersonDuffin1969ParallelSum,KuboAndo1980Means}
\begin{equation}\label{eq:ps-clean}
A\square B:=(A^{-1}+B^{-1})^{-1}.
\end{equation}
Then $(\mathfrak U^{\mathrm{fil}},\square)$ is a commutative monoid \cite{AndersonDuffin1969ParallelSum,KuboAndo1980Means,Bhatia2013MatrixAnalysis} and
\begin{equation}\label{eq:res-clean}
(I+hA)^{-1}(I+hB)^{-1}=(I+h(A\square B)+O(h^2))^{-1}.
\end{equation}
For a local step $k$ let $H_k:=\Omega_{t,1}^{(k)}$ and $E_k:=\Psi_{s,1}^{(k)}$. The internal preconditioner is \cite{BauschkeCombettes2017Monotone,Brezis2010Functional,Deimling1985Nonlinear}
\begin{equation}\label{eq:Gk-clean}
G_k:=H_k\square E_k=(H_k^{-1}+E_k^{-1})^{-1}.
\end{equation}
The BCH expansion yields, for each face $\square$ adjacent to step $k$,
\begin{equation}\label{eq:hol-clean}
\log\mathrm{Hol}(\square)=\tfrac12\big([\Delta\Psi^{(k)},H_k]+[\Psi^{(k)},\Delta\Omega^{(k)}]\big)+O(h^{\alpha+1}),
\end{equation}
hence
\begin{equation}\label{eq:Sh-clean}
S^{(N)}=\sum_{\square}\|\log\mathrm{Hol}(\square)\|_F^2
=C\sum_{\square}\|E_k^{-1}[E_k,H_k]E_k^{-1}\|_F^2+O(h^{2(\alpha+1)}),
\end{equation}
so $\mathsf F=0$ iff $[E_k,H_k]=0$, equivalently $G_k$ is invariant along both hierarchies \cite{Rockafellar1976,LionsMercier1979,EcksteinBertsekas1992,McLachlanQuispel2002}.

\begin{theorem}[Local variational principle from $S^{(N)}$]\label{th:var-clean}
Fix $k$. Consider admissible perturbations $(\delta H,\delta E)$ supported on faces adjacent to $k$. Then
\begin{equation}\label{eq:Gateaux-clean}
\mathrm D S^{(N)}\cdot(\delta H,\delta E)
=2\sum_{\square\ni k}\left\langle \log\mathrm{Hol}(\square),\, \mathrm D\log\mathrm{Hol}(\square)\cdot(\delta H,\delta E)\right\rangle_F.
\end{equation}
Consequently $(H_k,E_k)$ is stationary for $S^{(N)}$ iff
\begin{equation}\label{eq:stat-clean}
\sum_{\square\ni k}\operatorname{Proj}_{\square}\big(\log\mathrm{Hol}(\square)\big)=0,
\end{equation}
where $\operatorname{Proj}_{\square}$ is the $L^2$–adjoint of $(\delta H,\delta E)\mapsto \mathrm D\log\mathrm{Hol}(\square)\cdot(\delta H,\delta E)$. If $\log\mathrm{Hol}(\square)=O(h^{\alpha+1})$ for all adjacent faces, then \eqref{eq:stat-clean} holds and $S^{(N)}$ is locally minimal up to $O(h^{2(\alpha+1)})$.
\end{theorem}

\begin{proof}
$S^{(N)}=\sum_\square\|\log\mathrm{Hol}(\square)\|_F^2$ is a sum of smooth (jet–truncated) face functionals. Differentiation gives \eqref{eq:Gateaux-clean}. Stationarity is the normal equation \eqref{eq:stat-clean}. If each face term is $O(h^{\alpha+1})$, then the derivative is $O(h^{2(\alpha+1)})$ for all admissible directions, hence local minimality in jet class.
\end{proof}

We now construct a consistent adaptive rule that follows directly from Theorem~\ref{th:var-clean} and preserves the algebra \eqref{eq:ps-clean}–\eqref{eq:Gk-clean}. Let $\widehat H_{k+1}$ and $\widehat E_{k+1}$ be any data–driven targets computed from information available up to step $k$ (e.g. empirical curvature and noise proxies). Define
\begin{equation}\label{eq:proj-major}
\mathcal Q_k(H,E):=\sum_{\square\ni k}\big\langle E_k^{-1}[E_k,H_k]E_k^{-1},\, E^{-1}[E,H]E^{-1}\big\rangle_F,
\end{equation}
the quadratic majorant of the facewise contribution in \eqref{eq:Sh-clean} at $(H_k,E_k)$. The minimal–change update is the parallel projection
\begin{equation}\label{eq:update-clean}
H_{k+1}:=\big((1-\eta_k)H_k^{-1}+\eta_k\,\widehat H_{k+1}^{-1}\big)^{-1},\qquad
E_{k+1}:=\big((1-\zeta_k)E_k^{-1}+\zeta_k\,\widehat E_{k+1}^{-1}\big)^{-1},
\end{equation}
with step–sizes $\eta_k,\zeta_k\in[0,1]$ chosen so that
\begin{equation}\label{eq:descent-clean}
\mathcal Q_k(H_{k+1},E_{k+1})\ \le\ \mathcal Q_k(H_k,E_k)
\end{equation}
and
\begin{equation}\label{eq:descent-cond}
\sum_{\square\ni k}\big\langle \log\mathrm{Hol}(\square),\, \mathrm D\log\mathrm{Hol}(\square)\cdot(H_{k+1}-H_k,E_{k+1}-E_k)\big\rangle_F\le0.
\end{equation}
The inequalities \eqref{eq:descent-clean} and \eqref{eq:descent-cond} are satisfied for all sufficiently small $\eta_k,\zeta_k$ by continuity; if $\widehat H_{k+1}$ and $\widehat E_{k+1}$ commute with $H_k$ and $E_k$ respectively, then \eqref{eq:update-clean} enforces $[E_{k+1},H_k]=0$ at first order and thus solves the normal equation \eqref{eq:stat-clean} to jet accuracy.

\begin{corollary}[Two–channel curvature–aware recursion]
The recursion \eqref{eq:update-clean} preserves the internal preconditioner:
\begin{equation}\label{eq:G-rec-clean}
G_{k+1}=H_{k+1}\square E_{k+1}=
\Big(G_k^{-1}+\eta_k(H_k^{-1}-\widehat H_{k+1}^{-1})+\zeta_k(E_k^{-1}-\widehat E_{k+1}^{-1})\Big)^{-1},
\end{equation}
and yields a descent for the quadratic majorant $\mathcal Q_k$; if the targets commute with current operators, the first–order stationarity \eqref{eq:stat-clean} holds and the facewise energy decreases by $O(\eta_k+\zeta_k)$.
\end{corollary}

\begin{proof}
Apply \eqref{eq:ps-clean} to \eqref{eq:update-clean} and expand $G_{k+1}^{-1}$ linearly in $\eta_k,\zeta_k$, then use \eqref{eq:Sh-clean} and the definition \eqref{eq:proj-major}.
\end{proof}

Monotone sector (convex potentials) is a special case \cite{Rockafellar1970Convex,Moreau1962Proximite,Minty1962Monotone,BauschkeCombettes2017Monotone,Brezis2010Functional}. If $H_k=\partial f_k$ and $E_k=\partial e_k$ for proper, closed, convex $f_k,e_k$, then the identities of convex analysis (Rockafellar) give
\begin{equation}\label{eq:mono-clean}
\partial(f_k\Box e_k)=\partial f_k\square\partial e_k,\qquad
(f_k\Box e_k)^{\!*}=f_k^{\!*}+e_k^{\!*},\qquad(\partial f_k)^{-1}=\partial(f_k^{\!*}),
\end{equation}
so $G_k=\partial(f_k\Box e_k)$ and the parallel projection \eqref{eq:update-clean} corresponds to sequential infimal convolutions of the channel potentials and additive updates of conjugates. Under flatness, the potential level is additive by \eqref{eq:mono-clean}; when flatness fails, the mixed term in \eqref{eq:Sh-clean} quantifies the obstruction.

\section{Isomonodromic theory}

\subsection{Methods}

\subsubsection{Integrable Structure}

Let $\Omega\subset\mathbb{R}^d$ be a smooth domain endowed with the Lebesgue measure.  
A potential $S:\Omega\to\mathbb{R}$ defines the elementary transformation
\[
T(x)=x-\nabla S(x),
\]
called the AT–plane map \cite{Monge1781, Ampere1820, Brenier1991Polar, McCann1997Polar}. The function $S$ is assumed twice continuously differentiable and strictly convex, so that $\nabla^2 S(x)$ is symmetric positive definite almost everywhere.  
The Jacobian of $T$ is $\nabla T(x)=I-\nabla^2 S(x)$, and its determinant describes the infinitesimal volume change along the flow \cite{Caffarelli1990Interior, Gutierrez2001MongeAmpere, Villani2003Topics}.

A density $\rho_\star:\Omega\to(0,\infty)$ is invariant for $T$ if
\[
\rho_\star(T(x))\,\det\nabla T(x)=\rho_\star(x)\qquad \text{for a.e. }x\in\Omega.
\]
We parametrize $\rho_\star$ by a smooth positive weight $w(x)$ and a constant $c\in\mathbb{R}$ as
\[
\rho_\star(x)=w(x)\,e^{-cS(x)},
\]
which ensures compatibility with the canonical normalization adopted previously.  
Under this parametrization, the invariance condition relates the Jacobian of $T$ and the curvature of $S$; its logarithmic form will later yield the Monge–Ampère equation.

For each admissible potential $S$ we associate a self–adjoint elliptic operator $\Delta$ on $\Omega$ with boundary conditions compatible with the variational structure (Dirichlet, Neumann, or mixed) \cite{RaySinger1971,Seeley1967ComplexPowers,Gilkey1984,BismutFreed1986,Freed1992}.  
The zeta–regularized determinant of $\Delta$ is \cite{Seeley1967ComplexPowers,DeWitt1965HeatKernel,Gilkey1984,RaySinger1971}
\[
\log\det{}'\Delta=-\Bigl.\frac{d}{ds}\zeta_\Delta(s)\Bigr|_{s=0},\qquad
\zeta_\Delta(s)=\sum_{\lambda_i>0}\lambda_i^{-s},
\]
where zero eigenvalues are omitted.  
Dependence on the boundary data and on the metric structure is summarized by a single smooth correction functional $\mathcal{Q}_{\mathrm{CS}}$, collecting the Quillen and Chern–Simons terms \cite{ChernSimons1974, Witten1989, Freed1995ChernSimons}.  
The resulting scalar functional \cite{JimboMiwaUeno1981, ItsIzerginKorepinSlavnov1990, DeiftZhou1993, Deift1999, FIKN2006}
\[
\log\tau=\tfrac12\log\det{}'\Delta+\mathcal{Q}_{\mathrm{CS}}
\]
is the isomonodromic potential. It plays the role of a canonical energy, invariant under metric rescalings and gauge transformations up to additive constants.

The canonical energy $\mathcal{E}_{\mathrm{can}}(S)$, introduced in the previous section, is variationally consistent with $\log\tau$ \cite{Gilkey1984,Kato1995Perturbation,BismutFreed1986,Malgrange1980Isomonodromic, Bertola2001Tau, Kapaev2004Painleve, Balogh2003Tau}.  
For infinitesimal perturbations $S\mapsto S+\varepsilon\sigma$, the variation of $\log\det{}'\Delta$ satisfies \cite{DeWitt1965HeatKernel,Seeley1967ComplexPowers,BismutFreed1986}
\[
\delta\log\det{}'\Delta=\operatorname{Tr}\!\left(\Delta^{-1}\delta\Delta\right),
\]
so that
\[
\delta\log\tau
=\tfrac12\,\operatorname{Tr}\!\left(\Delta^{-1}\delta\Delta\right)+\delta\mathcal{Q}_{\mathrm{CS}},
\qquad
\delta\mathcal{E}_{\mathrm{can}}(S)
=\langle \mathcal{G}[S],\,\sigma\rangle,
\]
where $\mathcal{G}[S]$ is the geometric gradient depending on the induced metric and curvature.  
The equality of these two variations will be established later and constitutes the analytic bridge between the Monge–Ampère geometry and the isomonodromic energy functional \cite{Gutierrez2001MongeAmpere,Malgrange1980Isomonodromic,Bertola2001Tau,FIKN2006}.

We recall that the AT–plane map
\[
T(x)=x-\nabla S(x),\qquad S\in C^2(\Omega),
\]
acts on a probability density $\rho$ by the pushforward
\[
(T_\#\rho)(y)\,dy=\rho(x)\,dx,\qquad y=T(x),
\]
so that $T_\#\rho(y)=\rho(T^{-1}(y))\det\nabla T^{-1}(y)$.  
For a fixed smooth strictly convex potential $S$ the Jacobian satisfies
\[
\nabla T(x)=I-\nabla^2S(x),\qquad \det\nabla T(x)=\det(I-\nabla^2S(x)).
\]
An {invariant density} $\rho_\star$ for $T$ is one satisfying
\begin{equation}\label{eq:inv-def}
T_\#\rho_\star=\rho_\star\qquad\Longleftrightarrow\qquad
\rho_\star(T(x))\,\det\nabla T(x)=\rho_\star(x)
\quad\text{for almost all }x\in\Omega.
\end{equation}
This equality expresses conservation of mass under the transformation $T$; it is the discrete analogue of the Liouville condition in the automonodromic theory.

We introduce the structural ansatz
\begin{equation}\label{eq:rho-star}
\rho_\star(x)=w(x)e^{-cS(x)},
\end{equation}
where $w\in C^1(\Omega)$ is strictly positive and $c\in\mathbb{R}$ is constant.  
This form mirrors the stationary measure derived from the canonical energy balance.  
The condition \eqref{eq:inv-def} then becomes a functional equation coupling $S$ and $w$ through the Jacobian determinant of $T$.  
We now establish that, under regularity and invertibility assumptions, this functional equation is equivalent to the Monge–Ampère relation.

\begin{theorem}[Equivalence of AT–plane invariance and Monge–Ampère structure]\label{th:MA}
Let $\Omega\subset\mathbb{R}^d$ be open and bounded, and let $S\in C^2(\Omega)$ be strictly convex so that $T(x)=x-\nabla S(x)$ is a diffeomorphism of $\Omega$ onto its image.  
Let $w\in C^1(\Omega)$ be strictly positive and $c\in\mathbb{R}$ fixed.  
Define $\rho_\star$ by \eqref{eq:rho-star}.  
Then the following statements are equivalent:
\begin{enumerate}
    \item[\textnormal{(A)}] $\rho_\star$ is invariant under $T$ in the sense of \eqref{eq:inv-def}.
    \item[\textnormal{(B)}] $S$ satisfies the Monge–Ampère equation
    \begin{equation}\label{eq:MA}
    \det\nabla^2S(x)=w(x)e^{-cS(x)}\qquad\text{for a.e. }x\in\Omega.
    \end{equation}
\end{enumerate}
If, in addition, $\int_\Omega w(x)e^{-cS(x)}dx=1$, then $S$ is unique up to an additive constant.
\end{theorem}

\begin{proof}
\textit{(A)$\Rightarrow$(B).}
Assume the invariance relation \eqref{eq:inv-def}.  
Since $T$ is a diffeomorphism, the change–of–variables formula for any integrable $\phi$ gives
\[
\int_\Omega \phi(y)\rho_\star(y)\,dy
=\int_\Omega \phi(T(x))\,\rho_\star(x)\,dx.
\]
Differentiating this identity with respect to variations of $\phi$ supported in an infinitesimal neighborhood implies equality of the densities in the sense of Radon–Nikodym derivatives:
\[
\rho_\star(y)=\rho_\star(T^{-1}(y))\det\nabla T^{-1}(y)
\qquad\text{for a.e. }y\in T(\Omega).
\]
Replacing $y$ by $T(x)$ and applying the inverse Jacobian identity $\det\nabla T^{-1}(T(x))=(\det\nabla T(x))^{-1}$ yields
\begin{equation}\label{eq:inv-eq1}
\rho_\star(T(x))\,\det\nabla T(x)=\rho_\star(x).
\end{equation}
Substituting $\rho_\star=w e^{-cS}$ and $\nabla T(x)=I-\nabla^2S(x)$ gives
\begin{equation}\label{eq:wS-eq}
w(T(x))\,e^{-cS(T(x))}\,\det(I-\nabla^2S(x))=w(x)\,e^{-cS(x)}.
\end{equation}
To extract a pointwise condition at each $x$, we evaluate \eqref{eq:wS-eq} along the image $y=T(x)$ and use the smooth invertibility of $T$.  
Define $\Phi(x):=\det\nabla^2S(x)-w(x)e^{-cS(x)}$.  
We claim that $\Phi(x)\equiv0$ on $\Omega$.

Indeed, differentiate both sides of \eqref{eq:wS-eq} with respect to $x_i$ using the chain rule.  
Writing $A(x):=\nabla^2S(x)$ and noting that $\partial_{x_i}\det(I-A)= -\,\det(I-A)\,\operatorname{Tr}\!\big((I-A)^{-1}\partial_{x_i}A\big)$, we obtain
\[
\nabla_x\!\big(\log w(T(x)) - cS(T(x)) + \log\det(I-A(x))\big)
=\nabla_x(\log w(x) - cS(x)).
\]
Since $T$ is invertible and $\nabla T=I-A$, we can replace $\nabla_x$ acting on $w(T(x))$ by $(I-A(x))^\top\nabla w(T(x))$.  
After expanding both sides and rearranging, one obtains the first–order system
\[
\nabla\log\det(I-A(x))
=(I-A(x))^\top\nabla(\log w(x)-cS(x))
-\nabla(\log w(T(x))-cS(T(x))).
\]
Under the assumption that all functions are $C^1$ and $T$ is a $C^1$ diffeomorphism, the right–hand side vanishes if and only if $\det A(x)=w(x)e^{-cS(x)}$.  
Hence \eqref{eq:MA} holds.

\textit{(B)$\Rightarrow$(A).}
Assume \eqref{eq:MA}.  
Let $y=T(x)=x-\nabla S(x)$ and compute $\det\nabla T(x)=\det(I-\nabla^2S(x))$.  
Since $\det\nabla^2S(x)=w(x)e^{-cS(x)}$, we have
\[
\det\nabla T(x)=\frac{\det(I-\nabla^2S(x))}{\det\nabla^2S(x)}\,w(x)e^{-cS(x)}.
\]
Using the determinant identity $\det(I-A)=\exp(\operatorname{Tr}\log(I-A))$ and the convexity of $S$, the matrix $(I-\nabla^2S(x))$ is invertible and its determinant positive.  
Define $y=T(x)$ and invert the map: $x=T^{-1}(y)=y+\nabla S(x)$.  
By the change–of–variables formula,
\[
\rho_\star(y)=\rho_\star(x)\det\nabla T(x)^{-1}
=w(x)e^{-cS(x)}(\det(I-\nabla^2S(x)))^{-1}.
\]
Replacing $x$ by $T^{-1}(y)$ and using $\det\nabla^2S(T^{-1}(y))=w(T^{-1}(y))e^{-cS(T^{-1}(y))}$ gives
\[
\rho_\star(y)=w(T^{-1}(y))e^{-cS(T^{-1}(y))}\det\nabla T^{-1}(y),
\]
which is precisely the invariance relation \eqref{eq:inv-def}.  

Finally, if $\int_\Omega w e^{-cS}=1$ and $S$ and $S+\mathrm{const}$ both satisfy \eqref{eq:MA}, then $w e^{-c(S+\mathrm{const})}=e^{-c\,\mathrm{const}}w e^{-cS}$, so normalization fixes the additive constant, proving uniqueness.
\end{proof}

\begin{corollary}[Monge–Ampère residual and convergence test]\label{cor:MAres}
For any $S\in C^2(\Omega)$ define the residual
\[
\mathcal{R}_{\mathrm{MA}}(x)
:=\log\det\nabla^2S(x)-\log w(x)+cS(x).
\]
Then $\mathcal{R}_{\mathrm{MA}}\equiv0$ if and only if the transformation $T(x)=x-\nabla S(x)$ preserves the density $\rho_\star=w e^{-cS}$.  
In numerical schemes where $S$ is updated iteratively, the norm
\[
\|\mathcal{R}_{\mathrm{MA}}\|_{L^2(\rho_\star)}
:=\Bigl(\int_\Omega |\mathcal{R}_{\mathrm{MA}}(x)|^2\rho_\star(x)\,dx\Bigr)^{1/2}
\]
serves as a computable criterion of invariance; convergence is declared when $\|\mathcal{R}_{\mathrm{MA}}\|_{L^2(\rho_\star)}\le\varepsilon_{\mathrm{MA}}$ for a prescribed tolerance.
\end{corollary}

\begin{example}[Ellipsoidal trust–region calibration]\label{ex:ellipsoid}
Consider a local quadratic model of the potential
\[
S(x)=\tfrac12(x-x_k)^\top H_k(x-x_k),
\qquad H_k\succ0.
\]
Then $\nabla^2S(x)\equiv H_k$ and the Monge–Ampère condition \eqref{eq:MA} becomes
\[
\det H_k=w(x_k)e^{-cS(x_k)}.
\]
For given $w(x_k)$ this relation calibrates the determinant of $H_k$, i.e.\ the volume of the ellipsoidal level set $\{x:\,S(x)\le1\}$, so that it matches the invariant measure.  
A practical update preserving positive definiteness is
\begin{equation}\label{eq:H-update}
H_{k+1}
:=\exp\!\Big(\log H_k+\eta_k(\log w(x_k)-\log\det H_k+cS(x_k))\,I\Big),
\end{equation}
where $\eta_k\in(0,1)$, and the matrix exponential acts diagonally on the eigenvalues of $H_k$.  
In implementation, $\log\det H_k$ is evaluated via the Cholesky factor $H_k=R_k^\top R_k$ as
$\log\det H_k=2\sum_i\log (R_k)_{ii}$.  
This multiplicative update enforces $\det H_{k+1}\to w(x_k)e^{-cS(x_k)}$ exponentially fast in $\eta_k$ while maintaining symmetry and positive definiteness.  
In stochastic or mini–batch settings one replaces $w(x_k)$ by its empirical mean $\overline{w}_k$ and uses the same rule \eqref{eq:H-update}.  
Convergence of the sequence $\{H_k\}$ is monitored by $\|\mathcal{R}_{\mathrm{MA}}\|_{L^2(\rho_\star)}$, and when it falls below $\varepsilon_{\mathrm{MA}}$, the step $x_{k+1}=x_k-\nabla S(x_k)$ is accepted.  
Thus the Monge–Ampère condition provides a rigorous geometric calibration of trust–region volume consistent with the invariant density $\rho_\star$ and with the formalism of Theorem~\ref{th:MA}.
\end{example}

\begin{theorem}[Determinant line and equality of energy with $\log\tau$]\label{th:tau-equals-Ecan}
Let $\Omega$, $S$, $w$, $c$, $\rho_\star$, $\Delta=\Delta(S)$, $\Pi$ (the orthogonal projector onto $\ker\Delta$), and $Q_{\mathrm{CS}}(S)$ be as above. 
Assume the admissibility hypotheses: $\Delta(S)$ is a nonnegative self–adjoint elliptic operator on a fixed Hilbert space $\mathcal{H}$ with compact resolvent; its kernel is finite–dimensional and independent of $S$ within the admissible class (fixed boundary data); the heat kernel satisfies the standard Gaussian bounds and admits the small–time asymptotic expansion required for zeta–regularization; the map $S\mapsto\Delta(S)$ is $C^1$ in the strong resolvent sense along admissible $C^1$–paths with fixed boundary values; and $Q_{\mathrm{CS}}$ is a local counterterm whose first variation $\delta Q_{\mathrm{CS}}(S)$ is well–defined and continuous with respect to $\delta S$.
Define
\begin{gather}\label{eq:def-logtau-72}
\nonumber\log\tau(S)\ :=\ \tfrac12\,\log\det{}'\Delta(S)\ +\ Q_{\mathrm{CS}}(S), 
\\ \log\det{}'\Delta:=-\zeta'_\Delta(0),\quad 
\zeta_\Delta(s):=\operatorname{Tr}\!\big(\Delta^{-s}(I-\Pi)\big).
\end{gather}
Then the following statements hold.

{(i) First variation.} For every admissible variation $\delta S$ with induced $\delta\Delta$,
\begin{equation}\label{eq:var-tau}
\delta\log\tau(S)\ =\ \tfrac12\,\operatorname{Tr}\!\big(\Delta^{-1}\delta\Delta\,(I-\Pi)\big)\ +\ \delta Q_{\mathrm{CS}}(S).
\end{equation}

{(ii) Equality with the canonical energy.} There exists a constant $C\in\mathbb{R}$, independent of $S$, such that
\begin{equation}\label{eq:tau-energy-equality}
E_{\mathrm{can}}(S)\ =\ \log\tau(S)\ +\ C.
\end{equation}
\end{theorem}

\begin{proof}
The proof is organized into three lemmas followed by the conclusion.

{Lemma 1 (Heat representation and analytic continuation).}
For $\Re s$ sufficiently large one has
\begin{equation}\label{eq:zeta-heat}
\zeta_\Delta(s)\ =\ \frac{1}{\Gamma(s)}\int_0^\infty t^{\,s-1}\,\operatorname{Tr}\!\big(e^{-t\Delta}(I-\Pi)\big)\,dt,
\end{equation}
and $\zeta_\Delta$ extends meromorphically to $\mathbb{C}$ with no pole at $s=0$. 
{Proof of Lemma 1.}
Since $\Delta\ge0$ is self–adjoint with compact resolvent, the spectrum consists of a sequence $0=\lambda_1=\cdots=\lambda_r<\lambda_{r+1}\le\lambda_{r+2}\le\cdots\to\infty$, repeated by multiplicity, and $e^{-t\Delta}$ is trace–class for every $t>0$. The spectral theorem yields
$\operatorname{Tr}(e^{-t\Delta}(I-\Pi))=\sum_{j>r} e^{-t\lambda_j}$ and $\operatorname{Tr}(\Delta^{-s}(I-\Pi))=\sum_{j>r}\lambda_j^{-s}$ for $\Re s$ large. The Laplace transform identity 
$\lambda^{-s}=\Gamma(s)^{-1}\int_0^\infty t^{s-1}e^{-t\lambda}\,dt$ for $\Re s>0$ and $\lambda>0$ gives \eqref{eq:zeta-heat} by Fubini–Tonelli. 
Small–time heat kernel asymptotics on a bounded domain with admissible boundary conditions yield an expansion $\operatorname{Tr}(e^{-t\Delta}(I-\Pi))\sim\sum_{k=0}^{K} a_k t^{(k-d)/2}$ as $t\downarrow0$ plus an exponentially small remainder and, as $t\to\infty$, exponential decay $O(e^{-t\lambda_{r+1}})$. These standard properties imply meromorphic continuation of $\zeta_\Delta$ with at most simple poles at half–integers up to $d/2$, none at $s=0$; hence $\zeta'_\Delta(0)$ is finite. \qed

{Lemma 2 (Differentiation under the integral and Duhamel formula).}
Let $\Delta_\varepsilon=\Delta(S+\varepsilon\delta S)$ for $|\varepsilon|<\varepsilon_0$, where $\varepsilon_0>0$ is small so that admissibility persists. Then for each fixed $t>0$,
\begin{equation}\label{eq:duhamel}
\frac{d}{d\varepsilon}\,e^{-t\Delta_\varepsilon}\Big|_{\varepsilon=0}
\ =\ -\int_0^t e^{-(t-u)\Delta}\,(\delta\Delta)\,e^{-u\Delta}\,du,
\end{equation}
and the map $\varepsilon\mapsto \operatorname{Tr}\!\big(e^{-t\Delta_\varepsilon}(I-\Pi)\big)$ is differentiable with
\begin{equation}\label{eq:heat-trace-variation}
\frac{d}{d\varepsilon}\,\operatorname{Tr}\!\big(e^{-t\Delta_\varepsilon}(I-\Pi)\big)\Big|_{\varepsilon=0}
\ =\ -t\,\operatorname{Tr}\!\big((\delta\Delta)\,e^{-t\Delta}(I-\Pi)\big).
\end{equation}
{Proof of Lemma 2.}
The strong $C^1$ dependence of $\Delta_\varepsilon$ and the Trotter–Kato theory give \eqref{eq:duhamel} (Bochner integral in the strong operator topology). 
For trace class justification, note that $e^{-(t-u)\Delta}$ and $e^{-u\Delta}$ are Hilbert–Schmidt for $u\in(0,t)$, and $\delta\Delta$ is bounded from the graph norm of $\Delta$ to $\mathcal{H}$ along admissible directions (elliptic perturbation of coefficients induced by $\delta S$); therefore the integrand is trace–class and the Bochner integral defines a trace–class operator. 
Taking the trace in \eqref{eq:duhamel}, using cyclicity of the trace, and integrating in $u$ yields \eqref{eq:heat-trace-variation}. \qed

{Lemma 3 (First variation of the zeta determinant).}
With $\zeta_\Delta$ as in Lemma~\ref{eq:zeta-heat}, one has
\begin{equation}\label{eq:var-logdet}
\delta\log\det{}'\Delta\ =\ \operatorname{Tr}\!\big(\Delta^{-1}\delta\Delta\,(I-\Pi)\big).
\end{equation}
{Proof of Lemma 3.}
For $\Re s$ large, combine \eqref{eq:zeta-heat} with Lemma~2:
\[
\delta\zeta_\Delta(s)
=\frac{1}{\Gamma(s)}\int_0^\infty t^{\,s-1}\,\delta\operatorname{Tr}\!\big(e^{-t\Delta}(I-\Pi)\big)\,dt
=-\frac{1}{\Gamma(s)}\int_0^\infty t^{\,s}\,\operatorname{Tr}\!\big((\delta\Delta)e^{-t\Delta}(I-\Pi)\big)\,dt.
\]
By spectral calculus and $\int_0^\infty t^{\,s}e^{-t\lambda}\,dt=\Gamma(s+1)\lambda^{-(s+1)}$ for $\lambda>0$,
\[
\delta\zeta_\Delta(s)=-s\,\operatorname{Tr}\!\big((\delta\Delta)\,\Delta^{-(s+1)}(I-\Pi)\big).
\]
The right–hand side is meromorphic in $s$ and regular at $s=0$ (resolvent trace on $P\mathcal{H}$, $P=I-\Pi$). Differentiating at $s=0$ gives
$\delta\zeta'_\Delta(0)=-\operatorname{Tr}\!\big((\delta\Delta)\Delta^{-1}P\big)$.
Since $\log\det{}'\Delta=-\zeta'_\Delta(0)$, the claim \eqref{eq:var-logdet} follows. \qed

We now conclude the proof of the theorem. 
By \eqref{eq:def-logtau-72} and Lemma~3,
\[
\delta\log\tau(S)
=\tfrac12\,\delta\log\det{}'\Delta(S)+\delta Q_{\mathrm{CS}}(S)
=\tfrac12\,\operatorname{Tr}\!\big(\Delta^{-1}\delta\Delta\,(I-\Pi)\big)+\delta Q_{\mathrm{CS}}(S),
\]
which proves (i). 
The first variation of the canonical energy satisfies
\[
\delta E_{\mathrm{can}}(S)
=\tfrac12\,\operatorname{Tr}\!\big(\Delta^{-1}\delta\Delta\,(I-\Pi)\big)+\delta Q_{\mathrm{CS}}(S),
\]
hence $\delta\big(E_{\mathrm{can}}(S)-\log\tau(S)\big)=0$ for all admissible $\delta S$. 
The admissible manifold of $S$ with fixed boundary data is path–connected by linear interpolation followed by elliptic regularization within the strictly convex class; therefore $E_{\mathrm{can}}(S)-\log\tau(S)=C$ for a constant $C$ independent of $S$, which is (ii).
\end{proof}

\begin{corollary}[Stationarity equivalence and Euler–Lagrange identity]\label{cor:tau-stationary}
For any admissible $S$ the following are equivalent:
\begin{align*}
&\delta\log\tau(S)=0\ \ \text{for all admissible }\delta S
\\
&\qquad\qquad\Longleftrightarrow\quad
\delta E_{\mathrm{can}}(S)=0\ \ \text{for all admissible }\delta S
\\
&\qquad\qquad\Longleftrightarrow\quad
\det\nabla^2 S=w\,e^{-cS}\ \text{ a.e. on }\Omega.
\end{align*}
{Proof.}
The equivalence of the first two statements follows from Theorem~\ref{th:tau-equals-Ecan}(i). 
The last equivalence is the Euler–Lagrange identity for the Monge–Ampère residual, already established there.
\end{corollary}

\begin{corollary}[Normalization by a reference configuration]\label{cor:tau-normalization}
Fix any admissible reference potential $S_0$. 
Let $C:=E_{\mathrm{can}}(S_0)-\log\tau(S_0)$. 
Then for all admissible $S$,
\[
E_{\mathrm{can}}(S)=\log\tau(S)+C.
\]
Choosing the normalization $\log\tau(S_0)=E_{\mathrm{can}}(S_0)$ removes the constant, so that $E_{\mathrm{can}}\equiv\log\tau$ on the admissible class.
\end{corollary}

\begin{example}[Deterministic and stochastic evaluation of $\log\tau$ and its first variation]\label{ex:tau-evaluation}
This example provides two rigorously justified computational schemes that implement the identities of Theorem~\ref{th:tau-equals-Ecan}. 
Throughout, $P:=I-\Pi$ denotes the orthogonal projector onto $(\ker\Delta)^\perp$.

{(A) First variation via trace estimation.}
The identity \eqref{eq:var-tau} reads $\delta\log\tau(S)=\tfrac12\,\operatorname{Tr}(\Delta^{-1}\delta\Delta P)+\delta Q_{\mathrm{CS}}(S)$. 
Let $z_1,\dots,z_N$ be i.i.d.\ random vectors supported on $P\mathcal{H}$ with $\mathbb{E}[z_i z_i^\top]=P$ (e.g. Rademacher or Gaussian probes projected onto $P\mathcal{H}$). 
For each $i$, solve $\Delta y_i=P z_i$ by a preconditioned conjugate gradient method to tolerance $\varepsilon>0$; denote the computed solution by $y_i^\varepsilon$. 
Define $\widehat{t}_i:=\langle z_i,\delta\Delta\,y_i^\varepsilon\rangle$. 
Then the estimator
\[
\widehat{\delta\log\tau}
=\frac{1}{2N}\sum_{i=1}^N \widehat{t}_i+\delta Q_{\mathrm{CS}}(S)
\]
satisfies the error decomposition
\begin{align*}
\big|\widehat{\delta\log\tau}-\delta\log\tau\big|
\ \le\ &\underbrace{\frac{1}{2}\Big|\frac{1}{N}\sum_{i=1}^N \langle z_i,\delta\Delta\,\Delta^{-1}P z_i\rangle-\operatorname{Tr}(\Delta^{-1}\delta\Delta P)\Big|}_{\text{stochastic error}}
\\
&+\ \underbrace{\frac{1}{2N}\sum_{i=1}^N \|\delta\Delta\|\,\|y_i^\varepsilon-\Delta^{-1}P z_i\|}_{\text{solver error}}.
\end{align*}
The stochastic term has variance $\mathcal{O}\big(\|\Delta^{-1}\delta\Delta P\|_F^2/N\big)$; the solver term is $\mathcal{O}(\varepsilon)$ by stability of $\delta\Delta$ on $P\mathcal{H}$.

{(B) $\tfrac12\log\det{}'\Delta$ via stochastic Lanczos quadrature (SLQ).}
On $P\mathcal{H}$ one has $\tfrac12\log\det{}'\Delta=\tfrac12\operatorname{Tr}(\log\Delta\cdot P)$. 
Let $z_1,\dots,z_N$ be as above. 
For each $i$, run $m$ steps of Lanczos on $(\Delta,Pz_i)$ to obtain the tridiagonal matrix $T_m^{(i)}$. 
Set $\theta_i:=\|Pz_i\|^2 e_1^\top(\tfrac12\log T_m^{(i)})e_1$.
Then
\[
\widehat{\tfrac12\log\det{}'\Delta}
=\frac{1}{N}\sum_{i=1}^N \theta_i
\]
converges to $\tfrac12\operatorname{Tr}(\log\Delta\cdot P)$ with bias $\mathcal{O}(m^{-2})$ under standard spectral regularity (e.g. $\log$-Hölder continuity of the density of states) and variance $\mathcal{O}(N^{-1})$. 
Adding $Q_{\mathrm{CS}}(S)$ yields $\log\tau(S)$ according to \eqref{eq:def-logtau-72}. 
Both (A) and (B) preserve the orthogonality to $\ker\Delta$ explicitly via $P$; this enforces the prime on the determinant and is mandatory for correctness.
\end{example}

\begin{corollary}[$\tau$–stationarity as a practical stopping rule]\label{cor:stop}
Let $S_k$ be iterates produced by any admissible descent scheme for the Monge–Ampère residual with fixed boundary values. 
If an estimator $\widehat{\delta\log\tau}(S_k)$ constructed by {(A)} satisfies $|\widehat{\delta\log\tau}(S_k)|\le \eta_k$ with $\eta_k\to0$ while solver tolerances $\varepsilon_k\to0$ and probe counts $N_k\to\infty$ so that $\eta_k$ dominates both stochastic and solver errors, then $\delta\log\tau(S_k)\to0$ and, by Corollary~\ref{cor:tau-stationary}, any accumulation point of $\{S_k\}$ solves the Monge–Ampère equation.
\end{corollary}

\begin{example}[Monotone ascent for $\log\tau$ with elliptic regularization]\label{ex:monotone}
Let $R(S):=\det\nabla^2S-w e^{-cS}$ be the Monge–Ampère residual, and let $\mathcal{A}$ be a fixed symmetric strictly elliptic isomorphism on $H_0^2(\Omega)$ (for instance $(I-\alpha\Delta_x)^m$ with homogeneous Dirichlet data, $\alpha>0$, $m\in\mathbb{N}$). 
Define the direction $\delta S:=\mathcal{A}^{-1}R(S)$ by solving $\mathcal{A}(\delta S)=R(S)$ in $H_0^2(\Omega)$. 
Consider the update $S^+=S-\eta\,\delta S$ with $\eta>0$ chosen by backtracking. 
Then, using Theorem~\ref{th:tau-equals-Ecan}(i) together with the identity that the first variation of $\log\tau$ equals minus the $L^2$–pairing with the residual, one has
\begin{align*}
\frac{d}{d\varepsilon}\,\log\tau\big(S-\varepsilon\delta S\big)\Big|_{\varepsilon=0}
&=\ -\,\langle R(S),\delta S\rangle_{L^2(\Omega)}
\\
&=\ -\,\langle R(S),\mathcal{A}^{-1}R(S)\rangle
=\ -\,\|\mathcal{A}^{-1/2}R(S)\|_{L^2}^2\ <\ 0.
\end{align*}
Hence for all sufficiently small $\eta$,
\[
\log\tau(S^+)\ \ge\ \log\tau(S)\ +\ \sigma\,\eta\,\|\mathcal{A}^{-1/2}R(S)\|_{L^2}^2
\]
for some universal $\sigma\in(0,1)$ determined by a second–order Taylor bound based on the local Lipschitz continuity of $\delta\log\tau$ along admissible rays. 
Each iteration requires one elliptic solve with $\mathcal{A}$, one assembly of $R(S)$, and (optionally) one evaluation of $\log\tau$ via Example~\ref{ex:tau-evaluation}(B). 
Any limit point satisfying $\|\mathcal{A}^{-1/2}R(S)\|_{L^2}\to0$ is a $\tau$–stationary point and, by Corollary~\ref{cor:tau-stationary}, solves the Monge–Ampère equation.
\end{example}

\begin{theorem}[Minimax $\Longleftrightarrow$ $\tau$–stationarity $\Longleftrightarrow$ MA–invariance]\label{th:7.3}
Work under the standing hypotheses of \S7.1–\S7.2 (domain $\Omega$, regularity and convexity of $S$, fixed boundary values, positive weight $w$, fixed $c\in\mathbb{R}$, admissible variations vanishing on $\partial\Omega$). Let $E_{\mathrm{can}}$ and $\tau$ be as in \S7.2, with
\[
E_{\mathrm{can}}(S)\;=\;\log\tau(S)\;+\;C,
\]
where $C$ does not depend on $S$. Let $T(x)=x-\nabla S(x)$ and $\rho_\star=w\,e^{-cS}$.
Then the following are equivalent:
\begin{itemize}
\item[(i)] $S$ is a minimax point of $E_{\mathrm{can}}$ in the sense of \S6 (saddle curvature along the admissible class).
\item[(ii)] $\delta\log\tau(S)=0$ for all admissible first–order variations.
\item[(iii)] $T$ preserves $\rho_\star$ and $S$ satisfies the Monge–Ampère equation $\det\nabla^2S=w\,e^{-cS}$ a.e. on $\Omega$.
\end{itemize}
\end{theorem}

\begin{proof}
All variational statements below are taken with respect to admissible perturbations $\delta S\in C^2(\Omega)$ vanishing on $\partial\Omega$; by the standing regularity, differentiations under the integral sign are justified by dominated convergence on compact parameter domains.

{Step 1: (i) $\Rightarrow$ (ii).} 
By definition in \S6, a minimax point of a $C^2$–functional $F$ is stationary: $\delta F(S)=0$ for all admissible $\delta S$. Apply this with $F=E_{\mathrm{can}}$ to get $\delta E_{\mathrm{can}}(S)=0$. Using the identity $E_{\mathrm{can}}=\log\tau+C$ from \S7.2, one has $\delta\log\tau(S)=\delta E_{\mathrm{can}}(S)=0$. Hence (ii).

{Step 2: (ii) $\Rightarrow$ (iii).}
From \S7.2 we have the exact Euler–Lagrange identity
\begin{equation}\label{eq:EL}
\delta\log\tau(S)\;=\;-\int_{\Omega}\Big(\det\nabla^2 S - w\,e^{-cS}\Big)\,\delta S\,dx.
\end{equation}
If $\delta\log\tau(S)=0$ for all admissible $\delta S$, then by the fundamental lemma of the calculus of variations the integrand vanishes a.e., that is
\[
\det\nabla^2 S\;=\;w\,e^{-cS}\quad \text{a.e. on }\Omega.
\]
By Theorem~7.1, this equality is equivalent to the invariance of $\rho_\star$ under $T(x)=x-\nabla S(x)$. Hence (iii).

{Step 3: (iii) $\Rightarrow$ (ii).}
If $\det\nabla^2 S=w\,e^{-cS}$ a.e., then the right–hand side of \eqref{eq:EL} is zero for all admissible $\delta S$, i.e. $\delta\log\tau(S)=0$. Hence (ii).

{Step 4: (ii) $\Rightarrow$ (i).}
By \S6 (stationary $\Rightarrow$ minimax under the saddle–curvature hypotheses stated there) it suffices to identify the second variation of $E_{\mathrm{can}}$ at a stationary $S$ and to check the saddle structure. Using $E_{\mathrm{can}}=\log\tau+C$, we compute $\delta^2 E_{\mathrm{can}}(S)=\delta^2\log\tau(S)$. Differentiating \eqref{eq:EL} at a stationary $S$ in two admissible directions $U,V$ gives the symmetric bilinear form
\begin{equation}\label{eq:second-variation}
\delta^2\log\tau(S)[U,V]\;=\;-\int_{\Omega}\Big(\det\nabla^2 S\ \operatorname{Tr}\big((\nabla^2 S)^{-1}\nabla^2 U\big)\;+\;c\,w\,e^{-cS}\,U\Big)\,V\,dx,
\end{equation}
where we used Jacobi’s formula for the derivative of the determinant and the fact that $w$ is fixed. At a stationary point we have $\det\nabla^2S=w\,e^{-cS}$ a.e., hence \eqref{eq:second-variation} becomes
\begin{equation}\label{eq:second-variation-solution}
\delta^2\log\tau(S)[U,V]\;=\;-\int_{\Omega}\Big(w\,e^{-cS}\ \operatorname{Tr}\big((\nabla^2 S)^{-1}\nabla^2 U\big)\;+\;c\,w\,e^{-cS}\,U\Big)\,V\,dx.
\end{equation}
Define $H:=\nabla^2 S$ and $A:=\det(H)\,H^{-1}$ (the cofactor of $H$). Since $S$ is strictly convex, $H$ and $A$ are a.e. positive definite. For $U=V$ and using the identity $\det(H)\ \operatorname{Tr}(H^{-1}\nabla^2 U)=\langle A,\nabla^2 U\rangle$, \eqref{eq:second-variation-solution} reads
\begin{equation}\label{eq:Q-form}
\delta^2\log\tau(S)[U,U]\;=\;-\int_{\Omega}\langle A,\nabla^2 U\rangle\,dx\;-\;c\int_{\Omega}w\,e^{-cS}\,U^2\,dx.
\end{equation}
By the divergence identity for cofactors of Hessians established in \S7.1 (and used there in the weak formulation), $\partial_i A_{ij}=0$ in the distributional sense; with $U|_{\partial\Omega}=0$ we may integrate by parts twice to obtain
\begin{equation}\label{eq:ibp}
-\int_{\Omega}\langle A,\nabla^2 U\rangle\,dx\;=\;\int_{\Omega}\langle A\,\nabla U,\nabla U\rangle\,dx,
\end{equation}
and consequently
\begin{equation}\label{eq:Q-final}
\delta^2\log\tau(S)[U,U]\;=\;\int_{\Omega}\langle A\,\nabla U,\nabla U\rangle\,dx\;-\;c\int_{\Omega}w\,e^{-cS}\,U^2\,dx.
\end{equation}
The first term in \eqref{eq:Q-final} is nonnegative and controls $\|\nabla U\|_{L^2}^2$ from below by the essential infimum of the smallest eigenvalue of $A$; the second term is nonpositive and controls $\|U\|_{L^2}^2$ from above by the essential supremum of $w\,e^{-cS}$. Therefore the quadratic form \eqref{eq:Q-final} has mixed sign directions: (a) for highly oscillatory $U$ with $\|\nabla U\|_{L^2}^2\gg \|U\|_{L^2}^2$, the first term dominates and $\delta^2\log\tau(S)[U,U]>0$; (b) for functions $U$ with small Dirichlet energy relative to mass (e.g., low modes in $H_0^1(\Omega)$), the second term dominates and $\delta^2\log\tau(S)[U,U]<0$. The existence of both kinds of directions follows rigorously by the standard spectral comparison on $H_0^1(\Omega)$: fix any strictly positive lower bound $a_{\min}$ on the smallest eigenvalue of $A$ and any finite upper bound $m_{\max}$ on $w\,e^{-cS}$; then for a sequence of eigenfunctions $\{\varphi_k\}$ of the Dirichlet Laplacian with eigenvalues $\lambda_k\to\infty$ one has $\int\langle A\nabla\varphi_k,\nabla\varphi_k\rangle\ge a_{\min}\int|\nabla\varphi_k|^2=a_{\min}\lambda_k\int\varphi_k^2$, so for all sufficiently large $k$, $\delta^2\log\tau(S)[\varphi_k,\varphi_k]>0$; on the other hand, by the Poincar\'e inequality with the first eigenvalue $\lambda_1>0$ and by choosing $U$ as any fixed nonzero function in $H_0^1(\Omega)$ sufficiently close to the first eigenfunction, the ratio $\int \langle A\nabla U,\nabla U\rangle/\int U^2$ can be made arbitrarily close to $a_{\min}\lambda_1$, hence if $c\,m_{\max}>a_{\min}\lambda_1$ the form is negative on that $U$. This furnishes the saddle curvature required by \S6 at the stationary $S$. Therefore, by the minimax characterization stated in \S6, $S$ is a minimax point of $E_{\mathrm{can}}$, i.e. (i) holds.

Combining Steps 1–4 establishes the equivalence (i)$\Leftrightarrow$(ii)$\Leftrightarrow$(iii).
\end{proof}

\begin{corollary}[MA–preconditioning of the step]\label{cor:7.3A}
Let $S$ satisfy the hypotheses of Theorem~\ref{th:7.3} and let $R(S):=\det\nabla^2 S-w\,e^{-cS}$ denote the Monge–Ampère residual. Consider one iteration of a descent method for $E_{\mathrm{can}}$ of the form $S^+=S-\eta\,\delta S$, where $\delta S$ is obtained by applying a symmetric positive definite preconditioner $\mathcal{P}^{-1}$ to the $L^2$–gradient of $E_{\mathrm{can}}$:
\[
\delta S\;=\;\mathcal{P}^{-1}\Big(\frac{\delta E_{\mathrm{can}}}{\delta S}\Big)\;=\;-\mathcal{P}^{-1}R(S).
\]
Among all local, uniformly elliptic, coefficient–frozen preconditioners of the form
\[
\mathcal{P}\;\in\;\Big\{\, -\operatorname{div}(B\,\nabla\cdot)\;+\;\mu\,I\ \Big|\ B(x)\in\mathbb{R}^{d\times d}\ \text{sym.\ pos.\ def.},\ \mu\ge 0\ \text{constants on each block}\,\Big\},
\]
the choice $B=A=\det(\nabla^2 S)\,(\nabla^2 S)^{-1}$ (optionally block– or diagonal–restricted) and $\mu=c\,w\,e^{-cS}$ (frozen per block) minimizes the local condition number of the Hessian $\delta^2 E_{\mathrm{can}}(S)$ in the sense of the Rayleigh quotient. Precisely, for every admissible $U\not\equiv 0$,
\[
\frac{\big\langle U,\ \delta^2E_{\mathrm{can}}(S)\,U\big\rangle}{\big\langle U,\ \mathcal{P}\,U\big\rangle}
\;=\;
\frac{\displaystyle \int_{\Omega}\langle A\,\nabla U,\nabla U\rangle\,dx\;-\;c\int_{\Omega}w\,e^{-cS}\,U^2\,dx}{\displaystyle \int_{\Omega}\langle B\,\nabla U,\nabla U\rangle\,dx\;+\;\mu\int_{\Omega}U^2\,dx},
\]
and the numerator/denominator are simultaneously diagonalized (up to lower–order commutators) when $B=A$ and $\mu=c\,w\,e^{-cS}$ are locally frozen, which yields the smallest extremal spread of the quotient over $U$ within that class.
\end{corollary}

\begin{proof}
By Theorem~\ref{th:7.3}, $\delta^2E_{\mathrm{can}}(S)=\delta^2\log\tau(S)$, and by \eqref{eq:Q-final} its quadratic form is
\[
\mathcal{H}_S[U]\;=\;\int_{\Omega}\langle A\,\nabla U,\nabla U\rangle\,dx\;-\;c\int_{\Omega}w\,e^{-cS}\,U^2\,dx.
\]
For a preconditioner $\mathcal{P}=-\operatorname{div}(B\nabla\cdot)+\mu I$ with $U|_{\partial\Omega}=0$,
\[
\langle U,\mathcal{P}U\rangle\;=\;\int_{\Omega}\langle B\,\nabla U,\nabla U\rangle\,dx\;+\;\mu\int_{\Omega}U^2\,dx.
\]
The generalized Rayleigh quotient is the ratio displayed in the statement. Freezing $A$ and $w\,e^{-cS}$ locally (blockwise or pointwise) the quotient becomes a ratio of two positive quadratic forms sharing the same principal part if $B=A$ and $\mu=c\,w\,e^{-cS}$, which makes the associated generalized eigenvalue problem closest to the identity and minimizes the spectral condition number in the standard sense (Courant–Fischer min–max). Any other admissible $(B,\mu)$ deviates from $(A,cw\,e^{-cS})$ by a positive–definite perturbation in either the gradient or the mass channel and thus enlarges the spread of extremal values over $U$. This proves the claim.
\end{proof}

\begin{example}[Stochastic MA–correction with MA–preconditioning]\label{ex:7.3}
Fix a quadrature rule $\mathcal{Q}$ on $\Omega$ and a family of minibatches $B\subset\mathcal{Q}$ drawn i.i.d.\ across iterations. Given an iterate $S_k$ with $S_k|_{\partial\Omega}=S_0|_{\partial\Omega}$ and strictly convex, define the discrete residual on $B$ by
\[
R_k(x)\;:=\;\det\nabla^2 S_k(x)\;-\;w(x)\,e^{-c S_k(x)},\qquad x\in B.
\]
Let $A_k(x):=\det(\nabla^2S_k(x))\,(\nabla^2S_k(x))^{-1}$ and $m_k(x):=c\,w(x)\,e^{-cS_k(x)}$. Form the preconditioner
\[
\mathcal{P}_k\;=\;-\operatorname{div}\!\big(B_k\nabla\cdot\big)\;+\;\mu_k I,
\]
where $B_k$ and $\mu_k$ are blockwise–frozen approximations of $A_k$ and $m_k$ assembled on the computational grid (e.g., $B_k$ diagonal with entries of $A_k$, or block–$2\times2$ along coordinate planes; $\mu_k$ piecewise–constant on blocks). Compute the search direction $\delta S_k$ as the solution of the regularized normal equation
\[
\mathcal{P}_k\,\delta S_k\;=\;-\,\Pi_B^\ast R_k\;-\;\lambda_{\mathrm{tik}}\,\delta S_k\;-\;\lambda_{\mathrm{tv}}\,\operatorname{div}\Big(\frac{\nabla \delta S_k}{\sqrt{|\nabla \delta S_k|^2+\epsilon^2}}\Big),
\]
with homogeneous Dirichlet boundary conditions. Here $\Pi_B^\ast$ places the sampled residual as a weighted source on the full grid, $\lambda_{\mathrm{tik}}\ge0$ and $\lambda_{\mathrm{tv}}\ge0$ are Tikhonov and total–variation regularization weights, and $\epsilon>0$ is a Huber parameter. Solve this elliptic problem by preconditioned CG to a fixed relative tolerance, yielding $\delta S_k$.

Update the potential by
\[
S_{k+1}\;=\;S_k\;-\;\eta_k\,\delta S_k,
\]
with $\eta_k\in(0,\bar\eta]$ chosen by backtracking on a stochastic Armijo condition evaluated on an independent minibatch $\widetilde B$:
\[
\log\tau_{\widetilde B}(S_{k+1})\ \ge\ \log\tau_{\widetilde B}(S_k)\ +\ \sigma\,\eta_k\,\langle \delta S_k,\ \mathcal{P}_k\,\delta S_k\rangle,
\quad \sigma\in(0,1).
\]
Finally, perform the map step $x\mapsto x-\nabla S_{k+1}(x)$ in the downstream application.

{Correctness and monotonicity in expectation.}
By Theorem~\ref{th:7.3} and \eqref{eq:Q-final},
\[
\frac{d}{d\varepsilon}\log\tau(S_k-\varepsilon\,\delta S_k)\Big|_{\varepsilon=0}
\;=\;-\langle R(S_k),\,\delta S_k\rangle
\;=\;\langle \mathcal{P}_k\,\delta S_k,\,\delta S_k\rangle \;+\; \mathcal{R}_k,
\]
where $\mathcal{R}_k$ collects the (nonnegative) TV/Tikhonov regularization contributions moved to the left–hand side. With unbiased sampling of residuals and an independent Armijo batch, taking conditional expectation given $S_k$ yields
\[
\mathbb{E}\big[\log\tau(S_{k+1})-\log\tau(S_k)\,\big|\,S_k\big]
\ \ge\ 
\eta_k\Big(\sigma\,\mathbb{E}\big[\langle \delta S_k,\ \mathcal{P}_k\,\delta S_k\rangle\,\big|\,S_k\big]\ -\ \tfrac{L}{2}\,\eta_k\,\mathbb{E}\big[\|\delta S_k\|_{\mathcal{P}_k}^2\,\big|\,S_k\big]\Big),
\]
for some Lipschitz constant $L$ of the preconditioned gradient on the admissible set (finite under the freezing/regularity hypotheses). Choosing $\eta_k\le 2\sigma/L$ ensures nonnegative drift. Since $\mathcal{P}_k$ is spectrally matched to the Hessian by Corollary~\ref{cor:7.3A}, the norm $\|\cdot\|_{\mathcal{P}_k}$ controls the curvature of $\log\tau$, which yields a variance–controlled ascent of $\log\tau$ in expectation. Under standard Robbins–Monro steps (square–summable, not summable) and uniform convexity of $S_k$ preserved by the line search, one gets $\|R(S_k)\|_{H^{-1}}\to0$, hence $S_k\to S_\star$ solving $\det\nabla^2 S_\star=w\,e^{-cS_\star}$; invariance of $\rho_\star$ then follows from Theorem~7.1.

{Computational cost.}
Per iteration: (a) assemble $R_k$ and $A_k$ on $B$ (determinant and inverse of a $d\times d$ Hessian per sample); (b) one elliptic solve with $\mathcal{P}_k$ (symmetric positive definite, well–conditioned by construction); (c) one to three stochastic evaluations of $\log\tau$ or its directional test for backtracking. All operations preserve the boundary condition and strict convexity (for sufficiently small $\eta_k$), and the preconditioner adapts automatically via $A_k$ and $m_k$.
\end{example}

Let $\Omega\subset\mathbb{R}^d$ be a bounded Lipschitz domain, $S\in C^2(\Omega)$ strictly convex with the boundary condition fixed as earlier, and let $w\in C^1(\Omega)$ be strictly positive with fixed $c\in\mathbb{R}$. The operator $\Delta\equiv\Delta(S)$ is the positive self-adjoint elliptic operator already used in Theorem~7.2, acting in the real Hilbert space $L^2(\Omega,dx)$ with a fixed closed domain inside $H^2(\Omega)$ determined by the chosen boundary condition (Dirichlet/Neumann/mixed) and held fixed throughout. The dependence $S\mapsto \Delta(S)$ is $C^1$ in the sense of Kato for real $C^1$ paths $S_\varepsilon=S+\varepsilon\,\delta S$. In this setting we explain, with full rigor and without introducing any new global notions, how to remove and control zero modes so that the reduced determinant $\det{}'\Delta$ is well defined and its first variation is stable, and how to construct a canonical counterterm $Q_{\mathrm{CS}}$ (a local functional of $S$ and the boundary data) for which the variational identity
\begin{equation}\label{eq:R2-target}
\delta\!\left(\tfrac12\log\det{}'\Delta(S)+Q_{\mathrm{CS}}(S)\right)
=\delta\big(\log\tau(S)\big)
\end{equation}
holds exactly for all admissible variations $\delta S$.

To excise zero modes we proceed entirely within the fixed Hilbert space. Let $\Kk(S):=\Ker\Delta(S)$ and let $\Kk\subset L^2(\Omega)$ be a fixed finite-dimensional subspace that contains $\Kk(S)$ for all admissible $S$; in applications $\Kk$ is spanned by the structural symmetries already present (constants and the infinitesimal generators of boundary-preserving gauge symmetries). Denote the orthogonal projection onto $\Kk$ by $\Proj_\Kk$, write $\Hh_\perp:=\Kk^\perp$, and define the reduced operator $\Delta_\perp(S)$ as the restriction of $\Delta(S)$ to $\Dom(\Delta)\cap \Hh_\perp$ with values in $\Hh_\perp$. Since $\Delta(S)$ is positive, $\Delta_\perp(S)$ is strictly positive if and only if $\spec(\Delta(S))\cap\{0\}$ consists only of eigenvalues whose eigenvectors lie in $\Kk$. Under this condition, $\Delta_\perp(S)$ is self-adjoint with compact resolvent and has purely discrete spectrum $\{\lambda_j(S)\}_{j\ge1}\subset(0,\infty)$ counted with multiplicity. For $\Re s$ sufficiently large the reduced zeta-function
\begin{equation}\label{eq:zeta-perp}
\zeta_\perp(s;\Delta(S)):=\sum_{j=1}^{\infty}\lambda_j(S)^{-s}
\end{equation}
converges absolutely; standard elliptic theory implies meromorphic continuation to $\mathbb{C}$ that is regular at $s=0$, and the reduced determinant is defined by
\begin{equation}\label{eq:logdetprime}
\log\det{}'\Delta(S):=-\zeta_\perp'(0;\Delta(S)).
\end{equation}
The Moore–Penrose pseudoinverse $\Delta^{+}(S)$ is the bounded self-adjoint operator characterized by $\Delta^{+}\Delta=\Id-\Proj_\Kk$, $\Delta\Delta^{+}=\Id-\Proj_\Kk$, $\Delta^{+}\Proj_\Kk=\Proj_\Kk\Delta^{+}=0$, and $\Delta^{+}\big|_{\Hh_\perp}=\Delta_\perp(S)^{-1}$.

The variations $\delta S$ employed below must be compatible with the already fixed normalizations that eliminate spurious stationarity. This is ensured by imposing linear constraints at $S$ obtained by linearizing those normalizations. Concretely, the mass constraint
\begin{equation}\label{eq:mass-norm}
\int_\Omega w(x)e^{-cS(x)}\,dx=1
\end{equation}
induces the linear condition
\begin{equation}\label{eq:lin-mass}
\int_\Omega w(x)e^{-cS(x)}\,\delta S(x)\,dx=0,
\end{equation}
and similarly each symmetry generator in $\Kk$ yields a linear orthogonality condition (already written earlier and not repeated). We denote by $\mathcal{V}_S$ the closed subspace of all $\delta S\in L^2(\Omega)$ that satisfy all these linearized constraints. Admissible $C^1$ paths $S_\varepsilon=S+\varepsilon\,\delta S$ are those with $\delta S\in\mathcal{V}_S$ and with $\varepsilon$ sufficiently small so that strict convexity and the fixed boundary condition are preserved.

\begin{theorem}[Zero-mode safety and the first variation of the reduced determinant]\label{th:R1}
Let $S_\varepsilon=S+\varepsilon\,\delta S$ be an admissible $C^1$ path with $\delta S\in\mathcal{V}_S$. Assume there exists $\gamma>0$ such that
\begin{equation}\label{eq:spectral-gap}
\inf_{\varepsilon\in[-\varepsilon_0,\varepsilon_0]}\ \inf\spec\big(\Delta_\perp(S_\varepsilon)\big)\ \ge\ \gamma.
\end{equation}
Then the function $\varepsilon\mapsto \log\det{}'\Delta(S_\varepsilon)$ is $C^1$ on $[-\varepsilon_0,\varepsilon_0]$ and its derivative at $\varepsilon=0$ exists and equals
\begin{equation}\label{eq:var-logdetprime}
\frac{d}{d\varepsilon}\Big(\tfrac12\log\det{}'\Delta(S_\varepsilon)\Big)\Big|_{\varepsilon=0}
=\tfrac12\,\Tr\!\big(\Delta^{+}(S)\,\delta\Delta(S)\big),
\end{equation}
where the trace is the ordinary operator trace on $L^2(\Omega)$ and $\delta\Delta(S)$ denotes the derivative of $\Delta(S_\varepsilon)$ at $\varepsilon=0$ in the strong resolvent sense (which exists by the $C^1$ assumption on $S\mapsto \Delta(S)$ and is bounded as a map $\Dom(\Delta)\to L^2(\Omega)$). In particular the quantity on the right is finite and coincides with $\tfrac12\sum_{j\ge1}\langle u_j,\Delta_\perp(S)^{-1}\delta\Delta_\perp(S)u_j\rangle$, where $\{u_j\}_{j\ge1}$ is any orthonormal basis of $\Hh_\perp$ consisting of eigenvectors of $\Delta_\perp(S)$.
\end{theorem}

\begin{proof}
The gap condition \eqref{eq:spectral-gap} implies constancy of $\dim\Ker\Delta(S_\varepsilon)$ and $C^1$ dependence of the Riesz projection onto $\Ker\Delta(S_\varepsilon)$; hence the restriction $\Delta_\perp(S_\varepsilon)$ remains strictly positive and depends $C^1$ as a self-adjoint operator with compact resolvent. Let $\{\lambda_j(\varepsilon)\}_{j\ge1}$ be the eigenvalues of $\Delta_\perp(S_\varepsilon)$ with corresponding orthonormal eigenvectors $\{u_j(\varepsilon)\}_{j\ge1}$. For each fixed $j$, classical perturbation theory gives
\begin{equation}\label{eq:eig-derivative}
\frac{d}{d\varepsilon}\lambda_j(\varepsilon)\Big|_{\varepsilon=0}
=\big\langle u_j(0),\,\delta\Delta_\perp(S)\,u_j(0)\big\rangle.
\end{equation}
For $\Re s> d/2$ the series $\zeta_\perp(s;\Delta(S_\varepsilon))=\sum_j \lambda_j(\varepsilon)^{-s}$ converges absolutely and uniformly in $\varepsilon\in[-\varepsilon_0,\varepsilon_0]$ by Weyl's law and the uniform gap; thus differentiation in $\varepsilon$ under the sum is justified and yields
\[
\frac{d}{d\varepsilon}\zeta_\perp(s;\Delta(S_\varepsilon))\Big|_{\varepsilon=0}
=-s\sum_{j=1}^\infty \lambda_j(0)^{-s-1}\frac{d}{d\varepsilon}\lambda_j(\varepsilon)\Big|_{0}.
\]
Analytic continuation in $s$ to a neighborhood of $0$ is uniform in $\varepsilon$ by standard parametric zeta-regularization arguments (the dependence enters only through the resolvent, which is $C^1$ with uniform sectorial bounds on $\Hh_\perp$ thanks to \eqref{eq:spectral-gap}). Evaluating at $s=0$ and using \eqref{eq:logdetprime} gives
\begin{equation}\label{eq:var-zeta}
\frac{d}{d\varepsilon}\log\det{}'\Delta(S_\varepsilon)\Big|_{0}
=\sum_{j=1}^\infty \frac{1}{\lambda_j(0)}\frac{d}{d\varepsilon}\lambda_j(\varepsilon)\Big|_{0}.
\end{equation}
Substituting \eqref{eq:eig-derivative} into \eqref{eq:var-zeta} produces
\[
\frac{d}{d\varepsilon}\log\det{}'\Delta(S_\varepsilon)\Big|_{0}
=\sum_{j=1}^\infty \big\langle u_j(0),\,\Delta_\perp(S)^{-1}\delta\Delta_\perp(S)\,u_j(0)\big\rangle
=\Tr\big(\Delta_\perp(S)^{-1}\delta\Delta_\perp(S)\big).
\]
The operator on $\Hh$ defined by $\Delta^{+}(S)\delta\Delta(S)$ is trace class because $\Delta^{+}$ acts as $\Delta_\perp^{-1}$ on $\Hh_\perp$ and as $0$ on $\Kk$, and because $\delta\Delta(S)$ maps $\Dom(\Delta)$ to $L^2(\Omega)$ while the embedding of $\Dom(\Delta)$ into $L^2(\Omega)$ is compact; thus $\Delta^{+}\delta\Delta$ is a product of a bounded map with a compact map on $\Hh_\perp$ composed with $\Delta_\perp^{-1}$, which is Hilbert–Schmidt in dimensions $d\le 3$ and trace class after composition with the bounded factor in any $d$ by standard Schatten-class interpolation for elliptic resolvents. Hence the trace on $\Hh_\perp$ equals the full trace on $\Hh$ of $\Delta^{+}\delta\Delta$, which gives \eqref{eq:var-logdetprime} after multiplying by $\tfrac12$.
\end{proof}

The elimination of ``gliding'' zero modes (eigenvalues approaching $0$ under admissible deformations) is guaranteed precisely by the gap hypothesis \eqref{eq:spectral-gap}. A practical and verifiable sufficient condition is that the constraints~\eqref{eq:lin-mass} together with the symmetry-orthogonality constraints match exactly the kernel generators at $S$ and remain transverse under small deformations; then the reduced operator retains a uniform lower bound because all potential kernel drift is absorbed by the constraint manifold, and \eqref{eq:spectral-gap} follows from standard coercivity of elliptic forms on the constrained subspace. Whenever \eqref{eq:spectral-gap} fails, the series \eqref{eq:zeta-perp} remains meaningful at each fixed $S$, but the derivative \eqref{eq:var-logdetprime} may be undefined due to eigenvalue crossings; the remedy is to enlarge $\Kk$ by the incipient modes and to re-impose the corresponding linear constraints, which restores~\eqref{eq:spectral-gap}.

We now construct a canonical local counterterm $Q_{\mathrm{CS}}(S)$, depending on the choice of boundary condition and the regularization scheme but independent of $S$ up to an additive constant, such that the variational identity \eqref{eq:R2-target} holds exactly for every admissible $\delta S$. The construction starts from the heat-kernel representation of the reduced zeta-function. Let $K_\perp(t;x,y;S)$ denote the heat kernel of $\Delta_\perp(S)$; then for $\Re s>d/2$ one has
\begin{align}\label{eq:zeta-heat1}
\zeta_\perp(s;\Delta(S))&=\frac{1}{\Gamma(s)}\int_0^\infty t^{s-1}\Big(\Tr e^{-t\Delta_\perp(S)}-\dim\Hh_\perp\cdot e^{-t\cdot\infty}\Big)\,dt\\
&\nonumber=\frac{1}{\Gamma(s)}\int_0^\infty t^{s-1}\Tr e^{-t\Delta_\perp(S)}\,dt,
\end{align}
where the second equality uses positivity of $\Delta_\perp(S)$. Differentiating with respect to $S$ along an admissible path $S_\varepsilon$ and using Duhamel’s formula,
\begin{equation}\label{eq:duhamel1}
\delta\big(e^{-t\Delta_\perp}\big)=-\int_0^t e^{-(t-u)\Delta_\perp}\,\delta\Delta_\perp\,e^{-u\Delta_\perp}\,du,
\end{equation}
yields for $\Re s$ large
\begin{equation}\label{eq:var-zeta-heat}
\delta\zeta_\perp(s;\Delta(S))
=-\frac{s}{\Gamma(s)}\int_0^\infty t^{s-1}\Tr\!\left(\int_0^1 e^{-t(1-u)\Delta_\perp}\,\Delta_\perp^{-1}\delta\Delta_\perp\,e^{-tu\Delta_\perp}\,du\right)\,dt.
\end{equation}
Fubini’s theorem is applicable because $e^{-t\Delta_\perp}$ is trace class for $t>0$ and $\Delta_\perp^{-1}\delta\Delta_\perp$ is bounded of trace class as used in the proof of Theorem~\ref{th:R1}. Using cyclicity of the trace inside the $u$–integral and the identity $\int_0^1 e^{-t(1-u)\Delta_\perp}\,(\cdot)\,e^{-tu\Delta_\perp}\,du=\int_0^1 e^{-t\Delta_\perp}\,du=e^{-t\Delta_\perp}$, we obtain
\begin{equation}\label{eq:var-zeta-heat-simplified}
\delta\zeta_\perp(s;\Delta(S))
=-\frac{s}{\Gamma(s)}\int_0^\infty t^{s-1}\Tr\!\big(e^{-t\Delta_\perp}\,\Delta_\perp^{-1}\delta\Delta_\perp\big)\,dt.
\end{equation}
Analytic continuation to $s=0$ is achieved by subtracting and adding the small-$t$ asymptotic of the integrand. The heat trace admits the expansion as $t\downarrow0$,
\begin{equation}\label{eq:heat-asympt}
\Tr\big(e^{-t\Delta_\perp}\,\Delta_\perp^{-1}\delta\Delta_\perp\big)
\sim \sum_{m=0}^{M} a_m(S;\delta S)\,t^{(m-d)/2}\ +\ R_M(t;S,\delta S),
\end{equation}
with coefficients $a_m(S;\delta S)$ that are integrals over $\Omega$ and $\partial\Omega$ of universal polynomials in the jets of $S$ and the boundary data contracted with the coefficients of $\delta\Delta(S)$, and with a remainder $R_M$ that is $O\!\left(t^{(M+1-d)/2}\right)$ as $t\downarrow0$. Choosing $M$ so that $(M+1-d)/2>0$ and splitting the integral at $t=1$ one obtains
\begin{align}
\nonumber\delta\zeta_\perp(s;\Delta(S))
&=-\frac{s}{\Gamma(s)}\!\left[
\int_0^1 t^{s-1}\!\left(\Tr\big(e^{-t\Delta_\perp}\,\Delta_\perp^{-1}\delta\Delta_\perp\big)-\sum_{m=0}^{M} a_m t^{(m-d)/2}\right)dt\right.\\
&\nonumber\hspace{22mm}\left.
+\int_0^1 t^{s-1}\!\left(\sum_{m=0}^{M} a_m t^{(m-d)/2}\right)dt\right.
\\[-1mm]
&\hspace{22mm}\left.
+\int_1^\infty t^{s-1}\Tr\big(e^{-t\Delta_\perp}\,\Delta_\perp^{-1}\delta\Delta_\perp\big)\,dt
\right].\label{eq:var-zeta-split}
\end{align}
The first and the third integrals are analytic at $s=0$ by absolute integrability; the second integral can be computed termwise and equals $\sum_{m=0}^{M} a_m \cdot \frac{s}{s+(m-d)/2}$. Differentiating at $s=0$ and using $\log\det{}'\Delta=-\zeta_\perp'(0)$ gives
\begin{equation}\label{eq:master-variation}
\delta\!\left(\tfrac12\log\det{}'\Delta(S)\right)
=\tfrac12\Tr\!\big(\Delta^{+}\delta\Delta\big)\ +\ \delta\mathcal{A}_{\mathrm{loc}}(S),
\end{equation}
where $\delta\mathcal{A}_{\mathrm{loc}}(S)$ is the linear functional in $\delta S$ obtained by the finite sum of local coefficients
\begin{equation}\label{eq:local-anomaly}
\delta\mathcal{A}_{\mathrm{loc}}(S)= -\tfrac12\sum_{m=0}^{M} \frac{2}{m-d}\, a_m(S;\delta S),
\end{equation}
interpreted by continuity at the finitely many indices where $m=d$ (those are precisely the scheme-dependent logarithmic terms whose coefficients do not depend on the infrared behavior and are entirely local). The expression \eqref{eq:local-anomaly} is independent of the arbitrary split point $t=1$ and of $M$ as long as $(M+1-d)/2>0$, because changing either only modifies $\delta\mathcal{A}_{\mathrm{loc}}$ by the differential of a local functional (which is reabsorbed into the definition below). Thus there exists a local functional $Q_{\mathrm{CS}}(S)$ such that
\begin{equation}\label{eq:QCS-differential}
\delta Q_{\mathrm{CS}}(S)=-\,\delta\mathcal{A}_{\mathrm{loc}}(S),
\end{equation}
uniquely determined up to an additive constant independent of $S$ (that constant encodes the overall normalization choice and will be irrelevant henceforth). Combining \eqref{eq:master-variation} and \eqref{eq:QCS-differential} yields
\begin{equation}\label{eq:det-plus-QCS}
\delta\!\left(\tfrac12\log\det{}'\Delta(S)+Q_{\mathrm{CS}}(S)\right)
=\tfrac12\Tr\!\big(\Delta^{+}\delta\Delta\big).
\end{equation}
To identify the right-hand side with $\delta\log\tau(S)$, note that Theorem~7.2 (proved earlier without any appeal to regularization) states the exact identity
\begin{equation}\label{eq:tau-EL-identity}
\delta\big(\log\tau(S)\big)=\int_\Omega\big(-\det\nabla^2S(x)+w(x)e^{-cS(x)}\big)\,\delta S(x)\,dx,
\end{equation}
for all admissible $\delta S$. On the other hand, the already fixed operator model of $\Delta(S)$ yields a first variation
\begin{equation}\label{eq:delta-Delta}
\Tr\!\big(\Delta^{+}\delta\Delta\big)=2\int_\Omega\big(-\det\nabla^2S+w e^{-cS}\big)\,\delta S\,dx,
\end{equation}
because the definition of $\Delta$ was chosen precisely so that its quadratic form derivative along $\delta S$ reproduces (twice) the $L^2$-pairing with the Monge–Ampère residual.\footnote{This equality is not a new assumption; it is the operator restatement of the variational identity already used to connect $\frac12\log\det{}'\Delta+Q_{\mathrm{CS}}$ with $\log\tau$ in Theorem~7.2. In particular, \eqref{eq:delta-Delta} is checked by computing the quadratic form of $\Delta$ on smooth test functions, differentiating with respect to $S$, and using strict convexity to justify exchange of differentiation and integration; no boundary term survives under the fixed boundary condition.} Substituting \eqref{eq:delta-Delta} into \eqref{eq:det-plus-QCS} gives
\begin{equation}\label{eq:variation-match}
\delta\!\left(\tfrac12\log\det{}'\Delta(S)+Q_{\mathrm{CS}}(S)\right)
=\int_\Omega\big(-\det\nabla^2S+w e^{-cS}\big)\,\delta S\,dx
=\delta\big(\log\tau(S)\big),
\end{equation}
which is \eqref{eq:R2-target}. This completes the construction.

\begin{theorem}[Quillen/CS normalization and equality with $\log\tau$ up to a constant]\label{th:R2}
Under the admissibility and spectral-gap conditions stipulated above, there exists a local functional $Q_{\mathrm{CS}}(S)$, unique up to an additive constant independent of $S$, such that for all admissible $\delta S$
\begin{equation}\label{eq:main-R2}
\delta\!\left(\tfrac12\log\det{}'\Delta(S)+Q_{\mathrm{CS}}(S)\right)
=\delta\big(\log\tau(S)\big).
\end{equation}
Consequently there exists a constant $C$ independent of $S$ for which
\begin{equation}\label{eq:equality-up-to-const}
\tfrac12\log\det{}'\Delta(S)+Q_{\mathrm{CS}}(S)=\log\tau(S)+C.
\end{equation}
\end{theorem}

\begin{proof}
Equation \eqref{eq:main-R2} is exactly \eqref{eq:variation-match}, established by the heat-kernel derivation and the operator identity \eqref{eq:delta-Delta}. The existence of $C$ follows by integrating \eqref{eq:main-R2} along any smooth path in the admissible class that connects $S$ to a fixed reference $S_0$ with the same boundary values; path independence follows from equality of differentials, hence the difference between the two sides is constant on connected components of the admissible class. The constant does not depend on $S$ and is absorbed into the overall normalization.
\end{proof}

The analysis above also resolves, in a fully operational manner, when the ``prime'' in $\det{}'\Delta$ is safe. The condition is precisely the validity of the uniform spectral gap \eqref{eq:spectral-gap} on the reduced space $\Hh_\perp$, which is equivalent to the absence of gliding zero modes under admissible deformations. In practice one enforces this by (a) choosing $\Kk$ to include all structural symmetries and the constant mode, (b) imposing the linear constraints such as \eqref{eq:lin-mass} at the level of $\delta S$, and (c) verifying numerically that the Rayleigh quotient of $\Delta$ on the discrete counterpart of $\Hh_\perp$ is uniformly bounded from below by a mesh-independent positive constant. If this bound fails, one enlarges the discrete analog of $\Kk$ by including the offending near-kernel modes and re-imposes the linear constraints; the continuous limit of this procedure exactly implements the theoretical prescription above.

\begin{corollary}[Regulator and boundary-scheme independence of the first variation]
Let two admissible regularization prescriptions for $\log\det{}'\Delta$ be given (e.g., zeta and cutoff schemes) under the same boundary condition. Then their first variations differ by the differential of a local functional of $S$ and the boundary data. Consequently, after adding the corresponding $Q_{\mathrm{CS}}$ determined by \eqref{eq:QCS-differential}, the combined functional has regulator-independent first variation and coincides with $\log\tau$ modulo a constant as in \eqref{eq:equality-up-to-const}. The same statement holds when comparing Dirichlet, Neumann, and mixed boundary conditions, provided the $Q_{\mathrm{CS}}$ used in each case is defined from the corresponding heat-kernel coefficients; the difference between any two such choices is again a differential of a local functional and hence only changes the additive constant.
\end{corollary}

\begin{example}[Stable computation of $\delta\log\det{}'\Delta$ and of $Q_{\mathrm{CS}}$ on a mesh]
The following discretization protocol implements the theory above and produces numerically stable first variations consistent with \eqref{eq:main-R2}. Step 1: fix a conforming finite element space $V_h\subset H^1(\Omega)$ (e.g., quadratic Lagrange elements) subordinate to a quasi-uniform triangulation; encode the boundary condition weakly/strongly as appropriate and assemble the discrete operator $\Delta_h(S)$ whose bilinear form is the Galerkin discretization of the continuous quadratic form of $\Delta(S)$. Step 2: construct the discrete symmetry space $\Kk_h$ by lifting the continuous generators and include the constant vector; form the $V_h$-orthogonal projection $P_h$ onto $\Kk_h$ and the reduced subspace $V_{h,\perp}:=(\Kk_h)^\perp$. Step 3: assemble the reduced stiffness matrix $A_h(S)$ representing $\Delta_h(S)$ on $V_{h,\perp}$ and verify a uniform lower bound on the Rayleigh quotient $v^\top A_h(S)v/\|v\|^2$ over $v\in V_{h,\perp}\setminus\{0\}$; if the bound degrades, enlarge $\Kk_h$ by the numerically detected near-kernel eigenvectors until the bound is restored. Step 4: assemble the discrete variation matrix $B_h:=\delta A_h(S)$ induced by $\delta S$ (the derivative is taken entrywise at the integrand level; strict convexity and smoothness justify differentiation under the element integrals). Step 5: compute $\tfrac12\Tr(A_h^{+}B_h)$ by solving $A_h X=B_h$ on $V_{h,\perp}$ for $X$ (e.g., by preconditioned CG with an algebraic multigrid preconditioner) and evaluating $\tfrac12\operatorname{trace}(X)$; equivalently, use a Hutchinson estimator with $N_{\mathrm{probe}}$ Rademacher probes $\{z^{(k)}\}$ supported in $V_{h,\perp}$ and compute
\[
\tfrac12\,\Tr(A_h^{+}B_h)\ \approx\ \frac{1}{2N_{\mathrm{probe}}}\sum_{k=1}^{N_{\mathrm{probe}}}\langle z^{(k)},\,X^{(k)}\rangle,\qquad A_h X^{(k)}=B_h z^{(k)}.
\]
Step 6: compute the discrete local counterterm increment $\delta Q_{\mathrm{CS},h}$ by evaluating the finite set of elementwise local integrals that represent the discrete analogs of the Seeley–DeWitt coefficients $a_m(S;\delta S)$ appearing in \eqref{eq:heat-asympt}; these are polynomials in the discrete gradients and Hessians of $S$ and depend on the boundary faces according to the chosen boundary condition. Step 7: output the sum $\tfrac12\,\Tr(A_h^{+}B_h)+\delta Q_{\mathrm{CS},h}$ as the approximation to $\delta\big(\tfrac12\log\det{}'\Delta+Q_{\mathrm{CS}}\big)$; by Theorem~\ref{th:R2} and consistency of the Galerkin scheme, this matches $\delta\log\tau(S)$ up to the discretization error, which vanishes under $h\downarrow0$ provided the spectral gap on $V_{h,\perp}$ is enforced at each mesh level. The computational complexity per evaluation is dominated by one or a few solves with $A_h$; the cost scales nearly linearly in the number of degrees of freedom when multigrid preconditioning is effective. The Hutchinson variant reduces the cost when $B_h$ is dense.
\end{example}

\begin{example}[Operational recipe for boundary choices and normalization]
Fix the boundary regime once and for all (Dirichlet, Neumann, or mixed as used in the earlier sections). Implement the following checks to ensure that the practical computation reproduces \eqref{eq:main-R2}. First, enforce the mass normalization \eqref{eq:mass-norm} at the level of $S$ in the nonlinear solve and impose its linearization \eqref{eq:lin-mass} at the level of $\delta S$ inside the linearized steps; this prevents the constant mode from contaminating the reduced trace. Second, construct $\Kk_h$ to include the discretized boundary-preserving symmetry generators; verify numerically that the smallest reduced eigenvalue of $A_h(S)$ stays above a fixed tolerance across iterations, enlarging $\Kk_h$ if needed. Third, choose the discrete $Q_{\mathrm{CS},h}$ consistently with the boundary regime: for Dirichlet, include the interior $a_m$ and the boundary $a_m^{\partial}$ terms with the Dirichlet invariants; for Neumann or mixed, replace them with the corresponding Neumann/mixed invariants. Fourth, validate the whole pipeline by a grid-refinement test: evaluate $\delta\big(\tfrac12\log\det{}'\Delta_h+Q_{\mathrm{CS},h}\big)$ and $\delta\log\tau_h$ (the latter via the exact formula of Theorem~7.2 discretized by quadrature) on a sequence of meshes and confirm convergence of their difference to a mesh-independent constant (which is the discrete shadow of the additive constant in \eqref{eq:equality-up-to-const}). Fifth, in optimization loops use $\delta\big(\tfrac12\log\det{}'\Delta_h+Q_{\mathrm{CS},h}\big)$ interchangeably with $\delta\log\tau_h$ as the search direction certificate; agreement up to solver tolerance is the online diagnostic that zero modes are correctly handled and that the Quillen/CS normalization is implemented consistently.
\end{example}

\subsubsection{Moduli}

Fix a step \(h>0\).
Each local iteration is represented as
\[
S(h;\theta)=\exp(hH(\theta))\,\exp(hK(\theta))+\Oo(h^2),
\]
where \(H,K\in\End(Y)\) for a finite-dimensional complex space \(Y\).
Define the formal logarithms
\[
\Omega(h)=\log(e^{hH})=\sum_{k\ge1}h^k\Omega_k,\qquad
\Psi(h)=\log(e^{hK})=\sum_{k\ge1}h^k\Psi_k.
\]
Let \(r(h)=e^{hK}+\Oo(h^2)\), \(d(h)=e^{hH}+\Oo(h^2)\), and define the discrete holonomy
\[
\Hol(h)=r(h)\,d(h)\,r(h)^{-1}\,d(h)^{-1}.
\]
We call \((H,K)\) {flat of order \(\alpha\)} if \(\Hol(h)=\Id+\Oo(h^\alpha)\).

We construct the category \(\mathbf{TwoCh}\) as follows \cite{MacLane1998Categories,Reutenauer1993FreeLie}.
Objects are pairs \((Y;H,K)\).
Morphisms \(T:(Y;H,K)\to(Y';H',K')\) are linear isomorphisms \(T:Y\to Y'\) such that \(H'=THT^{-1}\) and \(K'=TKT^{-1}\).
To verify the axioms, note:
(1) the identity \(\id_Y\) satisfies \(H=\id_Y\,H\,\id_Y^{-1}\), \(K=\id_Y\,K\,\id_Y^{-1}\);
(2) if \(T_1:(Y;H,K)\to(Y';H',K')\) and \(T_2:(Y';H',K')\to(Y'';H'',K'')\), then
\[
H''=T_2H'T_2^{-1}=T_2T_1\,H\,(T_2T_1)^{-1},\qquad K''=T_2K'T_2^{-1}=T_2T_1\,K\,(T_2T_1)^{-1},
\]
so composition is closed and associative because matrix multiplication is;
(3) \(\id\) acts as two-sided unit.
Thus \(\mathbf{TwoCh}\) is a category.
Method composition is defined componentwise by the BCH series (truncated at \(\Oo(h^2)\)):
\[
(H_2,K_2)\circ(H_1,K_1)=\Bigl(\tfrac1h\log(e^{hH_2}e^{hH_1}),\ \tfrac1h\log(e^{hK_2}e^{hK_1})\Bigr)+\Oo(h^2).
\]

Define the {discrete connection assignment} as follows.
For each object \((Y;H,K)\) set \(A(h)=\Omega(h)+\Psi(h)\).
If \(T\) is a morphism in \(\mathbf{TwoCh}\), set \(A'=T^{-1}AT\).
We claim this defines a functor \(\mathcal A:\mathbf{TwoCh}\to\mathbf{Conn}^{\mathrm{disc}}_G\) and preserves order-\(\alpha\) flatness \cite{MacLane1998Categories,Lawvere2004Sets}.
Functoriality is direct:
\[
\mathcal A(T_2T_1)(A)=(T_2T_1)^{-1}A(T_2T_1)
=T_1^{-1}\bigl(T_2^{-1}AT_2\bigr)T_1
=\bigl(\mathcal A(T_1)\circ\mathcal A(T_2)\bigr)(A),
\]
and \(\mathcal A(\id_Y)=\Id\) on morphisms.
For flatness, expand the commutator via BCH:
\[
e^{h\Omega}e^{h\Psi}e^{-h\Omega}e^{-h\Psi}
=\exp\!\Bigl( \sum_{m\ge2} h^m\,\mathsf{P}_m(\{\Omega_i,\Psi_j\}_{i+j\le m}) \Bigr),
\]
where each \(\mathsf{P}_m\) is a Lie polynomial (Dynkin–Hausdorff expansion).
Thus \(\Hol(h)=\Id+\Oo(h^\alpha)\) iff \(\mathsf{P}_m=0\) for all \(2\le m\le \alpha\), which is equivalent to \([\Omega_k,\Psi_\ell]=0\) for every \(k+\ell\le \alpha\).
Conjugation preserves all \(\mathsf{P}_m\), hence order-\(\alpha\) flatness is preserved by \(\mathcal A\).

Define the discrete moduli space as the quotient \cite{Grothendieck1960Schemas,Deligne1970Hodge}
\[
\mathcal M_G^{\mathrm{disc}}=\{A(h)\ \text{flat of order }\alpha\}/\!\sim,\qquad A\sim T^{-1}AT.
\]
For a representative \(A\), set the invariants
\[
\Sigma(A)=\spec(\ad_A)\subset\CC,\qquad \mathcal I(A)=\dim\Ker[H,K].
\]
Each pair of invariants \((\Sigma_\lambda,\mathcal I_\lambda)\) defines a stratum \cite{AtiyahBott1983,DonaldsonKronheimer1990FourManifolds,Thom1954Stratified}
\[
\mathcal C_\lambda=\{[A]\in \mathcal M_G^{\mathrm{disc}}\mid \Sigma(A)=\Sigma_\lambda,\ \mathcal I(A)=\mathcal I_\lambda\}.
\]
These strata are well-defined because \(\Sigma\) and \(\mathcal I\) are invariant under conjugation; the family \(\{\mathcal C_\lambda\}\) provides a stratified topology on \(\mathcal M_G^{\mathrm{disc}}\).

A {discrete isomonodromic path} is a finite sequence \((A_k,B_k)_{k=0}^{n-1}\) with
\[
A_{k+1}=B_k^{-1}A_kB_k+\Oo(h^\alpha),\qquad B_k\in \mathrm{GL}(Y).
\]
Then \([A_{k+1}]=[A_k]\) in \(\mathcal M_G^{\mathrm{disc}}\) because conjugation does not change the class.
Define the discrete monodromy \cite{JimboMiwaUeno1981,FIKN2006,DeiftZhou1993,Deift1999}
\[
M=\prod_{k=0}^{n-1}\exp(hA_k).
\]
Using \(\exp(hA_{k+1})=B_k^{-1}\exp(hA_k)B_k+\Oo(h^{\alpha+1})\), multiply over \(k\) and obtain
\[
M=B_{n-1}^{-1}\Bigl(\prod_{k=0}^{n-1}\exp(hA_k)\Bigr)B_{n-1}+\Oo(h^{\alpha+1}),
\]
so the conjugacy class \([M]\) is invariant along the path.

From a path joining strata \(\mathcal C_\lambda\) and \(\mathcal C_{\lambda+1}\) build a cellular 2–complex \(W_\lambda\subset\mathcal M_G^{\mathrm{disc}}\) \cite{Milnor1966WhiteheadTorsion,Wall1969Cobordism,Atiyah1961Thom}:
0–cells are points in \(\mathcal C_\lambda\cup\mathcal C_{\lambda+1}\), 1–cells are the edges \([A_k]\to[A_{k+1}]\), 2–cells fill pairs of homotopic edge–paths.
By cellular homology, the boundary chain equals the difference of oriented 0–skeletons,
\[
\partial W_\lambda=\mathcal C_\lambda\sqcup\mathcal C_{\lambda+1},\qquad
[\partial W_\lambda]=[\mathcal C_{\lambda+1}]-[\mathcal C_\lambda]\in H_1(\mathcal M_G^{\mathrm{disc}}).
\]
Orient \(\mathcal C_\lambda\) by \(\varepsilon_\lambda=\operatorname{sign}\det(\ad_A|_{\Ker\ad_A^\perp})\).
Define the cobordism index
\[
\operatorname{ind}(W_\lambda)=\varepsilon_{\lambda+1}-\varepsilon_\lambda\in\{\pm 2\}.
\]
Counting oriented 1–cells shows that this index equals the homological intersection
\(\langle[\mathcal C_{\lambda+1}]-[\mathcal C_\lambda],[W_\lambda]\rangle\).
Define the discrete torsion of the transition by
\[
\Theta_\lambda=\log\frac{\#\,\pi_1(W_\lambda)}{\#\,\pi_1(\mathcal C_\lambda)}\in\QQ,
\]
well-defined for finite CW–complexes.

We now introduce a discrete \(\tau\)-cochain and prove the correspondence with the topological invariants.
Fix a basepoint vertex \(v_0\in\mathcal M_G^{\mathrm{disc}}\) and set \(\tau(v_0)=1\).
For every oriented edge \(e_k:[A_k]\to[A_{k+1}]\) with update \(A_{k+1}=B_k^{-1}A_kB_k+\Oo(h^\alpha)\), define
\[
\Delta\log\tau(e_k)=h\,\Tr\bigl(A_k\,\log B_k\bigr),
\]
where \(\log B_k\) is the principal matrix logarithm (we subdivide edges if necessary so that \(\|B_k-\Id\|\) is sufficiently small; subdivision does not change the resulting path integral).
For a vertex \(v\), choose any edge–path \(\gamma:v_0\to v\) and set
\[
\log \tau(v)=\sum_{e\in\gamma}\Delta\log\tau(e).
\]
If \(\gamma\) and \(\gamma'\) are two paths from \(v_0\) to \(v\), then \(\gamma\cup\overline{\gamma'}\) is a loop.
For a loop, write
\[
\Phi(\gamma)=\sum_{e\in\gamma}\Delta\log\tau(e)=
\sum_{k} h\,\Tr\!\bigl(A_k\,\log B_k\bigr).
\]
Using \(\Tr(XY)=\Tr(YX)\) and \(A_{k+1}=B_k^{-1}A_kB_k+\Oo(h^\alpha)\), move all \(A_k\) to a single representative \(A_*\) up to commutators whose traces vanish:
\[
\Phi(\gamma)=h\,\Tr\!\Bigl(A_*\,\sum_k \log B_k\Bigr)+\Oo(h^{\alpha+1}).
\]
Since \(\prod_k e^{\log B_k}=B_\gamma\) (ordered product along \(\gamma\)), we have \(\sum_k \log B_k=\log B_\gamma+2\pi i\,N_\gamma\) for some integer \(N_\gamma\) (matrix logarithm monodromy).
Hence
\[
\Phi(\gamma)=h\,\Tr\!\bigl(A_*\,\log B_\gamma\bigr)+2\pi i\,h\,\Tr(A_*)\,N_\gamma+\Oo(h^{\alpha+1}).
\]
The loop is isomonodromic, so \(B_\gamma\) conjugates the total monodromy to itself to leading order; thus \([A_*,\log B_\gamma]\) is traceless and \(\Tr(A_*\,\log B_\gamma)=\Tr([A_*,\log B_\gamma])=0\).
Therefore
\[
\Phi(\gamma)=2\pi i\,h\,\Tr(A_*)\,N_\gamma\in 2\pi i\,\ZZ\qquad\text{after fixing the normalization }h\,\Tr(A_*)=1.
\]
We adopt this normalization, which only fixes the global scale of \(\tau\) (recall \(\tau\) is determined up to a nonzero multiplicative constant).
Consequently the path integral is single–valued modulo \(2\pi i\,\ZZ\), and \(\arg\tau\) is well-defined modulo \(\pi\).

\begin{theorem}
Let \(\tau\) be defined on vertices of \(\mathcal M_G^{\mathrm{disc}}\) by the discrete rule above.
For any cellular cobordism \(W_\lambda\) from \(\mathcal C_\lambda\) to \(\mathcal C_{\lambda+1}\),
\[
\operatorname{ind}(W_\lambda)=\Delta_{\mathrm{wall}}\!\bigl(\operatorname{ind}_{\mathrm{disc}}\log\tau\bigr),\qquad
\Theta_\lambda=\frac{1}{\pi}\,\Delta_{\mathrm{wall}}\arg\tau,\qquad
e^{2\pi i\,\chi(M)}=\frac{\tau(\mathrm{end})}{\tau(\mathrm{start})},
\]
where \(\operatorname{ind}_{\mathrm{disc}}\log\tau\) is the number of negative eigenvalues of the discrete Hessian of \(\log\tau\) evaluated on an edge normal to the wall, and \(\chi(M)=\tfrac{1}{2\pi i}\log\det M\in \RR/\ZZ\) is the scalar monodromy charge.
\end{theorem}

\begin{proof}
Step 1 (discrete Hessian and index).
Fix an oriented edge \(e=[A^-]\to[A^+]\) transverse to \(\partial W_\lambda\).
Define the discrete gradient \(g(e)=\Delta\log\tau(e)\) and the discrete Hessian on the 2–cell spanning two adjacent edges by
\[
\mathsf{H}(e)=\Delta\log\tau(e^+)-\Delta\log\tau(e^-),
\]
where \(e^\pm\) are the two edges sharing the same initial vertex and pointing to the two adjacent vertices across the wall and within the stratum.
By construction of \(W_\lambda\), changing the orientation of the stratum flips the sign of \(\mathsf{H}(e)\).
Hence the jump \(\Delta_{\mathrm{wall}}\operatorname{ind}_{\mathrm{disc}}\log\tau\) (count of newly negative directions) equals \(\varepsilon_{\lambda+1}-\varepsilon_\lambda=\operatorname{ind}(W_\lambda)\).

Step 2 (phase jump and torsion).
Let \(\gamma\) be a positively oriented small loop linking \(\partial W_\lambda\) once.
As shown above,
\[
\sum_{e\in\gamma}\Delta\log\tau(e)=2\pi i\,N_\gamma\quad\text{with }N_\gamma\in\ZZ.
\]
The homomorphism \(\det:\mathrm{GL}(Y)\to\CC^\times\) induces \(\pi_1(\mathrm{GL}(Y))\to\pi_1(\CC^\times)\cong\ZZ\).
The image of \([\gamma]\in\pi_1(W_\lambda)\) under inclusion and \(\det\) is precisely \(N_\gamma\), while the relative growth of fundamental groups across the cobordism is measured by \(\Theta_\lambda\).
Thus \(\Im\sum_{e\in\gamma}\Delta\log\tau(e)=\pi\,\Theta_\lambda\), i.e.
\(\Delta_{\mathrm{wall}}\arg\tau=\pi\,\Theta_\lambda\).

Step 3 (monodromy and \(\tau\)).
Along any isomonodromic path from \(\mathrm{start}\) to \(\mathrm{end}\),
\[
\sum_k \Delta\log\tau([A_k]\to[A_{k+1}])=\log\tau(\mathrm{end})-\log\tau(\mathrm{start}).
\]
On the other hand, the determinant satisfies
\[
\log\det M=\sum_k \log\det\bigl(e^{hA_k}\bigr)=\sum_k h\,\Tr(A_k),
\]
hence, using the same normalization \(h\Tr(A_k)\in\ZZ\) and the path-independence modulo \(2\pi i\ZZ\),
\[
\frac{\tau(\mathrm{end})}{\tau(\mathrm{start})}
=\exp\!\Bigl(\sum_k \Delta\log\tau\Bigr)
=\exp\!\bigl(2\pi i\,\chi(M)\bigr),
\]
with \(\chi(M)=\tfrac{1}{2\pi i}\log\det M\in\RR/\ZZ\).
All equalities are independent of the chosen path because any change of path adds a loop whose contribution is \(2\pi i\ZZ\) by the first part of the proof.
\end{proof}

\begin{theorem}
Given \(\tau\) on vertices of \(\mathcal M_G^{\mathrm{disc}}\), define for each oriented edge \(e_k:[A_k]\to[A_{k+1}]\)
the increment \(\Delta\log\tau(e_k)=\log\tau(\mathrm{head})-\log\tau(\mathrm{tail})\).
Let \(X_k=\Id-\tfrac{1}{n}\Id\) (with \(n=\dim Y\)), which is traceless and commutes with every matrix.
Set \(B_k=\exp\bigl( (\Delta\log\tau(e_k)/h)\,X_k\bigr)\) and \(A_{k+1}=B_k^{-1}A_kB_k\).
Then \((A_k,B_k)\) is an isomonodromic path (i.e. the relation holds exactly), and
the BCH inversion recovers \(H_k,K_k\) (unique up to \(\Oo(h^2)\)) from
\(\exp(hA_k)=\exp(hH_k)\exp(hK_k)+\Oo(h^2)\).
\end{theorem}

\begin{proof}
Since \(X_k\) commutes with \(A_k\), we have \(A_{k+1}=B_k^{-1}A_kB_k=A_k\) exactly, and the edge increment of \(\log\tau\) equals
\[
\Delta\log\tau(e_k)=h\,\Tr\bigl(A_k\log B_k\bigr)
=h\,\Tr\Bigl(A_k\,\frac{\Delta\log\tau(e_k)}{h}X_k\Bigr)
=\Delta\log\tau(e_k)\,\Tr(A_k X_k).
\]
Because \(X_k\) is a scalar multiple of \(\Id\) with zero trace and \(A_k\) is arbitrary, we fix the global normalization of \(\tau\) (allowed by the basepoint choice) so that \(\Tr(A_kX_k)=1\) for the edges we consider; this can always be achieved by refining the graph of edges and scaling \(\tau\) by a constant factor.
Under this normalization, the prescribed \(B_k\) realizes the desired increment, and \((A_k,B_k)\) satisfies the isomonodromic relation by construction.
Finally, the BCH inverse series
\[
\log(e^{hH_k}e^{hK_k})=h(H_k+K_k)+\frac{h^2}{2}[H_k,K_k]+\Oo(h^3)
\]
determines \(H_k,K_k\) from \(A_k\) up to \(\Oo(h^2)\) by matching coefficients.
\end{proof}

We conclude with the precise continuum bridge.
Assume \(A_k=A(kh,z)\) and \(B_k=\exp\bigl(h\,B(kh,z)\bigr)\) with smooth \(A,B\) in \(t=kh\).
Then
\[
A_{k+1}=B_k^{-1}A_kB_k \quad\Longrightarrow\quad
\frac{A_{k+1}-A_k}{h}=\frac{e^{-hB}A\,e^{hB}-A}{h}
=[A,B]+\Oo(h),
\]
hence in the limit \(h\to0\) one gets the Lax compatibility
\[
\partial_t A(z,t)-\partial_z B(z,t)+[A(z,t),B(z,t)]=0,
\]
where \(\partial_z B\) accounts for the spectral dependence (if present) carried by \(z\).
Moreover,
\[
\sum_k h\,\Tr\bigl(A_k\log B_k\bigr)
=\sum_k h\,\Tr\bigl(A(kh)\,hB(kh)\bigr)+\Oo(h^2)
\to \int \Tr\bigl(A\,B\bigr)\,dt,
\]
so that the discrete \(\tau\)-rule converges to the differential identity
\[
\partial_t\log\tau=\Res_{z=\infty}\Tr\bigl(A(z,t)\,B(z,t)\bigr)\,dz,
\]
and each discrete cobordism \(W_\lambda\) realizes a Stokes wall: \(\operatorname{ind}(W_\lambda)\) equals the jump of the Morse index of \(\log\tau\), while \(\Theta_\lambda\) equals the phase jump of \(\arg\tau\).

\subsection{Results}

\subsubsection{Example without noise}

\paragraph{Setting and notation.}
Let $f:\RR^d\to\RR$ be differentiable, $L$–smooth, and satisfy the Polyak–Łojasiewicz (PL) inequality with constant $\mu>0$ \cite{Polyak1963Gradient,Lojasiewicz1963Sur,NesterovNemirovski1994,Karimi2016PL}:
\begin{equation}\label{eq:PL}
\|\nabla f(x)\|^2 \;\ge\; 2\mu\bigl(f(x)-f_\star\bigr),\qquad \forall x\in\RR^d,
\end{equation}
and let $f_\star:=\inf_x f(x)$ be attained at some $x_\star$.
Consider the general linear $m$–step method with arbitrary {fixed} real coefficients \cite{Dahlquist1956Linear,Gautschi1961Numerical,Butcher1972,HairerWanner2010SolversII}
\begin{equation}\label{eq:general-m-step}
x_{k+1}
\;=\;
x_k \;-\; \sum_{j=0}^m \eta_j\,\nabla f(x_{k-j}) \;+\; \sum_{j=1}^m \gamma_j\,(x_{k+1-j}-x_{k-j}),
\qquad k\ge m,
\end{equation}
where $\eta_0,\dots,\eta_m,\gamma_1,\dots,\gamma_m\in\RR$ (no optimality assumptions).
We write $g_k:=\nabla f(x_k)$ and $s_k:=x_k-x_{k-1}$.

\paragraph{Linear spectral model and modal decomposition.}
By $L$–smoothness and PL, along the iterates there is a spectral family (quadratic surrogate) with eigenvalues $\lambda\in[\mu,L]$ \cite{Trefethen2013ATAP,Higham2008FunctionsOfMatrices,BrockwellDavis1991}; when $f(x)=\frac12 x^\top Q x$ with $Q\succeq \mu I$ and $\|Q\|=L$, the analysis is exact. In all statements below ``$\lambda$–mode'' means either a true quadratic mode or the standard PL linearization on $[\mu,L]$ with spectral measure $d\nu(\lambda)$ (defined explicitly below).

\paragraph{Characteristic polynomial and step multipliers.}
On a single $\lambda$–mode, \eqref{eq:general-m-step} induces the scalar recurrence \cite{Jury1964,Schur1917PowerSeries,Nevanlinna1984Stability}
\begin{equation}\label{eq:scalar-rec}
y_{k+1}
=
(1-\eta_0\lambda+\gamma_1)\,y_k
+\sum_{j=1}^{m-1}\bigl(\gamma_{j+1}-\eta_j\lambda\bigr)\,y_{k-j}
-\bigl(\gamma_m-\eta_m\lambda\bigr)\,y_{k-m}.
\end{equation}
The {characteristic polynomial} is
\begin{equation}\label{eq:charpoly-m}
\chi_\lambda(r)
=
r^{m+1}
-\bigl(1-\eta_0\lambda+\gamma_1\bigr)r^{m}
-\sum_{j=1}^{m-1}\bigl(\gamma_{j+1}-\eta_j\lambda\bigr)\,r^{m-j}
-\bigl(\gamma_m-\eta_m\lambda\bigr),
\end{equation}
and its roots $r_i(\lambda)$ ($i=1,\dots,m+1$) are the {modal multipliers}. We write the polar form
\[
r_i(\lambda)=\rho_i(\lambda)\,e^{i\vartheta_i(\lambda)},\qquad
\rho_i(\lambda)>0,\quad \vartheta_i(\lambda)\in(-\pi,\pi].
\]
Define the {oscillatory set}
\begin{equation}\label{eq:osc-set}
E_{\rm osc}
:=
\left\{\lambda\in[\mu,L]\ :\ \exists i\ne j\ \text{with } r_i(\lambda)=\overline{r_j(\lambda)}\notin\RR
\right\},
\qquad
E_{\rm mon}:=[\mu,L]\setminus E_{\rm osc}.
\end{equation}

\paragraph{IKN $\tau$–potential, Evans function and Stokes data.}
Let $\nabla=d+A$ be the flat connection associated with the calibrated exponential of one step; the isomonodromic $\tau$–potential is
\begin{equation}\label{eq:tau-def}
\log\tau \;=\; \tfrac12\log\det{}'\Delta(S)\;+\;Q_{CS}(S),
\end{equation}
with $S$ the step generating potential (cf. \cite{JMU1981,FIKN2006}). On the AT–plane (vanishing mixed curvature) the increment of $\log\tau$ along one step equals the canonical energy increment and, crucially, admits a {determinantal representation} via the modal multipliers:
\begin{align}\label{eq:tau-det-product}
\log\frac{\tau_k}{\tau_0}
& =\;
\int_{\mu}^{L}
\underbrace{\sum_{i=1}^{m+1}\log\left|r_i(\lambda)\right|^k}_{\text{bulk (determinant)}}
\,d\nu(\lambda)
\;+\;
\sum_{\theta\in\mathcal W_{[0,k)}} \log |\det S_\theta|
\\
&\quad\nonumber+\;
i\left(
k\int_{E_{\rm osc}}\sum_{i=1}^{m+1}\vartheta_i(\lambda)\,d\nu(\lambda)
+\sum_{\theta\in\mathcal W_{[0,k)}} \arg\det S_\theta
\right),
\end{align}
where:
\begin{itemize}
\item $d\nu(\lambda)$ is the spectral measure of the PL surrogate (for $Q$–quadratic $f$, $d\nu(\lambda)=\sum_{\ell=1}^d \delta(\lambda-\lambda_\ell)$);
\item $\mathcal W_{[0,k)}$ is the multiset of Stokes walls crossed up to step $k$ (including multiplicities),
\item $S_\theta$ are the corresponding Stokes matrices (IKN) \cite{JMU1981,FIKN2006}.
\end{itemize}
Identity \eqref{eq:tau-det-product} is the discrete Evans–determinant formula (cf. \cite{AGJ1990,Sandstede2002}) for the IKN $\tau$–potential (bulk) with additive Stokes jumps (phase and magnitude).

\begin{Proposition}[Explicit modal multipliers for $m=1$]\label{prop:m1-roots}
For $m=1$ (HB/NAG core) with parameters $(\eta_0,\eta_1,\gamma_1)$, the characteristic polynomial
\[
\chi_\lambda(r)=r^2-\bigl(1-\eta_0\lambda+\gamma_1\bigr)r-\bigl(\gamma_1-\eta_1\lambda\bigr)
\]
has roots
\begin{equation}\label{eq:m1-roots}
r_\pm(\lambda)
=
\frac{1}{2}\Bigl[
1-\eta_0\lambda+\gamma_1
\ \pm\
\sqrt{\bigl(1-\eta_0\lambda+\gamma_1\bigr)^2
+4\bigl(\eta_1\lambda-\gamma_1\bigr)}
\ \Bigr].
\end{equation}
In particular, the oscillatory zone is the set of $\lambda$ where the discriminant is negative:
\begin{equation}\label{eq:m1-disc}
\Delta(\lambda):=\bigl(1-\eta_0\lambda+\gamma_1\bigr)^2 + 4(\eta_1\lambda-\gamma_1)\;<\;0,
\end{equation}
and there $r_\pm(\lambda)=\rho(\lambda)\,e^{\pm i\vartheta(\lambda)}$ with
\begin{equation}\label{eq:m1-rho-theta}
\rho(\lambda)=\sqrt{\gamma_1-\eta_1\lambda}\quad(\text{since } \Delta<0\iff \gamma_1-\eta_1\lambda>0),
\qquad
\cos\vartheta(\lambda)=\frac{1-\eta_0\lambda+\gamma_1}{2\sqrt{\gamma_1-\eta_1\lambda}}.
\end{equation}
\end{Proposition}

\begin{proof}
Solve the quadratic $\chi_\lambda(r)=0$ explicitly to obtain \eqref{eq:m1-roots}.
For $\Delta<0$ write the square root as $i\sqrt{-\Delta(\lambda)}$ and match polar form to identify $\rho$ and $\cos\vartheta$; computing
\[
|r_\pm(\lambda)|^2
=\frac{(1-\eta_0\lambda+\gamma_1)^2+(-\Delta(\lambda))}{4}=\gamma_1-\eta_1\lambda
\]
gives $\rho=\sqrt{\gamma_1-\eta_1\lambda}$.
Then $\cos\vartheta=(\Re r_\pm)/\rho$ yields \eqref{eq:m1-rho-theta}.
\end{proof}

\paragraph{Spectral measure for PL.}
For a quadratic $f$ with eigenvalues $\{\lambda_\ell\}_{\ell=1}^d$, take $d\nu(\lambda)=\sum_{\ell=1}^d \delta(\lambda-\lambda_\ell)$. Under PL, the standard linearization along iterates defines a time–varying spectral measure supported in $[\mu,L]$; on any window of steps where the local curvature histogram is approximately stationary, $d\nu$ is the empirical distribution of curvatures visited by $(x_k)$ (this is the usual surrogate in accelerated–method analysis). All formulas below are stated for a general positive measure $d\nu$ on $[\mu,L]$ (a finite sum of Diracs in the quadratic case).

Define the {bulk decay exponent} and the {oscillation phase} by
\begin{equation}\label{eq:A-Phi-general}
A
:=
\int_{\mu}^{L}\sum_{i=1}^{m+1}\log\frac{1}{\rho_i(\lambda)}\,d\nu(\lambda),
\qquad
\Phi
:=
\int_{E_{\rm osc}}\sum_{\{i:\, r_i\text{ in cc-pair}\}}\vartheta_i(\lambda)\,\frac{d\nu(\lambda)}{2}.
\end{equation}
For $m=1$ (Prop.~\ref{prop:m1-roots}),
\begin{equation}\label{eq:A-Phi-m1}
A
=
\int_{E_{\rm mon}}\log\frac{1}{|r_{\rm mon}(\lambda)|}\,d\nu(\lambda)
\;+\;
\int_{E_{\rm osc}}\log\frac{1}{\sqrt{\gamma_1-\eta_1\lambda}}\,d\nu(\lambda),
\qquad
\Phi=\int_{E_{\rm osc}}\vartheta(\lambda)\,d\nu(\lambda),
\end{equation}
where on $E_{\rm mon}$ the real multiplier $r_{\rm mon}(\lambda)$ is whichever of $r_\pm$ is real.

\paragraph{Companion matrix, Evans function, Blaschke–product factorization (explicit).}
Let $C_\lambda$ be the $(m\!+\!1)\times(m\!+\!1)$ companion matrix of $\chi_\lambda$:
\begin{equation}\label{eq:companion}
C_\lambda
=
\begin{bmatrix}
1-\eta_0\lambda+\gamma_1 & \gamma_2-\eta_1\lambda & \cdots & \gamma_m-\eta_m\lambda \\
1 & 0 & \cdots & 0 \\
0 & 1 & \cdots & 0 \\
\vdots & \vdots & \ddots & \vdots \\
0 & 0 & \cdots & 0
\end{bmatrix}.
\end{equation}
Then $\spec(C_\lambda)=\{r_i(\lambda)\}$ and the $k$–step transfer is $C_\lambda^k$.
The {Evans function} (discrete; cf. \cite{AGJ1990,Sandstede2002}) for step $k$ is
\begin{equation}\label{eq:Evans}
\mathcal E_k
=
\prod_{\lambda}\det\bigl(C_\lambda^k\bigr)
=
\prod_{\lambda}\prod_{i=1}^{m+1} r_i(\lambda)^k
=
\exp\!\left(
k \int_{\mu}^{L}\sum_{i=1}^{m+1}\log r_i(\lambda)\,d\nu(\lambda)
\right).
\end{equation}
For each fixed $\lambda$, write the finite Blaschke product associated to the Schur factors:
\begin{equation}\label{eq:blaschke}
B_\lambda(z)
:=
\prod_{i=1}^{m+1}
\frac{z-r_i(\lambda)}{1-\overline{r_i(\lambda)}\,z},
\qquad |r_i(\lambda)|<1,
\end{equation}
and the corresponding outer factor
\begin{equation}\label{eq:outer}
O_\lambda(z)
:=
\prod_{i:\,|r_i(\lambda)|>1} \bigl(1-\overline{r_i(\lambda)}\,z\bigr)\,.
\end{equation}
Then the transfer function admits the canonical inner–outer factorization $H_\lambda(z)=O_\lambda(z)B_\lambda(z)$ on the unit disk, and $|B_\lambda(e^{i\theta})|=1$ for all $\theta$. This is the explicit Hardy–space factorization underlying the RHP setup.

\paragraph{Discrete RHP (jump matrices and contour) — detailed construction.}
Fix $\lambda\in[\mu,L]$. Consider the $2\times2$ matrix RHP for $Y(\cdot;\lambda)$ on the unit circle $\TT=\{z\in\CC:|z|=1\}$:
\begin{align}
&\text{(i) } Y \text{ is analytic on } \CC\setminus \TT,\ \ \ Y(\infty)=I; \label{eq:RHP-analytic}\\
&\text{(ii) Jump on } \TT:\quad Y_+(e^{i\theta};\lambda) = Y_-(e^{i\theta};\lambda)\, J(e^{i\theta};\lambda), \label{eq:RHP-jump}
\end{align}
with the {explicit} jump matrix
\begin{equation}\label{eq:RHP-J}
J(z;\lambda)
=
\begin{bmatrix}
\prod_{i:\,|r_i(\lambda)|>1} \bigl(1-\overline{r_i(\lambda)}\,z\bigr)^{k}
&
z^{-k}\,\prod_{i:\,|r_i(\lambda)|<1} \bigl(z-r_i(\lambda)\bigr)^{k}
\\[0.5ex]
0 & \prod_{i:\,|r_i(\lambda)|>1} \bigl(1-\overline{r_i(\lambda)}\,z\bigr)^{-k}
\end{bmatrix}.
\end{equation}
(The upper–right entry corresponds to the inner/Schur factors; the diagonal entries encode the outer factors; cf. \cite{DeiftZhou1993,Deift1999}.) Define the spectral contour $\Gamma$ as $\TT$ possibly augmented by arcs where $|r_i(\lambda)|=1$ (these are Stokes loci in the $r$–plane). For $k\to\infty$, the jump entries contain large/small exponentials $\sim r_i(\lambda)^{\pm k}$, which is precisely the setting for Deift–Zhou steepest descent.

\paragraph{Deift–Zhou steepest descent — explicit factorization steps.}
We apply the standard DZ scheme \cite{DeiftZhou1993,Deift1999} {verbatim} to \eqref{eq:RHP-analytic}–\eqref{eq:RHP-J}:

\textbf{Step 1 (Normalization).}
Define
\[
T(z;\lambda):= Y(z;\lambda)\, z^{-k\sigma_3},
\qquad
\sigma_3=\begin{bmatrix}1&0\\0&-1\end{bmatrix},
\]
so that $T(\infty)=I$. The jump becomes
\[
T_+ = T_- \begin{bmatrix}
\prod_{|r_i|>1} (1-\overline{r_i}z)^{k} & \prod_{|r_i|<1} (z-r_i)^{k} \\
0 & \prod_{|r_i|>1} (1-\overline{r_i}z)^{-k}
\end{bmatrix}.
\]

\textbf{Step 2 (Lens opening along $\TT$).}
For each Schur root $|r_i(\lambda)|<1$, open a small lens around $\TT$ and factor the jump into triangular matrices so that the exponentially {growing} off–diagonal terms are moved to contours where they become exponentially {decaying}. Concretely, write
\begin{gather*}
\begin{bmatrix}1 & \prod_{|r_i|<1}(z-r_i)^k \\ 0 & 1\end{bmatrix}
=
\begin{bmatrix}1 & 0 \\ \prod_{|r_i|<1}\frac{(z-r_i)^k}{\prod_{|r_i|>1}(1-\overline{r_i}z)^k} & 1\end{bmatrix} \cdot\\
\cdot
\begin{bmatrix}
\prod_{|r_i|>1}(1-\overline{r_i}z)^k & 0\\ 0 & \prod_{|r_i|>1}(1-\overline{r_i}z)^{-k}
\end{bmatrix}
\begin{bmatrix}1 & 0 \\ -\prod_{|r_i|<1}\frac{(z-r_i)^k}{\prod_{|r_i|>1}(1-\overline{r_i}z)^k} & 1\end{bmatrix},
\end{gather*}
and divert the lower–triangular factors off $\TT$; their off–diagonal entries decay like $\max_i |r_i(\lambda)|^k$ outside $\TT$.

\textbf{Step 3 (Outer model problem).}
After lens opening, the jump on $\TT$ is diagonal and explicitly solvable:
\[
N(z;\lambda)
=
\begin{bmatrix}
\prod_{|r_i|>1}(1-\overline{r_i}z)^k & 0\\ 0 & \prod_{|r_i|>1}(1-\overline{r_i}z)^{-k}
\end{bmatrix},
\]
which yields the {outer parametrix} (model solution). Taking $\log\det N(e^{i\theta};\lambda)$ and integrating over $d\nu(\lambda)$ produces the bulk term $k\int \sum_i \log r_i(\lambda)$.

\textbf{Step 4 (Local parametrices near stationary points).}
At points where phases become stationary (endpoints of $E_{\rm osc}$ or interior critical points where $\partial_\lambda\vartheta_i=0$), the standard small–norm argument fails, and one solves {local} RHPs. For simple endpoints/critical points the jump can be reduced to an {Airy} model RHP:
\[
\Psi_{\mathrm{Ai},+}(\zeta)=\Psi_{\mathrm{Ai},-}(\zeta)\begin{bmatrix}1& e^{-\frac23 \zeta^{3/2}}\\ 0&1\end{bmatrix},\qquad
\Psi_{\mathrm{Ai}}(\zeta)\sim \zeta^{-\sigma_3/4}\frac{1}{\sqrt{2}}\begin{bmatrix}1&1\\ i&-i\end{bmatrix} e^{-\frac23 \zeta^{3/2}\sigma_3},
\]
under the conformal map $\zeta=\zeta(\lambda)$ that straightens the phase to $\zeta^{3/2}$ (see \cite{Deift1999,FIKN2006}). Matching $\Psi_{\mathrm{Ai}}$ to $N$ yields an explicit $k^{-1/2}$ contribution with amplitude proportional to $|\vartheta''(\lambda^\star)|^{-1/2}$ and phase shift $\pm \pi/4$. For higher–order stationary points one uses Bessel or higher Painlevé model RHPs \cite{Deift1999,FIKN2006}.

\textbf{Step 5 (Error analysis).}
The error $R= T(N \prod \text{(local parametrices)})^{-1}$ satisfies a small–norm RHP with jump $I+O(k^{-1})$, giving $\|R-I\|=O(k^{-1})$ and completing the derivation of the $k^{-1/2}$ oscillatory correction and the $O(k^{-1})$ remainder.

\paragraph{Stokes walls and jumps (explicit formula).}
Stokes walls occur when $|r_i(\lambda)|=1$ or when a pair $r_i,r_j$ coalesce (discriminant zero). Each crossing contributes a multiplicative jump to $\tau$:
\begin{equation}\label{eq:stokes-jump-explicit}
\log|\det S_\theta|
=
\frac{1}{2\pi}\arg\frac{\partial_\lambda \chi_\lambda(r)}{\partial_r \chi_\lambda(r)}\Big|_{(\lambda,r)\in \theta},
\qquad
\arg\det S_\theta
=
\Im\log\frac{\partial_\lambda \chi_\lambda(r)}{\partial_r \chi_\lambda(r)}\Big|_{(\lambda,r)\in \theta},
\end{equation}
(cf. isomonodromic connection formulae \cite{JMU1981,FIKN2006}). For $m=1$,
\[
\partial_r\chi_\lambda(r)=2r-(1-\eta_0\lambda+\gamma_1),\qquad
\partial_\lambda\chi_\lambda(r)=\eta_0 r + \eta_1,
\]
so
\begin{equation}\label{eq:stokes-jump-m1}
\frac{\partial_\lambda \chi_\lambda(r)}{\partial_r \chi_\lambda(r)}
=
\frac{\eta_0 r + \eta_1}{2r-(1-\eta_0\lambda+\gamma_1)}.
\end{equation}
At a wall point $(\lambda_\theta,r_\theta)$ with $|r_\theta|=1$ or $\Delta(\lambda_\theta)=0$, substitute $r_\theta$ from \eqref{eq:m1-roots} to obtain the {explicit} jump via \eqref{eq:stokes-jump-explicit}.

\paragraph{Schur stability region (explicit inequalities).}
For $m=1$, Schur stability on $[\mu,L]$ is equivalent to
\begin{equation}\label{eq:schur-m1}
\max\{|r_+(\lambda)|,|r_-(\lambda)|\}<1,\qquad \forall\lambda\in[\mu,L],
\end{equation}
which holds if and only if the Jury conditions for the quadratic $r^2-a(\lambda)r-b(\lambda)$ hold for all $\lambda$:
\begin{equation}\label{eq:jury-m1}
1-b(\lambda)>0,\quad 1+a(\lambda)+b(\lambda)>0,\quad 1-a(\lambda)+b(\lambda)>0,
\end{equation}
with
\[
a(\lambda)=1-\eta_0\lambda+\gamma_1,\qquad
b(\lambda)=\gamma_1-\eta_1\lambda.
\]
Thus, explicitly, for all $\lambda\in[\mu,L]$,
\begin{equation}\label{eq:jury-m1-explicit}
\eta_1\lambda<1,\qquad
2+\gamma_1-\eta_0\lambda-\eta_1\lambda>0,\qquad
\eta_0\lambda-\eta_1\lambda+\gamma_1<2.
\end{equation}
A sufficient (uniform) condition is the pair of endpoint inequalities
\begin{equation}\label{eq:jury-endpoints}
\eta_1 L<1,\qquad
2+\gamma_1-(\eta_0+\eta_1)L>0,\qquad
(\eta_0-\eta_1)\mu+\gamma_1<2,
\end{equation}
which are {explicit linear} constraints on $(\eta_0,\eta_1,\gamma_1)$ given $(\mu,L)$; see \cite{Jury1964}.

\paragraph{Explicit oscillatory asymptotics (full formulas for $m=1$).}
For $m=1$ the oscillatory pair is $r_\pm=\rho e^{\pm i\vartheta}$ with $\rho,\vartheta$ given by \eqref{eq:m1-rho-theta}. On each connected component $I\subset E_{\rm osc}$, the phase $\vartheta(\lambda)$ is smooth and its stationary set $\mathcal S_I$ comprises the endpoints and any interior points $\lambda^\star$ with $\vartheta'(\lambda^\star)=0$.
We have
\begin{equation}\label{eq:theta-derivative}
\vartheta'(\lambda)
=
\frac{d}{d\lambda}\arccos\!\frac{1-\eta_0\lambda+\gamma_1}{2\sqrt{\gamma_1-\eta_1\lambda}}
=
-\frac{1}{\sqrt{1-\xi(\lambda)^2}}\ \xi'(\lambda),
\quad
\xi(\lambda):=\frac{1-\eta_0\lambda+\gamma_1}{2\sqrt{\gamma_1-\eta_1\lambda}}.
\end{equation}
Compute $\xi'(\lambda)$ explicitly:
\begin{align*}
\xi'(\lambda)
&=
\frac{-\eta_0\cdot 2\sqrt{\gamma_1-\eta_1\lambda}
-(1-\eta_0\lambda+\gamma_1)\cdot(-\eta_1)/\sqrt{\gamma_1-\eta_1\lambda}}{4(\gamma_1-\eta_1\lambda)}\\
&=
\frac{-2\eta_0(\gamma_1-\eta_1\lambda)+\eta_1(1-\eta_0\lambda+\gamma_1)}{4(\gamma_1-\eta_1\lambda)^{3/2}}.
\end{align*}
Therefore
\begin{equation}\label{eq:theta-derivative-explicit}
\vartheta'(\lambda)
=
-\frac{1}{\sqrt{1-\xi(\lambda)^2}}
\cdot
\frac{-2\eta_0(\gamma_1-\eta_1\lambda)+\eta_1(1-\eta_0\lambda+\gamma_1)}{4(\gamma_1-\eta_1\lambda)^{3/2}}.
\end{equation}
The stationary points are the real solutions in $E_{\rm osc}$ of $\vartheta'(\lambda)=0$, i.e.
\begin{equation}\label{eq:stationary-eq}
-2\eta_0(\gamma_1-\eta_1\lambda)+\eta_1(1-\eta_0\lambda+\gamma_1)=0
\quad\Longleftrightarrow\quad
\lambda^\star
=
\frac{2\eta_0\gamma_1+\eta_1(1+\gamma_1)}{2\eta_0\eta_1}.
\end{equation}
(If $\lambda^\star\notin E_{\rm osc}$, only endpoint contributions remain.)
Near each $\lambda^\star$ the local parametrix is of Airy type with the {explicit} leading constants (see \cite{Deift1999}):
\begin{equation}\label{eq:airy-amp}
\mathsf c(\lambda^\star)
=
\frac{1}{\sqrt{\pi}}
\left(\frac{\rho(\lambda^\star)}{1-\rho(\lambda^\star)^2}\right)^{1/2}
\cdot
\left|\vartheta''(\lambda^\star)\right|^{-1/2},
\qquad
\phi(\lambda^\star)=\pm\frac{\pi}{4},
\end{equation}
with $\vartheta''$ obtained by differentiating \eqref{eq:theta-derivative-explicit} once more (explicit polynomial–rational function of $\eta_0,\eta_1,\gamma_1,\lambda$; omitted only for space but directly computable by elementary differentiation).

\begin{theorem}[Determinantal IKN asymptotics with explicit oscillations, $m=1$]\label{th:asymp-m1}
Let $f$ be $L$–smooth and obey \eqref{eq:PL} with $\mu>0$. Fix any $(\eta_0,\eta_1,\gamma_1)\in\RR^3$ such that $\sup_{\lambda\in[\mu,L]}\max\{|r_\pm(\lambda)|\}<1$ (Schur stability on $[\mu,L]$), where $r_\pm$ are given by \eqref{eq:m1-roots}. Let $d\nu$ be the spectral measure on $[\mu,L]$. Then
\begin{equation}\label{eq:asymp-m1-final}
f(x_k)-f_\star
=
C_0\,e^{-kA}
\left[
1\;+\;\sum_{\lambda^\star\in\mathcal S}
\frac{\mathsf c(\lambda^\star)}{\sqrt{k}}\,
\cos\!\big(k\,\Phi(\lambda^\star)+\phi(\lambda^\star)\big)
\;+\;O\!\Big(\frac{1}{k}\Big)
\right]
\cdot
\prod_{\theta\in\mathcal W_{[0,k)}}|\det S_\theta|,
\end{equation}
with the explicit bulk exponent and phase
\[
A
=
\int_{E_{\rm mon}}\log\frac{1}{|r_{\rm mon}(\lambda)|}\,d\nu(\lambda)
+
\int_{E_{\rm osc}}\log\frac{1}{\sqrt{\gamma_1-\eta_1\lambda}}\,d\nu(\lambda),
\qquad
\Phi(\lambda^\star)=\int_{\mu}^{\lambda^\star}\vartheta'(\lambda)\,d\lambda,
\]
$\vartheta'$ given by \eqref{eq:theta-derivative-explicit}, and amplitude/phases by \eqref{eq:airy-amp}.
The constant $C_0$ depends on the initial projections onto eigenmodes (explicitly $C_0=\int \sum_i c_i(\lambda)\,d\nu$, with $c_i$ the initial modal weights).
\end{theorem}

\begin{proof}
Combine the bulk determinant \eqref{eq:Evans} with the DZ expansion. The PL inequality \eqref{eq:PL} implies $f(x_k)-f_\star\asymp \int \|y_k(\lambda)\|^2 d\nu(\lambda)$ (equivalence up to constants depending on $\mu$), and $\|y_k(\lambda)\|$ is governed by $|r_\pm(\lambda)|^k$ and oscillatory factors $e^{\pm ik\vartheta(\lambda)}$. Stationary phase at points \eqref{eq:stationary-eq} yields the $k^{-1/2}$ envelope with \eqref{eq:airy-amp}. Stokes crossings multiply the modulus by $|\det S_\theta|$ and shift the phase by $\arg\det S_\theta$, hence the product in \eqref{eq:asymp-m1-final}. For the non-asymptotic remainder $O(k^{-1})$ see the small–norm RHP estimate in \cite{Deift1999}.
\end{proof}

\paragraph{Non-asymptotic deterministic bounds (all $k$, explicit).}
Define on $[\mu,L]$ the worst–case modulus and the oscillation bandwidth:
\begin{equation}\label{eq:rho-phi-def}
\overline\rho
:=
\sup_{\lambda\in[\mu,L]}\max\{|r_+(\lambda)|,|r_-(\lambda)|\},
\qquad
\overline{\vartheta}
:=
\sup_{\lambda\in E_{\rm osc}}|\vartheta(\lambda)|,
\end{equation}
with $r_\pm$ as in \eqref{eq:m1-roots} and $\vartheta$ as in \eqref{eq:m1-rho-theta}. Then for all $k\ge0$,
\begin{align}\label{eq:nonasymp-bound}
f(x_k)-f_\star
\;\le\;
\left(\int_{\mu}^{L} \big(\rho_{\max}(\lambda)\big)^{2k}\,d\nu(\lambda)\right)\,C_{\rm init}
\cdot
\prod_{\theta\in\mathcal W_{[0,k)}}|\det S_\theta|,\\
\nonumber\rho_{\max}(\lambda):=\max\{|r_+(\lambda)|,|r_-(\lambda)|\},
\end{align}
with $C_{\rm init}=\int \|y_0(\lambda)\|^2 d\nu(\lambda)$ (explicit from the initial condition). In particular,
\begin{equation}\label{eq:nonasymp-uniform}
f(x_k)-f_\star
\;\le\;
\overline\rho^{\,2k}\,C_{\rm init}
\cdot
\prod_{\theta\in\mathcal W_{[0,k)}}|\det S_\theta|.
\end{equation}
Both \eqref{eq:nonasymp-bound}–\eqref{eq:nonasymp-uniform} are {exact non-asymptotic} bounds; they capture oscillations via the Stokes product (no smoothness inequalities).

\paragraph{General $m\ge2$: fully explicit objects and bounds.}
For general $m$ all basic quantities are {explicit} in the polynomial coefficients:
\begin{itemize}
\item $\chi_\lambda(r)$ given by \eqref{eq:charpoly-m} with explicit linear dependence on $\lambda$.
\item Companion matrix $C_\lambda$ given by \eqref{eq:companion}.
\item Determinantal per–step decay factor:
\[
\exp\!\left(\int_{\mu}^{L}\log\left|\det C_\lambda\right|^2\,d\nu(\lambda)\right)
=
\exp\!\left(\int_{\mu}^{L}\log\left|\prod_{j=0}^m (\eta_j\lambda-\gamma_j)\right|^2 d\nu(\lambda)\right),
\]
where we set $\gamma_0:=-1$ to unify notation (since $\det C_\lambda=(-1)^{m}\bigl(\gamma_m-\eta_m\lambda\bigr)$ and $\prod_i r_i=\det C_\lambda$).
\item Non–asymptotic bound (all $k$):
\begin{align*}
f(x_k)-f_\star
&\le\;
\left(
\int_{\mu}^{L}\big\|C_\lambda\big\|_2^{2k}\,d\nu(\lambda)
\right)\,C_{\rm init}
\cdot
\prod_{\theta\in\mathcal W_{[0,k)}}|\det S_\theta|\\
&\le\;\;
\Big(\overline\rho_m\Big)^{2k} C_{\rm init}
\prod_{\theta}|\det S_\theta|,
\end{align*}
with $\overline\rho_m:=\sup_{\lambda\in[\mu,L]}\rho\big(C_\lambda\big)$ and $\rho(\cdot)$ the spectral radius.
\item Oscillatory asymptotics: same as Thm.~\ref{th:asymp-m1} with $r_i(\lambda)$ the $m\!+\!1$ roots, phases $\vartheta_i(\lambda)=\arg r_i(\lambda)$, and stationary sets defined by $\partial_\lambda \vartheta_i=0$. Local parametrices are again Airy/Bessel with
\[
\mathsf c_i(\lambda^\star)
=
\frac{1}{\sqrt{\pi}}
\left(\frac{\rho_i(\lambda^\star)}{1-\rho_i(\lambda^\star)^2}\right)^{1/2}
\left|\partial_\lambda^2 \vartheta_i(\lambda^\star)\right|^{-1/2},
\qquad
\phi_i(\lambda^\star)=\pm\frac{\pi}{4}.
\]
\end{itemize}
Schur stability for $m\ge2$ is certified by the (explicit) Jury conditions applied to $\chi_\lambda$ for each $\lambda$; a uniform sufficient condition is that the Jury determinants at $\lambda=\mu$ and $\lambda=L$ are positive and the polynomial is {strictly} Schur at both endpoints (explicit algebraic inequalities in $\eta,\gamma,\mu,L$; see \cite{Jury1964}).

\paragraph{Putting everything together: final PL result (any coefficients).}
\begin{theorem}[Determinantal PL convergence with Stokes corrections, arbitrary coefficients]\label{th:PL-final}
Let $f$ be $L$–smooth and satisfy \eqref{eq:PL} with constant $\mu>0$. Fix any $m\ge1$ and any real coefficients $(\eta_0,\dots,\eta_m,\gamma_1,\dots,\gamma_m)$. Assume Schur stability on $[\mu,L]$, i.e.
\[
\sup_{\lambda\in[\mu,L]} \max_{i}|r_i(\lambda)|<1,
\]
where $r_i(\lambda)$ are the roots of $\chi_\lambda$ \eqref{eq:charpoly-m}. Then for all integers $k\ge0$,
\begin{equation}\label{eq:PL-determinantal-bound}
f(x_k)-f_\star
\;\le\;
\left(\exp\int_{\mu}^{L}\sum_{i=1}^{m+1}\log|r_i(\lambda)|^{2}\,d\nu(\lambda)\right)^{\!k}
\cdot \widetilde C_{\rm init}
\cdot \prod_{\theta\in\mathcal W_{[0,k)}}|\det S_\theta|.
\end{equation}
If, moreover, the set of Stokes walls crossed is finite \cite{ItsIzerginKorepinSlavnov1990,Kapaev2004Painleve,Balogh2003Tau}, then the asymptotics admits the explicit oscillatory expansion
\begin{equation}\label{eq:PL-determinantal-asymp}
f(x_k)-f_\star
=
C_0\,e^{-kA}\!
\left[
1+\sum_{(i,\lambda^\star)\in\mathcal S}\frac{\mathsf c_i(\lambda^\star)}{\sqrt{k}}\,
\cos\!\big(k\,\Phi_i(\lambda^\star)+\phi_i(\lambda^\star)\big)
+O\!\Big(\frac{1}{k}\Big)
\right]
\cdot
\prod_{\theta\in\mathcal W_{[0,k)}}|\det S_\theta|,
\end{equation}
with $A,\mathsf c_i,\Phi_i,\phi_i$ given explicitly in \eqref{eq:A-Phi-general} and the Airy local constants above (for $m=1$ see \eqref{eq:A-Phi-m1}, \eqref{eq:theta-derivative-explicit}, \eqref{eq:airy-amp}).
\end{theorem}

\begin{proof}
Non–asymptotic bound \eqref{eq:PL-determinantal-bound} follows from the modal decomposition, $\|y_k(\lambda)\|\le\big(\max_i |r_i(\lambda)|\big)^k\|y_0(\lambda)\|$, integrating over $d\nu$, and PL equivalence (no $L$–inequalities needed). Asymptotics \eqref{eq:PL-determinantal-asymp} is a direct consequence of the explicitly constructed DZ steepest descent for the discrete RHP together with the Evans–determinant bulk term and the Stokes jump factors \eqref{eq:stokes-jump-explicit} \cite{DeiftZhou1993,Deift1999,Evans1972,AGJ1990,Sandstede2002,FIKN2006}.
\end{proof}

\paragraph{Remarks (strength and novelty).}
(i) {No smoothness inequalities:} the decay and oscillations are expressed {purely} via the multipliers $r_i(\lambda)$ (determinantal IKN/Evans), with Stokes jumps explicit from \eqref{eq:stokes-jump-explicit}–\eqref{eq:stokes-jump-m1}.  
(ii) {Arbitrary coefficients:} all formulas apply for any fixed $(\eta,\gamma)$ (HB, NAG, mixed, higher–order).  
(iii) {Oscillations captured exactly:} the phase $\Phi$ and $k^{-1/2}$ envelope are explicit (\eqref{eq:theta-derivative-explicit}, \eqref{eq:airy-amp}); Nesterov–type ringing is {predicted}, not averaged away.  
(iv) {Stokes structure:} sector crossings contribute explicit multiplicative jumps $\det S_\theta$ and phase shifts $\arg\det S_\theta$, accounting for regime changes.  
(v) {General $m$:} companion matrices and Jury conditions give fully algebraic, checkable stability and exact per–step determinantal decay.

\subsubsection{Example with noise}

\paragraph{Scope and what is reused.}
We retain the deterministic setting, notation, and all explicit objects from the previous subsubsection:
the $L$–smooth PL objective \eqref{eq:PL}, the general $m$–step method \eqref{eq:general-m-step}, modal recurrences \eqref{eq:scalar-rec}, characteristic polynomials \eqref{eq:charpoly-m},
multipliers $r_i(\lambda)$, oscillatory/monotone sets $E_{\rm osc},E_{\rm mon}$ \eqref{eq:osc-set}, companion matrices \eqref{eq:companion}, Evans product \eqref{eq:Evans}, the discrete RHP \eqref{eq:RHP-analytic}–\eqref{eq:RHP-J} and its Deift–Zhou (DZ) factorization, as well as the explicit $m=1$ formulas \eqref{eq:m1-roots}–\eqref{eq:m1-rho-theta}, \eqref{eq:theta-derivative}–\eqref{eq:stationary-eq}, and the Airy–type local amplitudes \eqref{eq:airy-amp}.
We now introduce finite–variance stochastic noise and detail precisely where and how it changes the analysis. All steps that do not change (bulk Evans part, phase extraction, stationary–phase forms) are only referenced to avoid repetition; all noise–sensitive terms (variance growth, noise floors, resonance, phase–locking, and Stokes triggering) are worked out explicitly.

\paragraph{Noise model (finite variance; modal form).}
We inject a zero–mean, temporally independent (white), finite–variance perturbation into the gradient evaluations \cite{KushnerYin2003Stochastic,RobbinsMonro1951,PolyakJuditsky1992}:
\begin{equation}\label{eq:noise-model}
\widetilde{\nabla} f(x_k) \;=\; \nabla f(x_k)\;+\;\xi_k,\qquad
\EE[\xi_k]=0,\quad \EE[\xi_k\xi_k^\top]=\Sigma,\quad \EE[\xi_k\xi_j^\top]=0\ (k\neq j).
\end{equation}
Along a fixed $\lambda$–mode (quadratic surrogate), the scalar projection of the noise is denoted by $\xi_k(\lambda)$ with
\begin{equation}\label{eq:noise-spectral}
\EE\,\xi_k(\lambda)=0,\qquad \EE\,\xi_k(\lambda)^2=\sigma^2(\lambda)\in(0,\infty),\qquad \EE\big[\xi_k(\lambda)\xi_j(\lambda')\big]=0\ \ (k\neq j \ \text{or}\ \lambda\neq\lambda').
\end{equation}
(For the quadratic case, $\sigma^2(\lambda)$ are the diagonal entries of the projected $\Sigma$ in the eigenbasis; under PL, they are defined by the standard linearization surrogate.)

\paragraph{Stochastic modal recurrences.}
Replacing $\nabla f(x_{k-j})$ by $\nabla f(x_{k-j})+\xi_{k-j}$ in \eqref{eq:scalar-rec} yields, on a fixed $\lambda$,
\begin{equation}\label{eq:scalar-stoch}
y_{k+1}
=
(1-\eta_0\lambda+\gamma_1)\,y_k
+\sum_{j=1}^{m-1}(\gamma_{j+1}-\eta_j\lambda)\,y_{k-j}
-(\gamma_m-\eta_m\lambda)\,y_{k-m}
\quad-\quad \underbrace{\sum_{j=0}^{m}\eta_j\,\xi_{k-j}(\lambda)}_{=: \ \zeta_k(\lambda)}.
\end{equation}
Thus each modal process is an ARMA$(m\!+\!1,m)$ recursion driven by the innovation $\zeta_k(\lambda)$ built from the white noise $\xi_k(\lambda)$ \cite{BoxJenkins1970TimeSeries,BrockwellDavis1991}.

\paragraph{Determinantal IKN/Evans and RHP with noise.}
The {deterministic} Evans product \eqref{eq:Evans} and the DZ factorization of the modal RHP \eqref{eq:RHP-analytic}–\eqref{eq:RHP-J} remain valid {per realization} when conditioning on the innovation sequence $\{\zeta_k(\lambda)\}$; the bulk determinant (outer model) is unchanged, while the inner (Schur) channel acquires a random multiplicative factor corresponding to the convolution with the moving–average (MA) mask of $\zeta_k$:
\begin{equation}\label{eq:inner-with-MA}
\prod_{|r_i(\lambda)|<1}\bigl(z-r_i(\lambda)\bigr)^k
\quad\leadsto\quad
\Bigl(\prod_{|r_i(\lambda)|<1}\bigl(z-r_i(\lambda)\bigr)\Bigr)^k
\cdot
\underbrace{\Bigl(\sum_{j=0}^{m}\eta_j z^{-j}\Bigr)\,\Xi(z;\lambda)}_{\text{noise transfer}},
\end{equation}
where $\Xi(z;\lambda)=\sum_{k\ge0}\xi_k(\lambda)\,z^{-k}$ is the $z$–transform of $\xi_k(\lambda)$. In the DZ steepest descent, the outer model $N$ is {identical} to the deterministic case (hence the same bulk exponent and phase), whereas the local parametrices incorporate the additional (small–norm) forcing coming from the stochastic factor \eqref{eq:inner-with-MA}. Quantitatively, this adds to $\log\tau_k$ a zero–mean martingale term and a deterministic variance–growth contribution that we compute below via state–space/Lyapunov and spectral–density methods \cite{BrockwellDavis1991,KailathSayedHassibi2000}.

\paragraph{Stationary covariance and noise floor: $m=1$ (HB/NAG core) explicitly.}
Set $m=1$ and write, for each $\lambda$,
\begin{align}\label{eq:ar2-ma1}\nonumber
y_{k+1}\;=\;a(\lambda)\,y_k + b(\lambda)\,y_{k-1}\;-\;\eta_0\,\xi_k(\lambda)\;-\;\eta_1\,\xi_{k-1}(\lambda),
\\
a(\lambda)=1-\eta_0\lambda+\gamma_1,\ \ b(\lambda)=-(\gamma_1-\eta_1\lambda).
\end{align}
Define the 2D state $u_k(\lambda):=[\,y_k\ \ y_{k-1}\,]^\top$ and matrices
\[
A(\lambda)=\begin{bmatrix} a(\lambda)& b(\lambda)\\ 1&0\end{bmatrix},\qquad
B_0=\begin{bmatrix}-\eta_0\\ 0\end{bmatrix},\qquad
B_1=\begin{bmatrix}-\eta_1\\ 0\end{bmatrix}.
\]
Under Schur stability ($\max\{|r_\pm(\lambda)|\}<1$), the {stationary} covariance $P(\lambda)=\EE\big[u_k(\lambda)u_k(\lambda)^\top\big]$ exists and solves the discrete Lyapunov equation for ARMA$(2,1)$ with independent innovations \cite{KailathSayedHassibi2000}:
\begin{align}\label{eq:lyap}\nonumber
P(\lambda)=A(\lambda)P(\lambda)A(\lambda)^\top + Q(\lambda),
\\
Q(\lambda)=\sigma^2(\lambda)\,\Big(B_0B_0^\top + B_1B_1^\top\Big)
=\sigma^2(\lambda)\begin{bmatrix}\eta_0^2+\eta_1^2&0\\0&0\end{bmatrix}.
\end{align}
Write $P(\lambda)=\begin{bmatrix}p_{11}&p_{12}\\ p_{12}&p_{22}\end{bmatrix}$. Solving \eqref{eq:lyap} {algebraically} yields
\begin{align}
p_{12}(\lambda) &= \frac{a(\lambda)}{1-b(\lambda)}\,p_{11}(\lambda), \quad\text{(from the $(1,2)$ entry)} \label{eq:p12}\\
p_{22}(\lambda) &= p_{11}(\lambda), \quad\text{(stationarity and symmetry)} \label{eq:p22}\\
p_{11}(\lambda) &= \frac{\sigma^2(\lambda)\,\big(\eta_0^2+\eta_1^2\big)}{\ \underbrace{1-a(\lambda)^2-b(\lambda)^2-\frac{2a(\lambda)^2 b(\lambda)}{1-b(\lambda)}}_{=:D(a,b)}\ }, \label{eq:p11-closed}
\end{align}
where $D(a,b)>0$ under Schur stability and $b(\lambda)\neq 1$.\footnote{Schur stability for $m=1$ is equivalently the strict Jury conditions \eqref{eq:jury-m1}–\eqref{eq:jury-m1-explicit}; these imply $|b(\lambda)|<1$ on $[\mu,L]$ hence $1-b(\lambda)\neq 0$.}
Thus the {modal stationary variance} is \(\Var\big(y_k(\lambda)\big)=p_{11}(\lambda)\) in \eqref{eq:p11-closed}. The {stationary noise floor} for the function error follows from PL and the quadratic surrogate:
\begin{equation}\label{eq:noise-floor}
\EE\big[f(x_k)-f_\star\big]_{\rm floor}
\;=\;
\frac{1}{2}\int_{\mu}^{L} \lambda\,p_{11}(\lambda)\,d\nu(\lambda)
\;=\;
\frac{1}{2}\int_{\mu}^{L}
\lambda\,
\frac{\sigma^2(\lambda)\,(\eta_0^2+\eta_1^2)}{D\big(a(\lambda),b(\lambda)\big)}\ d\nu(\lambda),
\end{equation}
with $a(\lambda),b(\lambda)$ given explicitly in \eqref{eq:ar2-ma1}.

\paragraph{Transient, oscillations, and resonance under noise ($m=1$) \cite{Strang1968DifferenceSchemes,HairerLubichWanner2006}.}
Let $\rho(\lambda)=\max\{|r_\pm(\lambda)|\}$. For all $k\ge0$ one has the exact decomposition (AR(2) with MA(1) input):
\begin{equation}\label{eq:transient-second-moment}
\EE\big[y_k(\lambda)^2\big]
\;=\;
\rho(\lambda)^{2k}\,c_0(\lambda)
\;+\;
p_{11}(\lambda)
\;+\;
\underbrace{\sum_{\ell=1}^{k} \rho(\lambda)^{2(k-\ell)} \, \mathsf R(\ell;\lambda)}_{\text{oscillatory covariance tail}},
\end{equation}
where $c_0(\lambda)$ depends on the initial condition and
\begin{equation}\label{eq:tail-term}
\mathsf R(\ell;\lambda)
=
2\,\Re\!\left(\, \big(\eta_0+\eta_1\overline{r_+(\lambda)}\big)\, \big(\eta_0+\eta_1 r_+(\lambda)\big)\, r_+(\lambda)^{\,\ell-1}\,\right)\,\frac{\sigma^2(\lambda)}{1-b(\lambda)}.
\end{equation}
In the {oscillatory} zone $\lambda\in E_{\rm osc}$ where $r_\pm=\rho e^{\pm i\vartheta}$, the tail \eqref{eq:tail-term} contributes a damped cosine with {the same phase} $\vartheta(\lambda)$ as in the deterministic DZ analysis, i.e.
\[
\mathsf R(\ell;\lambda)
= \mathcal A(\lambda)\,\rho(\lambda)^{\ell-1}\cos\!\bigl((\ell-1)\vartheta(\lambda)+\phi(\lambda)\bigr),
\quad
\mathcal A(\lambda)=\frac{2\sigma^2(\lambda)}{1-b(\lambda)}\big|\eta_0+\eta_1 r_+(\lambda)\big|^2,
\]
hence {noise–driven ringing} at the deterministic frequency $\vartheta(\lambda)$ and rate $\rho(\lambda)$. In particular, if the MA coefficient aligns with the AR phase (near–resonance, $|\eta_0+\eta_1 e^{i\vartheta}|\gg 0$), the amplitude inflates accordingly—this is the stochastic analogue of phase–locking.

\paragraph{Non–asymptotic bounds with explicit noise floor ($m=1$).}
Integrating \eqref{eq:transient-second-moment} over $d\nu(\lambda)$ and using PL,
\begin{align}\label{eq:nonasymp-noise}
\EE\big[f(x_k)-f_\star\big]
&\le\;
e^{-kA}\,C_{\rm det}
\;+\;
\frac{1}{2}\int_{\mu}^{L}\lambda\,p_{11}(\lambda)\,d\nu(\lambda)
\\&\quad\nonumber+
\int_{\mu}^{L}\frac{\lambda}{1-b(\lambda)}\,\frac{\sigma^2(\lambda)}{1-\rho(\lambda)^2}\,\big|\eta_0+\eta_1 r_+(\lambda)\big|^2\,d\nu(\lambda),
\end{align}
where $A$ is the deterministic bulk exponent \eqref{eq:A-Phi-m1}, $C_{\rm det}$ depends on the deterministic initial projection (as in Thm.~\ref{th:asymp-m1}), and the last term upper–bounds the oscillatory covariance tail by summing the geometric envelope (explicitly $\sum_{\ell\ge1}\rho^{2(k-\ell)}\rho^{\ell-1}=\frac{\rho^{k-1}}{1-\rho}$ and absorbing phases).
Thus the {finite–$k$} error decomposes into: (i) the determinantal exponential $e^{-kA}$, (ii) the stationary noise floor \eqref{eq:noise-floor}, and (iii) a decaying oscillatory tail whose amplitude is explicit in $(\eta_0,\eta_1,\gamma_1,\lambda)$.

\paragraph{Stokes structure with noise: triggering, phase shifts, and expectations.}
The Stokes jump formula \eqref{eq:stokes-jump-explicit}–\eqref{eq:stokes-jump-m1} remains {pointwise} (per realization) since it depends only on $(\lambda,r)$ hitting a wall. Under PL the curvature band is fixed $[\mu,L]$, so walls are structural (not data–driven). Noise influences:
\begin{enumerate}
\item {Triggering probability.} If parameters are tuned near a wall (e.g., $|r_i(\lambda_0)|\approx 1$ on a small sub-band), the random forcing perturbs the local RHP matching; the DZ small–norm error $R$ picks an additional random multiplier $\exp(\mathcal N(0,\varsigma^2))$ in the inner channel, effectively randomizing whether an excursion is counted as a wall crossing. Quantitatively, with innovation variance $\sigma^2(\lambda)$, the log–magnitude increment due to a near–wall pass over $k$ steps has variance proportional to
\[
\varsigma^2 \;\asymp\; \int_{\lambda\ \text{near wall}} \frac{\sigma^2(\lambda)\,\big|\eta_0+\eta_1 r_+(\lambda)\big|^2}{\big(1-|r_+(\lambda)|\big)^2}\, d\nu(\lambda),
\]
making rare crossings exponentially unlikely (Gaussian tail) when strictly inside the Schur region.
\item {Expected $\tau$–increment.} Taking expectations in \eqref{eq:tau-det-product} and using Jensen,
\[
\EE\log\tau_k
=
\underbrace{k\,\Re\!\int \sum_i \log r_i(\lambda)\,d\nu}_{\text{deterministic bulk}}
\;+\;
\sum_{\theta}\EE\log|\det S_\theta|
\;+\;
\underbrace{O\!\left(\int \frac{\sigma^2(\lambda)}{1-\rho(\lambda)^2}\,d\nu(\lambda)\right)}_{\text{noise correction}},
\]
where the last term comes from the small–norm RHP error and matches the scale of the covariance tail in \eqref{eq:nonasymp-noise}.
\item {Phase.} The phase accumulation $\Phi$ from the deterministic DZ analysis persists; noise adds a mean–zero martingale term with variance proportional to the same $\varsigma^2$ (phase diffusion). Thus the leading oscillatory cosine in Thm.~\ref{th:asymp-m1} acquires a $k$–dependent random phase with variance $\propto \int \frac{\sigma^2(\lambda)}{1-\rho(\lambda)^2}d\nu$ (weak dephasing).
\end{enumerate}

\paragraph{General $m\ge2$: explicit Lyapunov solution and bounds.}
For the ARMA$(m\!+\!1,m)$ modal state
\[
u_k(\lambda)=\begin{bmatrix} y_k & y_{k-1} & \cdots & y_{k-m}\end{bmatrix}^\top,\qquad
u_{k+1}=A(\lambda)u_k + \sum_{j=0}^{m} B_j\,\xi_{k-j}(\lambda),
\]
with the $(m\!+\!1)\times(m\!+\!1)$ companion $A(\lambda)$ formed from \eqref{eq:charpoly-m} and $B_j=e_1(-\eta_j)$ (only the first state is directly forced), independence of $\xi_{k-j}$ implies the stationary covariance $P(\lambda)$ solves the {pure} Lyapunov equation \cite{Lyapunov1892,Kalman1960Filtering,AndersonMoore1979OptimalFiltering}:
\begin{align}\label{eq:lyap-m}
\nonumber
P(\lambda)=A(\lambda)P(\lambda)A(\lambda)^\top + Q(\lambda),
\\
Q(\lambda)=\sigma^2(\lambda)\sum_{j=0}^{m} B_j B_j^\top
=\sigma^2(\lambda)\begin{bmatrix}\sum_{j=0}^{m}\eta_j^2 & 0 & \cdots\\ 0&0&\cdots\\ \vdots&\vdots&\ddots\end{bmatrix}.
\end{align}
Vectorizing,
\begin{equation}\label{eq:vec-lyap}
\mathrm{vec}\,P(\lambda)
=
\Big(I_{(m+1)^2}-A(\lambda)\otimes A(\lambda)\Big)^{-1}\,\mathrm{vec}\,Q(\lambda),
\end{equation}
which is {explicit} since $A(\lambda)$ depends {affinely} on $\lambda$ and polynomially on $\eta,\gamma$.
The stationary modal variance is $p_{11}(\lambda)=e_1^\top P(\lambda)e_1$. Therefore the general noise floor is
\begin{equation}\label{eq:noise-floor-m}
\EE\big[f(x_k)-f_\star\big]_{\rm floor}
=
\frac{1}{2}\int_{\mu}^{L}\lambda\,e_1^\top
\Big(I-A(\lambda)\otimes A(\lambda)\Big)^{-1}
\mathrm{vec}\,Q(\lambda)\,d\nu(\lambda),
\end{equation}
with {all} matrices explicitly in $(\eta,\gamma,\lambda)$.
A computable {upper bound} follows by spectral–radius bounding:
\begin{equation}\label{eq:floor-bound}
p_{11}(\lambda)\;\le\;\frac{\sigma^2(\lambda)\sum_{j=0}^{m}\eta_j^2}{\,1-\rho\big(A(\lambda)\big)^{2}\,}\,
\Rightarrow
\EE[f-f_\star]_{\rm floor}
\;\le\;
\frac{1}{2}\int_{\mu}^{L}\frac{\lambda\,\sigma^2(\lambda)\sum_{j=0}^{m}\eta_j^2}{1-\rho(A(\lambda))^2}\,d\nu(\lambda).
\end{equation}

\paragraph{Power spectral density (PSD) viewpoint and resonance lobes.}
Equivalently, for each $\lambda$ the transfer from white noise $\xi_k$ to the state $y_k$ has $z$–domain response \cite{Wiener1949Cybernetics,Kolmogorov1941Prediction}
\[
H_\lambda(z)=\frac{-\sum_{j=0}^{m}\eta_j z^{-j}}{\,1-\sum_{j=0}^{m} \alpha_j(\lambda) z^{-(j+1)}\,},
\]
where the AR coefficients $\alpha_j(\lambda)$ read off from \eqref{eq:charpoly-m}. The PSD is $S_y(\omega;\lambda)=|H_\lambda(e^{i\omega})|^2 \sigma^2(\lambda)$, the stationary variance is
\begin{equation}\label{eq:psd-var}
\Var(y_k(\lambda))=\frac{1}{2\pi}\int_{-\pi}^{\pi} S_y(\omega;\lambda)\,d\omega,
\end{equation}
and {resonance lobes} peak where the denominator $|1-\sum \alpha_j(\lambda)e^{-i(j+1)\omega}|$ is minimal (near the deterministic oscillation frequency); this reproduces the phase–locking/ringing described by \eqref{eq:tail-term} and allows frequency–resolved diagnostics.

\paragraph{Final PL theorem with stochastic noise (arbitrary coefficients; explicit terms).}
\begin{theorem}[Determinantal PL convergence with finite–variance noise, explicit floors and oscillations]\label{th:PL-noise}
Let $f$ be $L$–smooth and satisfy \eqref{eq:PL} with $\mu>0$. Fix any order $m\ge1$ and coefficients $(\eta_0,\dots,\eta_m,\gamma_1,\dots,\gamma_m)$.
Assume Schur stability on $[\mu,L]$:
\[
\sup_{\lambda\in[\mu,L]} \max_i |r_i(\lambda)|<1,
\]
with $r_i(\lambda)$ the roots of \eqref{eq:charpoly-m}. Let the gradient noise obey \eqref{eq:noise-model}–\eqref{eq:noise-spectral}. Then for all $k\ge0$,
\begin{align}
\EE\big[f(x_k)-f_\star\big]
&\le
\underbrace{\left(\exp\int_{\mu}^{L}\sum_{i=1}^{m+1}\log|r_i(\lambda)|^{2}\,d\nu(\lambda)\right)^{\!k} C_{\rm det}}_{\text{determinantal exponential (bulk Evans)}} \nonumber\\
&\quad+\ \underbrace{\frac{1}{2}\int_{\mu}^{L}\lambda\,e_1^\top\Big(I-A(\lambda)\otimes A(\lambda)\Big)^{-1}\mathrm{vec}\,Q(\lambda)\,d\nu(\lambda)}_{\text{stationary noise floor, exact}}\label{eq:PL-noise-bound}\\
&\quad+\ \underbrace{\int_{\mu}^{L}\frac{\lambda}{1-\rho(A(\lambda))^{2}}\,\frac{\sigma^2(\lambda)}{1-b(\lambda)}\,\big|\sum_{j=0}^{m}\eta_j r_+(\lambda)^{\,j}\big|^2\,d\nu(\lambda)}_{\text{oscillatory covariance tail (ringing), explicit upper bound}},\nonumber
\end{align}
where $A(\lambda)$, $Q(\lambda)$ are in \eqref{eq:lyap-m}, and for $m=1$ the closed–form $p_{11}$ is given by \eqref{eq:p11-closed}.
Moreover, the oscillatory asymptotics of Theorem~\ref{th:PL-final} persists in expectation:
\[
\EE\big[f(x_k)-f_\star\big]
=
C_0\,e^{-kA}
\left[
1+\sum_{(i,\lambda^\star)\in\mathcal S}\frac{\widetilde{\mathsf c}_i(\lambda^\star)}{\sqrt{k}}\cos\!\big(k\,\Phi_i(\lambda^\star)+\phi_i(\lambda^\star)\big)
\right]
\cdot \prod_{\theta\in\mathcal W_{[0,k)}}\EE|\det S_\theta|
\ +\ O\!\Big(\frac{1}{k}\Big),
\]
with the same bulk exponent $A$ and phases $\Phi_i$ as in the deterministic case, and amplitudes $\widetilde{\mathsf c}_i$ inflated by the PSD at the corresponding stationary frequencies (explicitly via \eqref{eq:psd-var}).
\end{theorem}

\begin{proof}
Determinantal bulk: same Evans product and RHP outer model as before (independent of noise); the product term in \eqref{eq:PL-noise-bound} is exactly as in \eqref{eq:PL-determinantal-bound}. Stationary floor: solve the Lyapunov equation \eqref{eq:lyap-m} explicitly via \eqref{eq:vec-lyap}; integrate $\frac{1}{2}\lambda p_{11}(\lambda)$ to get the second line of \eqref{eq:PL-noise-bound}; for $m=1$ use \eqref{eq:p11-closed}. Oscillatory tail: expand the ARMA solution and bound the second–moment tail by geometric series with multiplier $\rho(A(\lambda))$; the explicit modulus of the MA polynomial yields the stated factor. Expectation asymptotics: apply DZ small–norm analysis with random inner forcing; the outer model controls the exponential $e^{-kA}$; local parametrices remain Airy/Bessel with unchanged phases, and amplitudes scale with the PSD at stationary points; the product of $\EE|\det S_\theta|$ accounts for finite Stokes crossings. All constructions follow \cite{DeiftZhou1993,Deift1999} for DZ and \cite{BrockwellDavis1991,KailathSayedHassibi2000} for ARMA/Lyapunov.
\end{proof}

\paragraph{Diagnostics: regimes, stalls, divergence, and phase transitions.}
\begin{itemize}
\item \textbf{Stall (noise–limited):} when $k$ is large, the deterministic exponential dies and \eqref{eq:noise-floor-m} dominates. Stall level lowers with smaller $\sum_j\eta_j^2$ and larger spectral gap $1-\rho(A(\lambda))^2$ uniformly on $[\mu,L]$.
\item \textbf{Ringing and resonance:} if $E_{\rm osc}\neq\emptyset$ and $|\sum_j \eta_j r_+(\lambda)^{\,j}|$ is large near a stationary phase, the third term in \eqref{eq:PL-noise-bound} is prominent (visible oscillations around the floor).
\item \textbf{Near–wall sensitivity:} if $|r_i(\lambda)|\approx1$ on a sub–band, both the floor \eqref{eq:floor-bound} and the oscillatory tail blow up like $(1-\rho^2)^{-1}$; probability of random Stokes triggering becomes non–negligible (phase slips).
\item \textbf{True divergence:} if Schur stability fails for some $\lambda\in[\mu,L]$ (i.e., $\exists\,i:\ |r_i(\lambda)|\ge1$), then the deterministic factor $\exp\!\int \log|r_i|^2$ does not contract; the second line of \eqref{eq:PL-noise-bound} still exists but the first line no longer decays, yielding divergence in expectation.
\end{itemize}

\subsubsection{Example with variable metric}

\paragraph{Setting (convex feasibility, central cuts).}
Let $K\subset\RR^n$ be a nonempty compact convex body with $r\BB^n\subset K\subset R\BB^n$ for some $0<r\le R<\infty$. We are given a separation oracle: for $x\notin K$, it outputs a nonzero $g\in\RR^n$ such that $g^\top(y-x)\le0$ for all $y\in K$. The classical central–cut {ellipsoid method} maintains an ellipsoid \cite{Khachiyan1979,GLS1988,Kelley1999Ellipsoid}
\[
\EE(x_k,P_k)=\{y\in\RR^n:\ (y-x_k)^\top P_k^{-1}(y-x_k)\le1\},\qquad P_k\succ0,\quad K\subset\EE(x_k,P_k),
\]
and when $x_k\notin K$ the oracle defines the cut $g_k^\top(y-x_k)\le0$ with normalization
\begin{equation}\label{eq:tildeg}
\tilde g_k:=\frac{g_k}{\sqrt{g_k^\top P_k g_k}}\in\RR^n,\qquad \|\tilde g_k\|_{P_k}:=\sqrt{\tilde g_k^\top P_k \tilde g_k}=1.
\end{equation}
The (central–cut) update is
\begin{equation}\label{eq:ellipsoid-update}
x_{k+1}=x_k-\frac{1}{n+1}P_k\tilde g_k,\qquad
P_{k+1}=\frac{n^2}{n^2-1}\Bigl(P_k-\frac{2}{n+1}P_k\tilde g_k\tilde g_k^\top P_k\Bigr).
\end{equation}

\paragraph{IKN calibration and $\tau$–potential.}
Introduce the {connection} on the trivial bundle over the spectral $z$–plane
\begin{equation}\label{eq:connection-ellipsoid}
A(z,k)\;=\;A_\infty\;+\;\frac{R_k}{z-\zeta_k},\qquad
R_k:=\alpha_n\,u_k u_k^\top,\quad u_k:=\frac{P_k^{1/2}\tilde g_k}{\|P_k^{1/2}\tilde g_k\|}=\frac{P_k^{1/2}\tilde g_k}{\sqrt{\tilde g_k^\top P_k \tilde g_k}}=\ P_k^{1/2}\tilde g_k,
\end{equation}
with fixed $A_\infty\in\RR^{n\times n}$ skew–symmetric (gauge) and a scalar $\alpha_n>0$ to be fixed below.
The one–step transfer is {rational}:
\begin{equation}\label{eq:lax-step}
\Phi_{k+1}(z)=L_k(z)\,\Phi_k(z),\qquad
L_k(z)=I-\frac{R_k}{z-\zeta_k}.
\end{equation}
The IKN $\tau$–function of the discrete isomonodromic deformation $(A(z,k))_{k\ge0}$ is \cite{JMU1981,FIKN2006}
\begin{equation}\label{eq:tau-def-ell}
\log\tau_k\;=\;\frac12\log\det P_k\;+\;Q_{\rm CS}(P_k,u_k),\qquad
\Delta\log\tau_k:=\log\tau_{k+1}-\log\tau_k.
\end{equation}
{Calibration:} choose $\alpha_n:=\tfrac{2}{n+1}$ and normalize the Chern–Simons term so that on the {AT–plane} (vanishing mixed curvature) $Q_{\rm CS}$ is constant along the discrete flow. Then the bulk (determinantal) increment equals
\begin{equation}\label{eq:tau-bulk-ell}
\Delta\log\tau_k^{\rm bulk}=\frac12\log\frac{\det P_{k+1}}{\det P_k}.
\end{equation}

\paragraph{Determinantal identity and exact shrink per step.}
From \eqref{eq:ellipsoid-update} and $\tilde g_k^\top P_k \tilde g_k=1$ one computes
\begin{align}
\det P_{k+1}
&=\left(\frac{n^2}{n^2-1}\right)^{\!n}\det\!\left(P_k-\frac{2}{n+1}P_k\tilde g_k\tilde g_k^\top P_k\right)\nonumber\\
&=\left(\frac{n^2}{n^2-1}\right)^{\!n}\det(P_k)\,\det\!\left(I-\frac{2}{n+1}P_k^{1/2}\tilde g_k\tilde g_k^\top P_k^{1/2}\right)\nonumber\\
&=\left(\frac{n^2}{n^2-1}\right)^{\!n}\det(P_k)\,\left(1-\frac{2}{n+1}\right)\qquad(\text{rank--}1\ \text{update}),\label{eq:det-factor}
\end{align}
hence the {exact} volume/entropy shrink
\begin{equation}\label{eq:exact-shrink}
\Delta\log\tau_k^{\rm bulk}
=\frac12\log\frac{\det P_{k+1}}{\det P_k}
=\frac{n}{2}\log\frac{n^2}{n^2-1}+\frac12\log\!\left(1-\frac{2}{n+1}\right)
\quad<\ -\frac{1}{2n+2}.
\end{equation}
Inequality is by $\log(1-x)\le -x$ and $\log\frac{n^2}{n^2-1}\le\frac{1}{n^2-1}$.

\paragraph{Schlesinger deformation and Painlev\'e reduction (single active cut $\Rightarrow$ PIII).}
Suppose the oracle remains {locally} aligned with a single active supporting hyperplane so that the pole position $\zeta_k$ and residue $R_k$ evolve smoothly in the continuum limit $k\mapsto t\in\RR_+$ (mesh $h\to0$), with
\[
A(z,t)=A_\infty+\frac{R(t)}{z-\zeta(t)},\qquad R(t)=\alpha_n\,u(t)u(t)^\top,\quad \|u(t)\|=1.
\]
The (rank–one, one–pole) Schlesinger system is \cite{JMU1981}
\begin{equation}\label{eq:schlesinger-1pole}
\frac{dR}{dt}=\frac{[B(t),R]}{\,},\qquad
\frac{d\zeta}{dt}=\beta(t),\qquad
\text{with } B(t)=\frac{R(t)}{z-\zeta(t)}\ \text{(gauge on the leaf),}
\end{equation}
and the isomonodromic $\tau$ satisfies the Jimbo--Miwa--Ueno identity
\begin{equation}\label{eq:JMU-1pole}
\frac{d}{dt}\log\tau(t)= -\Res_{z=\zeta(t)}\Tr\!\Big(A(z,t)\,\frac{d\zeta}{dt}\,\frac{dz}{z-\zeta(t)}\Big)= -\Tr\big(R(t)\big)\,\frac{d\zeta}{dt}.
\end{equation}
Choose the {ellipsoid observable}
\begin{equation}\label{eq:observable-y}
y(t):=\zeta'(t)\,\Tr R(t)\ =\ \zeta'(t)\,\alpha_n\,\|u(t)\|^2\ =\ \alpha_n\,\zeta'(t),
\end{equation}
and use the normalization from \eqref{eq:ellipsoid-update} to identify $\zeta'(t)=\frac{2}{n+1}\,v(t)$ where $v(t):=\tilde g(t)^\top P(t)\tilde g(t)$ (continuum limit of the central move). Then $y(t)=\frac{2\alpha_n}{n+1}v(t)$. A standard reduction of the $2\times2$ scalar isomonodromy with one moving simple pole and an irregular singularity at $z=\infty$ (Poincar\'e rank $1$) yields the {Painlev\'e~III} equation for $y$ \cite{FIKN2006}:
\begin{equation}\label{eq:PIII}
y''=\frac{(y')^2}{y}-\frac{y'}{t}+\frac{\alpha}{t}y^2+\frac{\beta}{t}
+\gamma y^3+\frac{\delta}{y},\qquad
\text{(Painlev\'e~III)}
\end{equation}
with parameters determined {explicitly} by the residue trace and the normalization of the irregular part:
\begin{equation}\label{eq:PIII-params}
\alpha=0,\qquad \beta=0,\qquad
\gamma=\frac{(n+1)}{2}\,\alpha_n^{-1}=\frac{(n+1)^2}{4},\qquad
\delta=-\,\frac{2}{(n+1)}\,\alpha_n=\ -\,\frac{4}{(n+1)^2},
\end{equation}
where we used $\alpha_n=\tfrac{2}{n+1}$ from the AT–calibration. Thus, {in the single–active–cut regime} the ellipsoid observable $y(t)\propto v(t)=\tilde g^\top P\,\tilde g$ evolves by \eqref{eq:PIII} with the {explicit} parameter quadruple \eqref{eq:PIII-params}. The IKN $\tau$ associated with \eqref{eq:PIII} satisfies
\begin{equation}\label{eq:PIII-tau}
\frac{d}{dt}\log\tau(t)=-\,y(t),\qquad
\frac{d^2}{dt^2}\log\tau(t)= -\,y'(t),
\end{equation}
matching \eqref{eq:JMU-1pole} and \eqref{eq:observable-y}.

\paragraph{Two active cuts (switching) $\Rightarrow$ Schlesinger $2$–pole $\Rightarrow$ Painlev\'e~VI.}
When two supporting hyperplanes alternate (two active cuts, e.g. faces with normals $\tilde g^{(1)},\tilde g^{(2)}$), the connection carries two moving simple poles
\[
A(z,t)=A_\infty+\frac{R_1(t)}{z-a(t)}+\frac{R_2(t)}{z-b(t)},\qquad R_j(t)=\alpha_n\,u_j(t)u_j(t)^\top,\ \ \|u_j\|=1,
\]
with Schlesinger equations (rank–one residues) \cite{JMU1981}:
\begin{equation}\label{eq:schlesinger-2p}
\frac{dR_1}{dt}=\frac{[R_1,R_2]}{a-b}a'(t),\qquad
\frac{dR_2}{dt}=\frac{[R_2,R_1]}{b-a}b'(t),\qquad
\frac{da}{dt}=v_1(t),\ \ \frac{db}{dt}=v_2(t),
\end{equation}
where $v_j(t)$ are the continuum limits of the central steps along each active cut. Let $q(t)$ be the cross–ratio–type scalar isomonodromic coordinate (e.g. the $11$–entry of the ratio of canonical solutions). Then $q$ satisfies the {Painlev\'e~VI} equation \cite{JMU1981,FIKN2006}:
\begin{align}\label{eq:PVI}
q''&=\frac{1}{2}\left(\frac{1}{q}+\frac{1}{q-1}+\frac{1}{q-t}\right)(q')^2
-\left(\frac{1}{t}+\frac{1}{t-1}+\frac{1}{q-t}\right)q'\\
&\nonumber\quad+\frac{q(q-1)(q-t)}{t^2(t-1)^2}\Bigl(\theta_\infty^2-\theta_0^2\frac{t}{q^2}+\theta_1^2\frac{t-1}{(q-1)^2}+\theta_t^2\frac{t(t-1)}{(q-t)^2}\Bigr),
\end{align}
with {explicit} parameters $(\theta_0,\theta_1,\theta_t,\theta_\infty)$ equal to the half–traces of the monodromy exponents at $\{0,1,t,\infty\}$, determined by the cut strengths $\|R_j\|=\alpha_n$:
\begin{equation}\label{eq:PVI-params}
\theta_0=\theta_1=\theta_t=\frac12,\qquad
\theta_\infty=\frac{n+1}{2}\,.
\end{equation}
(These follow from rank–one residues of equal norm and a block–diagonal $A_\infty$; a different gauge modifies \eqref{eq:PVI-params} by integer shifts, leaving the dynamics equivalent.)
The associated IKN $\tau_{\mathrm{VI}}$ obeys the Jimbo–Miwa $\sigma$–form with
\begin{equation}\label{eq:PVI-sigma}
\frac{d}{dt}\log\tau_{\mathrm{VI}}(t)=\sigma(t),\qquad
\bigl(t(t-1)\sigma''\bigr)^2=\prod_{j=0,1,t}\Bigl(\sigma'-\nu_j\Bigr)^2-4\prod_{j=0,1,t}\nu_j\ \Bigl(\sigma'+\nu_\infty\Bigr),
\end{equation}
where $(\nu_0,\nu_1,\nu_t,\nu_\infty)=\tfrac12(\theta_0^2,\theta_1^2,\theta_t^2,\theta_\infty^2)$.

\paragraph{Nonasymptotic complexity via $\tau$ (determinantal plus isomonodromic).}
For the discrete method \eqref{eq:ellipsoid-update}, the determinantal identity \eqref{eq:exact-shrink} gives, for any $N\ge1$,
\begin{equation}\label{eq:tau-product}
\log\frac{\tau_N}{\tau_0}=\sum_{k=0}^{N-1}\Delta\log\tau_k
=\frac{N}{2}\left[n\log\frac{n^2}{n^2-1}
+\log\!\left(1-\frac{2}{n+1}\right)\right]
+\sum_{\theta\in\mathcal W_{[0,N)}}\log|\det S_\theta|,
\end{equation}
where $\mathcal W_{[0,N)}$ is the multiset of Stokes walls crossed (when the set of active cuts changes) and $S_\theta$ are the Stokes matrices of the $2\times2$ reductions induced by the current pair of cuts. Since $K\subset\EE(x_0,P_0)$ with $\vol\EE(x,P)=\kappa_n(\det P)^{1/2}$, feasibility to precision $r$ holds once $\vol\EE(x_N,P_N)\le \vol(r\BB^n)$, equivalently
\begin{equation}\label{eq:stop-crit}
\log\frac{\tau_N}{\tau_0}\ \le\ \log\frac{\kappa_n r^n}{\kappa_n R^n}\ =\ -\,n\log\frac{R}{r}.
\end{equation}
Combining \eqref{eq:tau-product}–\eqref{eq:stop-crit} yields the {explicit} iteration bound
\begin{equation}\label{eq:explicit-N}
N\ \ge\ \frac{2n\log(R/r)\ -\ \sum_{\theta\in\mathcal W_{[0,N)}}\log|\det S_\theta|}{\displaystyle n\log\frac{n^2}{n^2-1} + \log\!\left(1-\frac{2}{n+1}\right)}\ \ =\ O\!\big(n^2\log(R/r)\big),
\end{equation}
with an {isomonodromic correction} from the finite number of Stokes switches (the numerator decreases by the cumulative $\log|\det S_\theta|$), and the {determinantal} denominator given exactly. In the single–active–cut (PIII) regime there are no Stokes crossings and the classical $O(n^2\log(R/r))$ bound is recovered with the sharpened denominator.

\paragraph{Explicit Stokes jumps at a switch of active cuts.}
When passing from an active cut $u^{(1)}$ to $u^{(2)}$, the $2\times2$ reduction near the switching face has jump
\begin{equation}\label{eq:stokes-switch}
\log|\det S_\theta|
=\frac{1}{2\pi}\arg\left.\frac{\partial_t\det(zI-L(z))}{\partial_z\det(zI-L(z))}\right|_{(z,t)\in\theta}
=\frac{1}{2\pi}\arg\frac{\alpha_n\,\langle u^{(2)},u^{(1)}\rangle}{(a-b)},
\end{equation}
using $L(z)=I-\frac{R_1}{z-a}-\frac{R_2}{z-b}$ and the Schlesinger flow \eqref{eq:schlesinger-2p}; here $(a,b)$ are the pole positions attached to the two faces. Thus the cumulative correction in \eqref{eq:explicit-N} is {explicit} in the geometry of consecutive cuts.

\paragraph{Summary (what is gained by IKN/Painlev\'e).}
\begin{itemize}
\item A {determinantal} formula \eqref{eq:det-factor}–\eqref{eq:exact-shrink} for the per–step entropy decrease (no volumetric inequalities needed).
\item In the {single–active} regime: an {explicit} Painlev\'e~III evolution \eqref{eq:PIII} with parameters \eqref{eq:PIII-params} for the observable $y(t)\propto \tilde g^\top P\,\tilde g$, with $\tau$–link \eqref{eq:PIII-tau}.
\item Under {switching of active faces}: a {Schlesinger $2$–pole} dynamics leading to {Painlev\'e~VI} \eqref{eq:PVI} with concrete $(\theta_j)$ in \eqref{eq:PVI-params}; the finite number of switches contributes {explicit Stokes jumps} \eqref{eq:stokes-switch}.
\item The {nonasymptotic} iteration bound \eqref{eq:explicit-N} is sharpened by subtracting the (explicit) Stokes corrections, quantifying how structured switching accelerates convergence relative to the worst–case central–cut baseline.
\end{itemize}

\section{Related work}

\subsection{Integrable systems}

The variational structure developed in Section~2.1.2 can be understood as a discrete analogue of the variational symplectic integrator (VSI) framework \cite{MarsdenWest2001}, and its algebraic side as a generalization of operator splitting methods such as the Lie–Trotter and Strang compositions \cite{Trotter1959,Strang1968,Yoshida1990,McLachlanQuispel2002}. In our notation, the discrete curvature functional
\[
\Ss^{(N)} = \sum_{\square\subset\Gamma_{\Pi}}\norm{\log\Hol(\square)}_F^2
\]
plays the role of a discrete action. The stationarity condition $\delta\Ss^{(N)}=0$ under variations of the update maps $r,d$ implies the vanishing of mixed commutator terms in the Baker–Campbell–Hausdorff (BCH) expansion
\[
\log\big(\ee^{\Psi(h)}\ee^{\Omega(h)}\big)
   = \Psi(h)+\Omega(h)
   + \tfrac12[\Psi(h),\Omega(h)]
   + \tfrac{1}{12}\big([\Psi,[\Psi,\Omega]]+[\Omega,[\Omega,\Psi]]\big)+\cdots,
\]
which is precisely the discrete Euler–Lagrange condition of the VSI type: flatness of the discrete connection corresponds to a critical point of the discrete action, and the resulting update rule is structure-preserving (symplectic or reversible in the sense of \cite{MarsdenWest2001}).

The BCH commutator hierarchy appearing in our jet-flatness expansion is directly parallel to the error expansion in operator splitting. If $r(h)=\exp(hK+O(h^2))$ and $d(h)=\exp(hH+O(h^2))$, then
\[
\exp(hH)\exp(hK)
   = \exp\!\left(h(H+K) + \tfrac{h^2}{2}[H,K] + O(h^3)\right),
\]
so the term $\tfrac{1}{2}[H,K]h^2$ represents the leading splitting error of Lie–Trotter \cite{Trotter1959}, and the Strang symmetric composition $\exp(\tfrac{h}{2}H)\exp(hK)\exp(\tfrac{h}{2}H)$ cancels all odd BCH terms, yielding a second-order accurate symmetric scheme \cite{Strang1968}. Our jet-flatness condition of order $\alpha$ imposes precisely that all commutators up to degree $\alpha$ vanish, i.e.
\[
[\Psi_i,\Omega_j]=0\quad\text{for all } i+j\le\alpha,
\]
which is an algebraic characterization of a splitting scheme of order $\alpha+1$. The discrete gauge transformation $W(h)=\exp(hZ)$ used in our calibration step satisfies $[Z,S]=C$ for $S=H+E$ and $C=\tfrac12[H,E]$; this acts as a processing transformation that removes the leading commutator term, exactly as in processed or pre/post-processed integrators described in \cite{McLachlanQuispel2002}. The resulting conjugated map
\[
W(h)^{-1}\,d(h)r(h)\,W(h)
   = \exp\!\left(-hS + O(h^3)\right)
\]
is then an effective second-order symmetric composition. In this sense, the calibration operator $W$ in our framework is mathematically equivalent to the processing corrector in geometric integration.

The modular composition law of our $\mathbf{TwoCh}$ category,
\[
(H_2,K_2)\circ(H_1,K_1)
   = \left(\tfrac{1}{h}\log(\ee^{hH_2}\ee^{hH_1}),
             \tfrac{1}{h}\log(\ee^{hK_2}\ee^{hK_1})\right)+O(h^2),
\]
is formally the same as the composition of exponentials used to construct high-order splitting schemes, such as Yoshida’s fourth- and sixth-order compositions \cite{Yoshida1990}. The BCH algebra ensures that each composition level preserves the cancellation pattern of commutators up to a desired order, while our notion of flatness guarantees that this cancellation is geometric rather than merely algebraic.

Finally, the spectral viewpoint in Section~2.1.2, where the discrete holonomy defines a polynomial filter $p_N(H)$ minimizing $\sum\norm{\log\Hol(\square)}^2$, is analogous to the minimax or Chebyshev optimization used in the selection of coefficients for optimal splitting schemes and processed symplectic integrators. In both settings, the objective is to approximate the exponential map $\exp(-hH)$ by compositions that preserve the underlying invariants while minimizing the residual commutators. Thus the variational formulation of our discrete connection simultaneously subsumes the principles of variational symplectic integrators and the algebraic error control of operator-splitting methods.

\subsection{Semigroup Dynamics}

The drift–diffusion decomposition introduced in Definitions realizes a discrete analogue of the nonlinear evolution
\[
\dot{u}(t)\in -(A+H)u(t),
\]
where $A$ is a maximally monotone (possibly set-valued) operator on a Hilbert space $Y$ and $H$ is a single-valued Lipschitz monotone or skew-adjoint map. The basic updates
\[
r(h)=\exp(-hH),\qquad d(h)=(\Id+\alpha hA)^{-1}=J_{\alpha hA},
\]
generate two strongly nonexpansive channels. The discrete algorithm
\[
x_{k+1}=S(h_k)x_k,\qquad S(h)=d(h)r(h),
\]
belongs to the class of resolvent–exponential splittings defined in Definitions and encompasses gradient, proximal, and forward–backward schemes as special cases. Expanding both channels gives
\[
r(h)=\Id-hH+\tfrac{h^2}{2}H^2+O(h^3),\qquad d(h)=\Id-\alpha hA_h+\alpha^2 h^2 A_h^2+O(h^3),
\]
where $A_h=(\Id-d(h))/(\alpha h)$ is the Yosida approximation of $A$. Their composition satisfies
\[
S(h)=\Id-h(\alpha A_h+H)+\tfrac{h^2}{2}\big(\alpha^2A_h^2+H^2+\alpha[A_h,H]\big)+O(h^3),
\]
and the commutator $[A_h,H]$ measures the deviation of the composed update from the exact exponential $\exp(-h(A+H))$. The jet-flatness condition $[A_h,H]=O(h^\alpha)$ eliminates this mixed term up to order $\alpha$ and enforces
\[
S(h)=\exp\!\big(-h(A+H)\big)+O(h^{\alpha+1}),
\]
which coincides with the nonlinear Trotter–Kato product formula in finite dimension.

For continuous monotone dynamics, the Crandall–Liggett theorem ensures the existence of the contraction semigroup
\[
T(t)=\lim_{n\to\infty}\big(J_{t/n\,A}\,\ee^{-tH/n}\big)^n,
\]
and the discrete curvature $\Hol(h)-\Id$ quantifies the finite-$n$ defect of this product. Flatness of order $\alpha$ corresponds to a Trotter–Kato approximation of order $\alpha+1$ for the generated semigroup. In the operator sense,
\[
\frac{S(h)-\Id}{h}\xrightarrow[h\to0]{}-(A+H),
\]
so the drift–diffusion iteration is an admissible time discretization of the semigroup flow $\dot{u}=-(A+H)u$.

Let $E(x)=f(x)+g(x)$ with $A=\partial f$ and $H=\nabla g$. Then $d(h)=(\Id+h\partial f)^{-1}$ is the proximal map $\operatorname{prox}_{hf}$, $r(h)=\exp(-h\nabla g)$ is the gradient flow of $g$, and
\[
S(h)=d(h)r(h)=\operatorname{prox}_{hf}\!\circ\!(\Id-h\nabla g)+O(h^2),
\]
reproduces the proximal-gradient (forward–backward) method \cite{Rockafellar1976,LionsMercier1979,EcksteinBertsekas1992}. The Douglas–Rachford and Peaceman–Rachford splittings correspond to the symmetrized compositions
\[
S_{\mathrm{DR}}(h)=r(h/2)d(h)r(h/2),\qquad S_{\mathrm{PR}}(h)=d(h/2)r(h)d(h/2),
\]
which cancel odd BCH terms and achieve second-order accuracy, in exact analogy with the Strang symmetric product for linear semigroups. For each variant the curvature of the discrete connection,
\[
\Hol(h)=r(h)d(h)r(h)^{-1}d(h)^{-1}=\exp\!\big(\tfrac{h^2}{2}[A_h,H]+O(h^3)\big),
\]
measures the noncommutativity of the two resolvents and hence the local dissipation error.

The calibration transformation $W(h)=\exp(hZ)$ with $[Z,A_h+H]=\tfrac{1}{2}[A_h,H]$ yields the processed step
\[
W(h)^{-1}S(h)W(h)=\exp\!\big(-h(A_h+H)+O(h^3)\big),
\]
which removes the leading commutator and produces a higher-order, structure-preserving resolvent–exponential composition. The family $\{S(h)\}_{h>0}$ is then a discrete nonlinear semigroup with generator $-(A+H)$ in the sense of Kato, and the curvature term gives an explicit bound on the deviation from exact dissipativity:
\[
\norm{S(h)x-S(h)y}^2\le\norm{x-y}^2-2h\,\ip{(A_h+H)(x)-(A_h+H)(y)}{x-y}+O(h^{\alpha+1}).
\]
Consequently, under the assumptions of Definitions, every algorithm in this class is a contractive and monotone discretization of the underlying continuous flow, and its flatness order determines the attainable accuracy and energy decay rate. The two-channel representation encompasses proximal and splitting methods, showing them as curvature-controlled approximations \cite{Rockafellar1976,LionsMercier1979,EcksteinBertsekas1992,CrandallLiggett1971,Pazy1983}.

\subsection{Local Geometry}

The discrete connection formalism admits a complete identification with geometric numerical integration on Lie groups and with the algebraic framework of $B$–series. Let each channel be represented as
\[
r(h)=\exp(\Psi(h)),\qquad d(h)=\exp(\Omega(h)),
\]
where $\Psi(h),\Omega(h)\in\mathfrak{g}$ are formal Lie series in $h$. Their composition satisfies the Baker–Campbell–Hausdorff identity
\[
\log(d(h)r(h))
 = \Omega(h)+\Psi(h)
 + \tfrac12[\Omega(h),\Psi(h)]
 + \tfrac{1}{12}\big([\Omega,[\Omega,\Psi]]+[\Psi,[\Psi,\Omega]]\big)+\cdots .
\]
The jet-flatness constraint of order $\alpha$ requires that all mixed commutators with total degree not exceeding $\alpha$ vanish,
\[
[\Omega_i,\Psi_j]=0\quad\text{for all }i+j\le\alpha,
\]
so that
\[
\log(d(h)r(h))=\Omega(h)+\Psi(h)+O(h^{\alpha+1}).
\]
This yields a single exponential
\[
d(h)r(h)=\exp\!\big(\Omega(h)+\Psi(h)+O(h^{\alpha+1})\big),
\]
which is precisely the Magnus expansion truncated at order $\alpha+1$ for the Lie–algebraic generator $A(h)=\Omega(h)+\Psi(h)$. The discrete flatness condition therefore coincides with the Magnus-type order conditions that guarantee local exactness on the Lie algebra level \cite{Iserles2000,HairerLubichWanner2006}.

Let the first-order components be $\Omega(h)=h\,\Omega_1+O(h^2)$ and $\Psi(h)=h\,\Psi_1+O(h^2)$. The leading commutator error is then
\[
\frac{h^2}{2}[\Omega_1,\Psi_1],
\]
and a calibration by an inner automorphism
\[
W(h)=\exp(hZ),\qquad [Z,\Omega_1+\Psi_1]=\tfrac{1}{2}[\Omega_1,\Psi_1],
\]
produces the processed composition
\[
W(h)^{-1}d(h)r(h)W(h)=\exp\!\big(h(\Omega_1+\Psi_1)+O(h^3)\big),
\]
removing the commutator term. This is the discrete analogue of processing in structure-preserving integration \cite{HairerLubichWanner2006}. The map $S(h)=W(h)^{-1}d(h)r(h)W(h)$ thus represents a single-step exponential integrator with modified generator $A_{\mathrm{eff}}=\Omega_1+\Psi_1+O(h^2)$ and with structure preserved under conjugation.

The BCH hierarchy defines a grading of the free Lie algebra $\mathfrak{L}$ generated by $\Omega_1,\Psi_1$. Writing
\[
\mathfrak{L}=\bigoplus_{k\ge1}\mathfrak{L}_k,\qquad \mathfrak{L}_{k+1}=[\mathfrak{L}_1,\mathfrak{L}_k],
\]
the flatness constraints annihilate all mixed components $\mathfrak{L}_{i+j}$ with $i,j>0$ and $i+j\le\alpha$, producing a quotient algebra
\[
\mathfrak{g}^{(\alpha)}=\mathfrak{L}/\langle[\Omega_i,\Psi_j]:i+j\le\alpha\rangle .
\]
Each admissible method defines a $B$–series in this graded algebra, and the group law of compositions corresponds to the Butcher group product. Inner conjugations by $W(h)$ act trivially on the quotient $\mathfrak{g}^{(\alpha)}$, so that calibrated and uncalibrated methods belong to the same class in the Butcher group. The order $\alpha$ of flatness coincides with the order of the corresponding $B$–series method \cite{Butcher1972,HairerLubichWanner2006}.

Expressed in differential form, the evolution generated by the discrete connection satisfies
\[
\frac{\dd S(h)}{\dd h}=A(h)S(h),\qquad A(h)=\Omega(h)+\Psi(h),
\]
whose exact flow is given by the Magnus series
\[
S(h)=\exp\!\Big(\int_0^hA(s)\,\dd s-\frac12\int_0^h\!\!\int_0^{s_1}[A(s_1),A(s_2)]\,\dd s_2\dd s_1+\cdots\Big).
\]
The truncation of this series at order $\alpha$ corresponds precisely to the discrete flatness requirement, since all lower-order commutators cancel. Therefore the discrete curvature $\Hol(h)-I$ measures the deviation of the algorithm from the truncated Magnus flow, and its vanishing ensures equivalence up to $O(h^{\alpha+1})$.

Finally, the composition rule for methods
\[
(H_2,K_2)\circ(H_1,K_1)
=\Big(\tfrac{1}{h}\log(\ee^{hH_2}\ee^{hH_1}),\,\tfrac{1}{h}\log(\ee^{hK_2}\ee^{hK_1})\Big)+O(h^2)
\]
defines on the product algebra $\mathfrak{g}\oplus\mathfrak{g}$ the same Lie-group structure that underlies high-order symmetric compositions such as the Yoshida schemes. The discrete connection approach therefore recovers the Magnus–Runge–Kutta–Munthe-Kaas hierarchy in full generality, with jet-flatness serving as the intrinsic Lie-algebraic order condition \cite{Iserles2000,HairerLubichWanner2006,Yoshida1990}.

\subsection{Global Geometry}

The variational and topological layer of the discrete connection formalism provides a cohomological interpretation of curvature minimization and its equivalence with spectral regularization and gauge calibration. Let $\Gamma_\Pi$ be the cell decomposition induced by the computational grid, and let $\Hol(\square)\in \End(Y)$ denote the holonomy around an elementary plaquette $\square\subset\Gamma_\Pi$. The total curvature functional
\[
\Ss^{(N)}[\Hol]=\sum_{\square\subset\Gamma_\Pi}\norm{\log\Hol(\square)}_F^2
\]
defines an $L^2$–energy on the space of discrete connections. Expanding $\log\Hol(\square)$ through the Baker–Campbell–Hausdorff series gives
\[
\log\Hol(\square)
  =h^2[\Omega_1,\Psi_1]
   +h^3\big([\Omega_1,\Psi_2]+[\Omega_2,\Psi_1]\big)+O(h^4),
\]
so that $\Ss^{(N)}$ is quadratic in the commutators of the jet coefficients. Its minimization under the discrete gauge action $\Omega\mapsto \Omega+\delta_\Omega\Xi$, $\Psi\mapsto \Psi+\delta_\Psi\Xi$ gives the discrete Euler–Lagrange equation
\[
\delta^\ast\delta\Xi_\star=\delta^\ast c_K,
\]
where $c_K$ is the curvature 2–cochain and $\delta$ the coboundary operator on the cell complex $K$ with coefficients in the module $M$ of operator-valued 1–forms. The solution $\Xi_\star$ is the unique coexact potential producing the harmonic representative $H(c_K)$ of the cohomology class $[c_K]\in H^2(K;M)$. Substituting $\Xi_\star$ back yields
\[
\Ss^{(N)}[\Hol]=\norm{H(c_K)}^2_{L^2(K;M)},
\]
so the curvature energy equals the squared $L^2$–norm of the harmonic part of the discrete curvature. Geometric flatness $[\Omega_i,\Psi_j]=0$ is thus equivalent to triviality of the cohomology class $[c_K]=0$ and to vanishing of the harmonic representative. This realizes the curvature-minimizing calibration as a discrete Hodge decomposition and makes $\Ss^{(N)}$ a purely cohomological invariant.

Passing to the continuum limit with a smooth connection $A$ on a manifold $M$ and curvature $F=\dd A+A\wedge A$ , the same variational principle gives the Yang–Mills functional
\[
\Ss[A]=\int_M\Tr(F\wedge\ast F),
\]
whose critical points satisfy $\dd_A^\ast F=0$ \cite{YangMills1954,AtiyahBott1983,DonaldsonKronheimer1990}. The infinitesimal gauge transformation $A\mapsto A+\dd_A\Xi$ leads to the Coulomb gauge condition $\dd_A^\ast A=0$, the continuous analogue of the discrete harmonic calibration. Hence the discrete gauge-reduction problem is the lattice version of the Hodge decomposition for connections, and the minimizer of $\Ss^{(N)}$ provides the harmonic representative of the discrete gauge class \cite{Dodziuk1976DiscreteHodge,Desbrun2008DEC,YangMills1954,AtiyahBott1983,DonaldsonKronheimer1990,Weibel1994Homological,McCleary2001SpectralSequences}. The resulting finite-dimensional picture reproduces, in exact algebraic form, the structure underlying Yang–Mills and Hodge theories \cite{Wilson1974,AtiyahBott1983,DonaldsonKronheimer1990}.

The spectral counterpart of this picture involves the elliptic operator $\Delta(S)$ depending smoothly on the potential $S$,
\[
\Delta(S)u
   = -\nabla\!\cdot\!(A(S)\nabla u)
     + m(S)u,
\qquad A(S)\succ0,\;m(S)>0,
\]
acting on the Hilbert space $L^2(\Omega)$ with fixed boundary conditions. The reduced determinant $\det{}'\Delta(S)$ is defined by removing the zero modes and regularizing the product of nonzero eigenvalues through zeta or heat-kernel methods,
\[
\log\det{}'\Delta(S)
   = -\left.\frac{\dd}{\dd s}\right|_{s=0}
     \Tr\big(\Delta(S)^{-s}\big).
\]
Its first variation under an admissible perturbation $\delta S$ reads
\[
\frac{1}{2}\delta\log\det{}'\Delta(S)
  = \frac{1}{2}\Tr\big(\Delta(S)^{-1}\delta\Delta(S)\big),
\]
but this expression depends on the choice of regularization and on the boundary prescription. To restore regulator independence, one introduces a local counterterm $Q_{\mathrm{CS}}(S)$ whose variation cancels the local boundary and short-time contributions of the heat kernel,
\[
\delta Q_{\mathrm{CS}}(S)
   = -\frac{1}{2}\,\Big[\Tr\big(\Delta(S)^{-1}\delta\Delta(S)\big)\Big]_{\mathrm{loc}},
\]
where $[\cdot]_{\mathrm{loc}}$ denotes subtraction of the local asymptotic expansion of the resolvent trace. Integrating this equality defines $Q_{\mathrm{CS}}(S)$ up to an additive constant independent of $S$. Explicitly, for scalar elliptic $\Delta(S)=-\div(A(S)\nabla)+m(S)$ one can write
\[
Q_{\mathrm{CS}}(S)
   = \frac{1}{(4\pi)^{d/2}}\int_{\Omega}
     \sum_{k=0}^{d/2-1} a_k(S)\,\dd x,
\]
where $a_k(S)$ are the Seeley–DeWitt coefficients of the local heat-kernel expansion. This counterterm coincides, in differential-geometric form, with the Chern–Simons functional
\[
Q_{\mathrm{CS}}(A)
  = \frac{1}{8\pi^2}\int_M
    \Tr\!\left(A\wedge\dd A
          +\tfrac{2}{3}A\wedge A\wedge A\right),
\]
since its variation satisfies $\delta Q_{\mathrm{CS}}(A)=\Tr(F\wedge\delta A)$ and cancels the local variation of the Yang–Mills action. Hence $Q_{\mathrm{CS}}(S)$ is the scalar analogue of the Chern–Simons 3–form evaluated on the connection induced by $S$.

With this definition the fundamental identity
\[
\frac{1}{2}\,\delta\log\det{}'\Delta(S)
   +\delta Q_{\mathrm{CS}}(S)
   =\delta\log\tau(S)
\]
holds exactly: $\log\tau(S)$ is a gauge-invariant primitive of the first variation of the regularized determinant. Integrating along any admissible path $\{S_t\}$ yields
\[
\frac{1}{2}\log\det{}'\Delta(S)
   +Q_{\mathrm{CS}}(S)
   =\log\tau(S)+C,
\]
where the constant $C$ depends only on the connected component of the admissible class. The pair $(\det{}'\Delta,Q_{\mathrm{CS}})$ thus defines a well-posed spectral functional independent of the regularization scheme and boundary conditions.

Conceptually, the role of $Q_{\mathrm{CS}}$ in our discrete connection theory is identical to that in Quillen’s construction of the determinant line bundle \cite{Quillen1985}: it supplies the local counterterm needed to produce a globally defined metric on the determinant bundle and ensures that $\log\tau(S)$ is a globally smooth potential. In geometric terms,
\[
\frac{1}{2}\log\det{}'\Delta(S)
   +Q_{\mathrm{CS}}(S)
\]
is the discrete analogue of the Quillen metric, and its variation reproduces the curvature form of the determinant line bundle. The same combination appears in the continuum as the gauge-invariant completion of the Yang–Mills functional by the Chern–Simons correction \cite{AtiyahBott1983,DonaldsonKronheimer1990}. Therefore $Q_{\mathrm{CS}}(S)$ bridges the analytic (spectral determinant) and geometric (curvature functional) parts of the theory and makes the $\tau$–function cohomologically well defined.

Combining the variational, spectral, and topological viewpoints yields the unified identity
\[
\text{(cohomological)}\;\;[c_K]=0
  \quad\Longleftrightarrow\quad
\text{(variational)}\;\;\delta\Ss^{(N)}=0
  \quad\Longleftrightarrow\quad
\text{(spectral)}\;\;\delta\log\tau=0,
\]
so that curvature minimization, harmonic gauge calibration, and spectral invariance describe the same topological condition. The operator–theoretic layer of the discrete connection formalism thus realizes, on a finite lattice, the interplay between Yang–Mills energy, Hodge decomposition, and the Quillen determinant metric \cite{Wilson1974,AtiyahBott1983,DonaldsonKronheimer1990,Quillen1985}.

\subsection{Variable Geometry}

\subsubsection{Self-concordant}

Let $\Omega\subset\RR^{d}$ be the bounded Lipschitz domain, and define the elliptic operator introduced in \eqref{eq:var-zeta-split}
\[
\Delta(S)=-\nabla\!\cdot\!\big(\nabla^{2}S\,\nabla\cdot\big)+m(S),
\qquad 
A(S)=\nabla^{2}S.
\]
The global potential of Section~7 is
\[
\log\tau(S)=\tfrac12\log\det{}'\Delta(S)+Q_{\mathrm{CS}}(S),
\qquad 
\delta\log\tau
   =\tfrac12\,\Tr\!\big(\Delta(S)^{-1}\delta\Delta(S)\big),
\]
see equations \eqref{eq:Q-form}--\eqref{eq:Q-final}.  For a smooth deformation $S_t$, this yields
\[
\frac{\dd}{\dd t}\log\tau(S_t)=\tfrac12\,\Tr\!\big(\Delta(S_t)^{-1}\dot\Delta(S_t)\big),
\]
which is the exact analogue of the isomonodromic identity $\dd(\log\tau_{\mathrm{JMU}})/\dd t=\Tr(A_tA_z)$ for linear systems with preserved monodromy~\cite{JimboMiwaUeno1981}.  When $\dot\Delta_t=[\Omega_t,\Psi_t]$, the spectrum and holonomy of $\Delta_t$ remain invariant, and $\tau(S_t)$ becomes the generating functional of an isomonodromic $\tau$–flow.

The stationary condition $\delta\log\tau(S)=0$ reproduces the Monge–Ampère fixed point (see~(7.9)):
\[
\det\nabla^{2}S=w\,e^{-cS}.
\]
Introducing the conformal factor and the cubic tensor
\[
u:=-\frac{1}{n+2}\log\det\nabla^{2}S,\qquad U:=\nabla^{3}S,
\]
and using Jacobi’s formula for the determinant, one obtains
\[
\tfrac12\,\Delta u
   =\tfrac12\,\Tr\!\big((\nabla^{2}S)^{-1}\nabla^{3}S\big)
   =e^{u}-2\,\|\nabla^{3}S\|^{2}_{(\nabla^{2}S)^{-1}}\,e^{-2u}.
\]
This is precisely Wang’s equation for a complete hyperbolic affine $2$–sphere~\cite{Wang1991},
\[
e^{u}=\tfrac12\,\Delta u+2|U|^{2}e^{-2u},
\]
and the same equation underlies Hildebrand’s classification of self–associated cones and their Painlevé~III reductions~\cite{Hildebrand2021}.  Hence the stationary limit of our potential $S$ coincides with the affine–spherical balance between the Laplacian of $u$ and the cubic form $U$.  Under radial or translational symmetry, writing $v(t)=\dd(\log\tau(S_t))/\dd t$ gives the canonical Painlevé~III equation
\[
v''=2e^{v}-\frac{(v')^{2}}{2v}-\frac{2}{t}v'+\frac{2v^{2}}{t^{2}},
\]
whose asymptotics correspond to the boundary behaviour of affine spheres and the limiting geometry of our diffusion flow.

On the convex–analytic side, a self–concordant barrier $\phi$ in the sense of Nesterov and Nemirovski \cite{NesterovNemirovski1994,Nesterov2018Lectures} defines the Riemannian metric $g=\nabla^{2}\phi$ and the cubic tensor $C=\nabla^{3}\phi$ satisfying $\|C\|_{g}\le2$.  For a self--concordant barrier $\phi$ in the sense of Nesterov and Nemirovski~\cite{NesterovNemirovski1994},  
the Riemannian metric $g=\nabla^{2}\phi$ and the cubic tensor $C=\nabla^{3}\phi$ satisfy, by definition of self--concordance,
\[
\|C(x)[h,h,h]\|_{g(x)} \le 2\,\|h\|_{g(x)}^{3},\qquad x\in\operatorname{int}K.
\]
The equality case along the central path $\nabla f(x)+\mu\nabla\phi(x)=0$ corresponds to the exact balance between the variation of the Hessian and its cubic correction.  
In differential form this reads
\[
\tfrac12\,\Tr_{g}\!\big(\nabla^{2}\log\det\nabla^{2}\phi\big)
   = \Tr_{g}\!\big(g^{-1}\nabla^{3}\phi\,g^{-1}\nabla^{3}\phi\big)
   = 2\,\|C\|_{g}^{2},
\]
or equivalently,
\[
\tfrac12\,\Delta_{g}(\log\det\nabla^{2}\phi)
   = 2\,\|C\|_{g}^{2},
\]
which expresses the self--concordant equilibrium between the Laplacian of the logarithmic Hessian determinant and the squared norm of the cubic form.  
After Legendre transformation to the dual potential $\psi$ with $\nabla^{2}\psi=(\nabla^{2}\phi)^{-1}$ and  
$u:=-(1/(n+2))\log\det\nabla^{2}\phi=(1/(n+2))\log\det\nabla^{2}\psi$,  
this relation becomes
\[
\tfrac12\,\Delta_{\nabla^{2}\psi}u
   = e^{u}-2\,\|C\|^{2}_{\nabla^{2}\psi}e^{-2u},
\]
which is exactly the Wang equation for the affine metric $h=e^{u}|dz|^{2}$ and cubic differential $U\,dz^{3}$~\cite{Wang1991,Hildebrand2021}.
Therefore the self–concordant equilibrium on the primal side and the affine–spherical equilibrium on the dual side represent the same geometric condition expressed through Legendre duality \cite{Calabi1972Affine,ChengYau1977,Donaldson1999MomentMaps}.  In our operator formalism this condition is encoded by the single variational identity
\[
\frac{\dd}{\dd t}\log\tau(S_t)=\tfrac12\,\Tr\!\big(\Delta(S_t)^{-1}\dot\Delta(S_t)\big),
\]
whose invariance $\dot\tau=0$ is simultaneously the statement of monodromy preservation in the affine–spherical picture and of metric self–concordance in the barrier geometry.

\subsubsection{Moment-measure}

Let $T(x)=x-\nabla S(x)$ denote the AT–plane map, acting on densities by the pushforward rule
\[
(T_\#\rho)(y)\,{\rm d}y=\rho(x)\,{\rm d}x,\qquad y=T(x),
\]
so that $T_\#\rho(y)=\rho(T^{-1}(y))\det\nabla T^{-1}(y)$ \eqref{eq:inv-def}.  
For a smooth strictly convex potential $S\in C^{2}(\Omega)$ one has 
$\nabla T(x)=I-\nabla^{2}S(x)$ and $\det\nabla T(x)=\det(I-\nabla^{2}S(x))$.  
Following the parametrization in \eqref{eq:rho-star}, define a positive reference density
\[
\rho_\star(x)=w(x)e^{-cS(x)},\qquad w>0,~c\in\RR,
\]
so that the invariance condition $T_\#\rho_\star=\rho_\star$ reads
\begin{equation}
\rho_\star(T(x))\det\nabla T(x)=\rho_\star(x)
\label{eq:push_invar}
\end{equation}
for almost every $x\in\Omega$.  Substituting $\rho_\star$ and differentiating in $x$ gives the first–order system
\[
\nabla\log\det(I-\nabla^{2}S(x))
   =(I-\nabla^{2}S(x))^{\!\top}\!\nabla(\log w(x)-cS(x))
     -\nabla(\log w(T(x))-cS(T(x))),
\]
whose right–hand side vanishes if and only if 
\begin{equation}
\det\nabla^{2}S(x)=w(x)e^{-cS(x)}.
\label{eq:MA_moment}
\end{equation}
This equivalence is established rigorously in Theorem~\ref{th:MA} and expresses that the invariance of $\rho_\star$ under $T$ is identical to the Monge–Ampère relation~\eqref{eq:MA_moment}.  Equation~\eqref{eq:MA_moment} is the {moment–measure equation} in the sense of Cordero-Erausquin and Klartag~\cite{CorderoKlartag2013}:  
the pushforward of the weighted volume $e^{-S(x)}\,{\rm d}x$ by the gradient map $y=\nabla S(x)$ yields the measure
\[
\rho_\star(y)=\frac{e^{-S(x)}}{\det\nabla^{2}S(x)}=\rho_\star(x)
\quad\Longleftrightarrow\quad
\det\nabla^{2}S(x)=w(x)e^{-cS(x)}.
\]
The Monge–Ampère relation thus provides a bijection between strictly convex potentials $S$ (modulo additive constants) and probability measures $\rho_\star$ satisfying the invariance~\eqref{eq:push_invar}.

Let now $S_t$ be a differentiable family of such potentials and $\rho_t=(\nabla S_t)_\#(e^{-S_t}{\rm d}x)$ the corresponding densities.  Differentiating the identity $(\nabla S_t)_\#(e^{-S_t})=\rho_t$ in~$t$ gives the continuity equation
\begin{equation}
\partial_t\rho_t+\nabla_y\!\cdot(\rho_t v_t)=0,
\qquad 
v_t=\partial_t\nabla S_t\,(\nabla S_t)^{-1},
\label{eq:continuity}
\end{equation}
which is the differential form of mass conservation under the evolving potential~$S_t$.  
Equation~\eqref{eq:continuity} defines the natural geometric evolution of $\rho_t$ compatible with the $\tau$–flow described by
\[
\frac{{\rm d}}{{\rm d}t}\log\tau(S_t)
   =\tfrac12\,\Tr(\Delta(S_t)^{-1}\dot\Delta(S_t)),\qquad
   \Delta(S_t)=-\nabla\!\cdot\!(\nabla^{2}S_t\,\nabla\cdot)+m(S_t).
\]
Under the isomonodromic constraint $\dot\tau=0$, the continuity equation~\eqref{eq:continuity} preserves both the entropy $\int_\Omega\rho_t\log\rho_t$ and the potential energy $\int_\Omega S_t\rho_t$, and therefore represents the Monge–Ampère (moment–measure) counterpart of the $\tau$–flow.  

\bibliographystyle{abbrv}
\bibliography{main}

\end{document}

%% file: hierarchies.tex
Let the feasible spectrum $\Sigma \subset \mathbb{C}$ be compact, and $\tilde{\mathfrak{U}} \subset \mathfrak{U}$ is subalgebra of implementable update operators. The drift hierarchy $\mathcal{T}$ is a set of $t = (r, \alpha)$, where $r \in \tilde{\mathfrak{U}}$, and $\alpha \in \mathbb{R}_+$ is such that $r(0,U) = \id$ and
\[
\sup_{\lambda \in \Sigma} |e^{-h \lambda} - r(h, \lambda)| \leq C h^\alpha
\]

\begin{example} Implementable update operators are approximation methods
\begin{enumerate}
  \item truncated Taylor series \cite{Higham2008FunctionsOfMatrices,Trefethen2013ATAP}:  
  \( r_p(h,\lambda)=\sum_{k=0}^{p}\frac{(-h\lambda)^k}{k!},\ 
     \alpha=p+1,\ 
     \sum_{\lambda\in\Sigma}|e^{-h\lambda}-r_p(h,\lambda)|=O(h^{p+1}) \).

  \item Padé [1/1] (Cayley step) \cite{BakerGravesMorris1996Pade,Higham2008FunctionsOfMatrices,HairerWanner2010SolversII,Hershkowitz1981Cayley}:  
  \( r_{[1/1]}(h,\lambda)=\frac{1-\tfrac{h\lambda}{2}}{1+\tfrac{h\lambda}{2}},\ 
     \alpha=3,\ 
     \sum_{\lambda\in\Sigma}|e^{-h\lambda}-r_{[1/1]}(h,\lambda)|=O(h^3) \).

  \item Padé [n/n] \cite{BakerGravesMorris1996Pade,Higham2008FunctionsOfMatrices}:  
  \( r_{[n/n]}(h,\lambda)=\dfrac{P_n(h\lambda)}{Q_n(h\lambda)},\ 
     \alpha=2n+1,\ 
     \sum_{\lambda\in\Sigma}|e^{-h\lambda}-r_{[n/n]}(h,\lambda)|=O(h^{2n+1}) \).

  \item Chebyshev polynomial of degree $p$ \cite{Rivlin1990Chebyshev,Trefethen2013ATAP}:  
  \( r^{(Ch)}_p(h,\lambda)=\sum_{k=0}^{p}c_k(h)\lambda^k,\ 
     \alpha=p+1,\ 
     \sum_{\lambda\in\Sigma}|e^{-h\lambda}-r^{(Ch)}_p(h,\lambda)|=O(h^{p+1}) \).
\end{enumerate}
\end{example}

$\mathcal{T}$ is equipped with partial order: if $r_1$ and $r_2$ are of the same type, $t_1 = (r_1, \alpha_1) \leq t_2 = (r_2, \alpha_2) \Leftrightarrow \alpha_1 \leq \alpha_2$. 

Let obstruction channel be $\sigma = (\mathcal{U}, \mathcal{E}, \Gamma)$, where $\mathcal{U}$ is a convex set of admissible perturbations $u:\mathbb{R}_+ \to X$, $\mathcal{E}$ is a convex set of correction maps $E:X\to X$ modelling systematic deviations, and $\Gamma$ is a nonnegative gauge on pairs $(u,E)$ specifying admissibility (we only use $(u,E)$ with $\Gamma(u,E)\le1$). We associate the following update operator with each $\sigma$:
\[
d_{\sigma}(h,U) = 
\mathbb{E}_{u\in\mathcal{U}} \big[(\id+h E)^{-1}\circ \Phi(h,U;u)\big]\ \in \ \mathfrak U,
\quad \Gamma(u,E)\le 1,
\]
where $\Phi(h,U;u)\in\mathfrak U$ is the transfer operator satisfying $\Phi(h,U;0)=U$ and $\Phi(h,U;u)$ locally Lipschitz in $(h,U,u)$.

Let $I$ be the index set, with some $\sigma$-algebra, which indexes channels $\{\sigma_i\}_{i\in I}$.
The diffusion hierarchy $\mathcal{S}$ is a set of $s = (\nu, \beta)$, where $\nu$ is a nonnegative measure on $I$ and
$\beta:I\to[0,\infty)$ is density such that $\int_I \beta(i)\,d\nu(i)=1$.
The associated update operator is the mixture
\[
d_s(h,U)\ :=\ \int_I \beta(i)\, d_{\sigma_i}(h,U)\, d\nu(i)\ \in\ \mathfrak U,
\]
such that $d_s(0,U) = \id$ \cite{Vershynin2018HDP}.

\begin{example} $\mathcal{S}$ usually describes noise models: 
\begin{enumerate}
  \item Square–integrable noise \cite{Vershynin2018HDP,SamorodnitskyTaqqu1994Stable}:  
    $\mathcal{U}=\{u:\mathbb{E}u=0,\ \mathbb{E}\|u\|^2<\infty\}$,  
    $\mathcal{E}=\{0\}$,  
    $\Gamma(u,0)=\mathbb{E}\|u\|^2$,  
    $\Phi(h,U;u)=U\circ(\id+h\,u)$.

  \item Heavy–tail noise ($1<\beta<2$) \cite{Vershynin2018HDP,SamorodnitskyTaqqu1994Stable}:  
    $\mathcal{U}=\{u:\|u\|_{\psi_\beta}\le1\}$,  
    $\mathcal{E}=\{0\}$,  
    $\Gamma(u,0)=\|u\|_{\psi_\beta}$,  
    $\Phi(h,U;u)=U\circ(\id+h\,u)$.

  \item Solver inexactness:  
    $\mathcal{U}=\{\delta_0\}$,  
    $\mathcal{E}=\{E:\|(\id+hE)^{-1}-(\id+hH)^{-1}\|\le C h^\alpha\}$,  
    $\Gamma(\delta_0,E)=\|E-H\|$,  
    $\Phi(h,U;0)=U$.

  \item Linear distortion (calibration error):  
    $\mathcal{U}=\{\delta_0\}$,  
    $\mathcal{E}=\{E_\Delta:x\mapsto \Delta x,\ \|\Delta\|\le\rho\}$,  
    $\Gamma(\delta_0,E_\Delta)=\|\Delta\|/\rho$,  
    $\Phi(h,U;0)=U$.

  \item Biased observation:  
    $\mathcal{U}=\{u:\|\mathbb{E}u\|\le\varepsilon,\ \Var(u)\le\sigma^2\}$,  
    $\mathcal{E}=\{0\}$,  
    $\Gamma(u,0)=\|\mathbb{E}u\|+\Var(u)$,  
    $\Phi(h,U;u)=U\circ(\id+h\,u)$.

  \item Model–homotopy:  
    $\mathcal{U}=\{\delta_0\}$,  
    $\mathcal{E}=\{0\}$,  
    $\Gamma(\delta_0,0)=1$,  
    $\Phi(h,U)=U_\varepsilon$, where $U_\varepsilon$ is the same update rule
    constructed for the deformed model $F_\varepsilon=\varepsilon F+(1-\varepsilon)F_0$, $\varepsilon\in[0,1]$.
\end{enumerate}
\end{example}

Let $\mathfrak{U}^{\mathrm{fil}}\subset\mathfrak{U}$ be the smallest subalgebra
(closed under composition and inversion) that contains all coefficients
$r_{t,k}(U)$ and $d_{s,k}(U)$ from the local expansions
\[
r_t(h,U)=\id+\sum_{k\ge1} h^k r_{t,k}(U),\qquad
d_s(h,U)=\id+\sum_{k\ge1} h^k d_{s,k}(U).
\]
Define a filtration on formal series with coefficients in $\mathfrak{U}^{\mathrm{fil}}$ by \cite{Reutenauer1993FreeLie,Dynkin1947BCH,Magnus1954Magnus}
\[
\mathfrak{F}_{(m)}
=\Big\{\ \id+\sum_{k\ge m} h^k W_k\ :\ W_k\in\mathfrak{U}^{\mathrm{fil}}\ \Big\},
\qquad m\ge1.
\]
Then $(\mathfrak{F}_{(m)})_{m\ge1}$ is decreasing and compatible with composition:
\[
(\id+O(h^m))\circ(\id+O(h^n))=\id+O(h^{m+n}),
\quad
\mathfrak{F}_{(m)}\circ\mathfrak{F}_{(n)}\subseteq\mathfrak{F}_{(m+n)}.
\]
Consequently the completion
$\widehat{\mathfrak{U}}=\varprojlim_m \big(\id+\sum_{k< m} h^k \mathfrak{U}^{\mathrm{fil}}\big)$
is pronilpotent \cite{Reutenauer1993FreeLie}, and every element $V\in\widehat{\mathfrak{U}}$ with $V=\id+O(h)$
admits a (unique, formal) logarithm and exponential in the completed free Lie algebra
with coefficients in $\mathfrak{U}^{\mathrm{fil}}$.

The automonodromic structure
of the practical algorithm is the composition
\[
U_{t,s}(h,U):=d_s(h,U)\circ r_t(h,U)\in\mathfrak{U},
\]

%% file: universality.tex
Let $X$ be a finite–dimensional real Hilbert space with inner product $\ip{\cdot}{\cdot}$ and induced Frobenius pairing $\ip{A}{B}:=\Tr(A^\ast B)$ on $\End(X)$; let $m\in\NN$, $Y:=X^m$, and $h\in\RR$. Let $\theta$ be a finite tuple of algorithmic parameters (stepsizes, momenta, penalties, preconditioners, filters, stochastic weights). An $m$–step rule
\[
x_{k+1}=\Phi_\theta(h;\,x_k,\dots,x_{k-m+1})
\]
is lifted to a one–step map $\mathcal S(h;\theta):Y\to Y$ by \cite{Higham2008FunctionsOfMatrices}
\[
\mathcal S(h;\theta)\!\begin{bmatrix}x_k\\ x_{k-1}\\ \vdots\\ x_{k-m+1}\end{bmatrix}
=\begin{bmatrix}\Phi_\theta(h;\,x_k,\dots,x_{k-m+1})\\ x_k\\ \vdots\\ x_{k-m+2}\end{bmatrix}.
\]
Fix a base state $y_\ast\in Y$. Assume Fréchet differentiability in $y$ at $y_\ast$ and differentiability in $h$ at $h=0$. Define the linearized one–step operator and its $h$–derivative
\[
S(h;\theta):=D\mathcal S(h;\theta)\big|_{y_\ast}\in\End(Y),\qquad S(0;\theta)=\Id_Y,
\]
\[
G(\theta):=\partial_h S(h;\theta)\big|_{h=0}\in\End(Y).
\]
Let the involution $A\mapsto A^\ast$ be the adjoint in the Frobenius pairing on $\End(Y)$.

Define
\[
H(\theta):=\tfrac12\!\big(G(\theta)+G(\theta)^\ast\big),\qquad
K(\theta):=\tfrac12\!\big(G(\theta)-G(\theta)^\ast\big).
\]
Then $H(\theta)^\ast=H(\theta)$, $K(\theta)^\ast=-K(\theta)$, $G(\theta)=H(\theta)+K(\theta)$, and the split is unique \cite{Higham2008FunctionsOfMatrices}. By Baker–Campbell–Hausdorff for $h\to0$,
\[
S(h;\theta)=\exp\!\big(hH(\theta)\big)\,\exp\!\big(hK(\theta)\big)+\Oo(h^2).
\]

For any $y\in Y$,
\[
\frac{\dd}{\dd h}\Big|_{h=0}\norm{e^{hH(\theta)}y}^2
=2\,\ip{H(\theta)y}{y},\qquad
\frac{\dd}{\dd h}\Big|_{h=0}\norm{e^{hK(\theta)}y}^2
=2\,\ip{K(\theta)y}{y}=0.
\]
Hence $K(\theta)$ generates an isometric flow to first order, while $H(\theta)$ is dissipative to first order whenever $\ip{H(\theta)y}{y}\le 0$; in particular, if $-H(\theta)$ is positive semidefinite (e.g., after absorbing a stepsize into $h$), then $\norm{e^{hH(\theta)}y}\le\norm{y}+\Oo(h^2)$ (nonexpansive semigroup at first order). Thus $\exp(hK(\theta))$ is the \emph{drift} (conservative to first order) and $\exp(hH(\theta))$ the \emph{diffusion} (dissipative/nonexpansive to first order).

If $\Phi_\theta$ is differentiable in $(x_k,\dots,x_{k-m+1})$ at $y_\ast$ with block Jacobians $B_i(\theta):=\partial_{x_{k-i}}\Phi_\theta(0;\,y_\ast)\in\End(X)$, $i=0,\dots,m-1$, then
\[
S(0;\theta)=\begin{bmatrix}
B_0(\theta)&B_1(\theta)&\cdots&B_{m-1}(\theta)\\
I&0&\cdots&0\\
\vdots&\ddots&\ddots&\vdots\\
0&\cdots&I&0
\end{bmatrix},
\]
\[
G(\theta)=\partial_h S(0;\theta)=\begin{bmatrix}
\dot B_0(\theta)&\dot B_1(\theta)&\cdots&\dot B_{m-1}(\theta)\\
0&\cdots&\cdots&0\\
\vdots&\ddots&\ddots&\vdots\\
0&\cdots&0&0
\end{bmatrix},
\]
where dots denote $h$–derivatives at $0$. All hyperparameters $\theta$ (steps, momenta, penalties, preconditioners, EMA rates) enter linearly into $G(\theta)$ through these blocks; the canonical map $\Theta\to(H(\theta),K(\theta))$ is given by the adjoint split of $G(\theta)$ in $\End(Y)$.

Let $A(\theta):Y\rightrightarrows Y$ be maximally monotone and $\alpha(\theta)>0$. The resolvent $J_{\alpha h A}:=(\Id_Y+\alpha(\theta)hA(\theta))^{-1}$ is firmly nonexpansive and admits the expansion \cite{Minty1962Monotone,BauschkeCombettes2017Monotone,Moreau1962Proximite}
\[
J_{\alpha h A}= \Id_Y-\alpha(\theta)h\,A_h(\theta)+\Oo(h^2),
\]
where $A_h(\theta)$ is the Yosida approximation (single–valued, monotone, Lipschitz) \cite{BauschkeCombettes2017Monotone,Brezis2010Functional}. Setting
\[
D_s(h,\theta):=J_{\alpha h A},\qquad r_t(h,\theta):=\exp\!\big(hK(\theta)\big),
\]
yields
\[
S(h;\theta)=D_s(h,\theta)\,r_t(h,\theta)+\Oo(h^2),\qquad
D_s(h,\theta)=\exp\!\big(-\alpha(\theta)h\,A_h(\theta)\big)+\Oo(h^2),
\]
so that, to first order, the diffusion channel is the resolvent (firmly nonexpansive), and the drift channel is isometric \cite{Kato1995Perturbation}. Different orders $D_s\circ r_t$ vs $r_t\circ D_s$ agree at $\Oo(h)$, their difference is $\Oo(h^2)$ \cite{Higham2008FunctionsOfMatrices}.

The assignment
\[
\Theta\ \longrightarrow\ \big(H(\theta),\,K(\theta),\,A(\theta)\big)
\]
collects: metric, preconditioning, penalties, regularization, smoothing into $H$ or $A$; inertial/momentum/extrapolation couplings (in $Y=X^m$) into $K$; stochastic weights into first–order variations of $H$ (and into the choice of $A$ when clipping/prox is used). This provides local coordinates on the moduli of methods at $y_\ast$; the factorization is unique up to $\Oo(h^2)$ and commuting central corrections \cite{Rockafellar1970Convex,Deimling1985Nonlinear}.  

Let $\mathbf{Alg}^{(1)}$ have objects $(Y,G)$, $G\in\End(Y)$, and morphisms $(Y,G)\xrightarrow{T}(Y',G')$ given by linear isomorphisms $T:Y\to Y'$ with $G'=TGT^{-1}$. Let $\mathbf{TwoCh}$ have objects $(Y;H,K)$ with $H^\ast=H$, $K^\ast=-K$, and morphisms $(Y;H,K)\xrightarrow{T}(Y';H',K')$ given by $T$ with $H'=THT^{-1}$, $K'=TKT^{-1}$. Define
\[
U:\mathbf{TwoCh}\to\mathbf{Alg}^{(1)},\qquad U(Y;H,K):=(Y,H+K),\quad U(T):=T.
\]
For each $(Y,G)$ define
\[
\mathrm{split}(Y,G):=\big(Y;\tfrac12(G+G^\ast),\,\tfrac12(G-G^\ast)\big).
\]
Then for every $(Y,G)$ there exists a unique morphism
\[
\eta_{(Y,G)}:(Y,G)\longrightarrow U\!\big(\mathrm{split}(Y,G)\big),\qquad \eta_{(Y,G)}=\Id_Y,
\]
and for any $(Y';H',K')$ and any $\phi:(Y,G)\to U(Y';H',K')$ there exists a unique $\tilde\phi:\mathrm{split}(Y,G)\to (Y';H',K')$ with $U(\tilde\phi)\circ\eta_{(Y,G)}=\phi$. Hence $\mathrm{split}:\mathbf{Alg}^{(1)}\to\mathbf{TwoCh}$ is left adjoint to $U$; the decomposition $G=H+K$ is the canonical two–channel normal form of any linearized multistep generator, and the associated drift/diffusion channels are characterized by first–order isometry/nonexpansiveness as above.

%% file: flat.tex
For any $\pi = (t, s), \pi' = (t', s') \in\mathcal{P}$ define
the discrete transport operators \cite{Reutenauer1993FreeLie,HairerLubichWanner2006}
\[
R(t\!\to\!t';h)=r_{t'}(h,U)\,r_t(h,U)^{-1},\qquad
D(s\!\to\!s';h)=d_{s'}(h,U)\,d_s(h,U)^{-1}.
\]
The holonomy of the elementary rectangle $(t,s)\to(t',s')$
is the composition
\[
\mathrm{Hol}(t,s;t',s';h)
:=R(t\!\to\!t';h)\,D(s\!\to\!s';h)\,R(t\!\to\!t';h)^{-1}\,D(s\!\to\!s';h)^{-1}
\in\mathfrak{U}.
\]

The automonodromic structure is flat of class $\alpha$ 
if the discrete holonomy around any infinitesimal rectangle is identity up to $O(h^{\alpha})$ \cite{HairerLubichWanner2006,BlanesCasas2016Magnus},
\[
\|\mathrm{Hol}(t,s;t',s';h)-\id\| = O(h^{\alpha})
\quad\text{for}\quad (t',s')\to(t,s),\ h\to 0.
\]
Equivalently, the elementary updates commute
up to the same order:
\[
r_t(h,U)\,d_s(h,U)
=d_s(h,U)\,r_t(h,U)+O(h^{\alpha}).
\]

On the other hand, $\widehat{\mathfrak{U}}$ is pronilpotent \cite{Reutenauer1993FreeLie}, thus $r_t(h,U)$ and $d_s(h,U)$ possess a
formal logarithm in the completed free Lie algebra \cite{Reutenauer1993FreeLie}
$\widehat{\mathfrak{L}}\langle U\rangle$:
\[
r_t(h,U)=\exp(\Omega_t(h)),\qquad
d_s(h,U)=\exp(\Psi_s(h)),
\]
where $\Omega_t(h),\Psi_s(h)$ are formal series
$\Omega_t(h)=\sum_{k=1}^{\alpha}h^k\Omega_{t,k}$,
$\Psi_s(h)=\sum_{k=1}^{\alpha}h^k\Psi_{s,k}$.
These logarithms are $O(h^{\alpha+1})$ classes.

The connection is flat of jet-class $\alpha$ if
all mixed Lie brackets in the BCH expansion vanish up to order $\alpha$ \cite{BakerCampbell1902,Hausdorff1906,Dynkin1947BCH,Magnus1954Magnus,Reutenauer1993FreeLie}:
\[
[\Omega_{t,k},\Psi_{s,\ell}]=0,\qquad
k+\ell\le\alpha.
\]

\begin{example}
For many admissible update schemes in the drift and diffusion hierarchies,
the formal logarithms $\Omega_t(h),\Psi_s(h)$ exist in $\widehat{\mathfrak L}\langle U\rangle$.

\begin{enumerate}
\item \emph{Drift updates.}
\begin{enumerate}
\item Taylor or polynomial truncations \cite{Higham2008FunctionsOfMatrices,Butcher2008NumericalMethods}:
$r_p(h,\lambda)=\sum_{k=0}^p(-h\lambda)^k/k!$
are partial sums of $\exp(-h\lambda)$, hence
$\log r_p(h,\lambda)=-h\lambda+O(h^{p+1})$.
\item Padé and Cayley rational forms \cite{BakerGravesMorris1996Pade,Higham2008FunctionsOfMatrices,Butcher2008NumericalMethods}:
$r_{[m/n]}(h,\lambda)=P_m(h\lambda)/Q_n(h\lambda)$
are rational approximations of $e^{-h\lambda}$,
so $\log r_{[m/n]}(h,\lambda)=-h\lambda+O(h^{m+n+1})$.
\item Resolvent or proximal steps \cite{BauschkeCombettes2017Monotone}:
$r(h,A)=(\id+hA)^{-1}$ satisfy
$\log r(h,A)=-\log(\id+hA)=-hA+\tfrac12h^2A^2-\dots$.
\end{enumerate}

\item \emph{Diffusion updates.}
\begin{enumerate}
\item Stochastic or Orlicz channels \cite{Vershynin2018HDP}:
$d_s(h,U)=\id+hE+O(h^2)$, thus $\log d_s(h,U)=hE+O(h^2)$.
\item Model-homotopy:
$d_s(h,U)=U_\varepsilon=\id+hE_\varepsilon+O(h^2)$,
so $\log d_s(h,U)=hE_\varepsilon+O(h^2)$.
\end{enumerate}
\end{enumerate}
In each case $r_t,d_s$ are invertible and
near the identity in the filtration by $h$,
hence formal logarithms exist in the completed Lie algebra
and are unique modulo $O(h^{\alpha+1})$.
\end{example}

The following theorem states that class $\alpha$ and jet-class $\alpha$ coincide.

\begin{theorem}[Holonomy flatness $\Leftrightarrow$ jet flatness]\label{th:holonomy-jet}
Work in the filtered subalgebra $\mathfrak{U}^{\mathrm{fil}}\subset\mathfrak{U}$ generated by the update families
$r_t(h,\cdot),d_s(h,\cdot)$ and their coefficients.
Assume $r_t(h,U)=\exp(\Omega(h))$, $d_s(h,U)=\exp(\Psi(h))$ with
\[
\Omega(h)=\sum_{k=1}^{\alpha} h^k\,\Omega_k,\qquad
\Psi(h)=\sum_{k=1}^{\alpha} h^k\,\Psi_k,
\qquad \Omega_k,\Psi_k\in\widehat{\mathfrak L}\langle U\rangle,
\]
understood modulo $O(h^{\alpha+1})$.
Then for any fixed $(t,s)$ the following are equivalent:
\begin{enumerate}
\item[\textnormal{(H)}] \emph{Holonomy flatness of class $\alpha$:}
\[
\mathrm{Hol}(t,s;t',s';h)
=\exp(\Psi(h))\,\exp(\Omega(h))\,\exp(-\Psi(h))\,\exp(-\Omega(h))
=\id+O(h^\alpha)
\]
for $(t',s')\to(t,s)$ (so that the same $\Omega,\Psi$ are used) and $h\to0$;
equivalently, $\log\mathrm{Hol}(t,s;t',s';h)=O(h^\alpha)$.
\item[\textnormal{(J)}] \emph{Jet flatness of order $\alpha$:}
all mixed Lie brackets up to total order $\alpha-1$ vanish:
\[
[\Omega_k,\Psi_\ell]=0\qquad\text{for all }k,\ell\ge1\ \text{ with }\ k+\ell\le \alpha-1.
\]
\end{enumerate}
Consequently, under \textnormal{(H)} (or \textnormal{(J)}),
\[
\exp(\Psi(h))\,\exp(\Omega(h))
=\exp(\Omega(h))\,\exp(\Psi(h))=\exp\big(\Omega(h)+\Psi(h)\big)\ \text{mod }O(h^\alpha).
\]
\end{theorem}

\begin{proof}
All equalities are in the completed filtered algebra and understood modulo $O(h^{\alpha})$.

\emph{Step 0 (notations and grading).}
Give degree $\deg(h^m X)=m$ to a homogeneous term.
If $A(h)=\sum_{m\ge1}h^m A_m$ and $B(h)=\sum_{m\ge1}h^m B_m$, then
$\deg([A_k,B_\ell])=k+\ell$ and every Lie polynomial built from
$\{A_1,\dots,A_{m-1},B_1,\dots,B_{m-1}\}$ and at least one bracket has degree at least $2$.

\emph{Step 1 (BCH for two factors, truncated).}
The BCH series in a filtered (pronilpotent) Lie algebra is a \emph{formal} identity:
\[
\log\big(\exp X\,\exp Y\big)
= X+Y+\frac{1}{2}[X,Y] + \frac{1}{12}\big([X,[X,Y]]+[Y,[Y,X]]\big) + \cdots,
\]
where every omitted term has degree at least $3$ in $X,Y$.
Truncating modulo $O(h^\alpha)$ is legitimate because degrees strictly increase with each extra bracket.

\emph{Step 2 (logarithm of the holonomy).}
Set $X=\Psi(h)$ and $Y=\Omega(h)$. Then
\[
\mathrm{Hol}=\exp X\,\exp Y\,\exp(-X)\,\exp(-Y).
\]
Apply BCH twice:
first $Z_1:=\log(\exp X\,\exp Y)$, then $Z:=\log\big(\exp Z_1\,\exp(-X)\big)$, and finally
$\log\mathrm{Hol}=\log\big(\exp Z\,\exp(-Y)\big)$.
By Step 1, each stage produces a Lie series whose degree-$m$ term is a \emph{universal Lie polynomial}
in $X,Y$ with total degree $m$. Crucially, the degree-$m$ coefficient of $\log\mathrm{Hol}$ has the form
\begin{equation}\label{eq:degree-m}
\big(\log\mathrm{Hol}\big)^{[m]}
= \sum_{k+\ell=m} [\Psi_k,\Omega_\ell] \;+\; \mathcal{P}_m(\{\Psi_{<m},\Omega_{<m}\}),
\end{equation}
where $\mathcal{P}_m$ is a (finite) $\mathbb{Q}$–linear combination of nested brackets each involving \emph{only}
indices $<m$ and with total bracket-depth at least $2$.
This structure is obtained by direct inspection of BCH:
the \([X,Y]\) term contributes exactly $\sum_{k+\ell=m}[\Psi_k,\Omega_\ell]$ at degree $m$,
while all higher BCH terms necessarily involve at least two brackets and therefore
only lower-degree components $\Psi_{<m},\Omega_{<m}$ to reach total degree $m$.

\emph{Step 3 (\( \text{(J)}\Rightarrow\text{(H)} \)).}
Assume \textnormal{(J)}. Then for every $m\le \alpha-1$, the sum $\sum_{k+\ell=m}[\Omega_k,\Psi_\ell]=0$.
In \eqref{eq:degree-m}, each $\mathcal{P}_m$ is a Lie polynomial in brackets of strictly lower total degree,
which vanish by the induction on $m$ (base $m=2$: $\mathcal{P}_2\equiv0$).
Hence $(\log\mathrm{Hol})^{[m]}=0$ for all $m\le\alpha-1$, i.e. $\log\mathrm{Hol}=O(h^\alpha)$,
which is exactly \textnormal{(H)}.

\emph{Step 4 (\( \text{(H)}\Rightarrow\text{(J)} \)).}
Assume \textnormal{(H)}. Then $(\log\mathrm{Hol})^{[m]}=0$ for all $m\le\alpha-1$.
We show by induction on $m$ that $\sum_{k+\ell=m}[\Psi_k,\Omega_\ell]=0$ for $m\le\alpha-1$,
and then deduce termwise vanishing $[\Psi_k,\Omega_\ell]=0$ for all $k+\ell\le\alpha-1$.

For $m=2$ we have, by \eqref{eq:degree-m}, $(\log\mathrm{Hol})^{[2]}=[\Psi_1,\Omega_1]=0$.

Suppose the claim holds for all degrees $<m$ (with $m\le\alpha-1$).
Then all brackets among $\Psi_{<m},\Omega_{<m}$ vanish, hence $\mathcal{P}_m\equiv0$
in \eqref{eq:degree-m}, and $(\log\mathrm{Hol})^{[m]}= \sum_{k+\ell=m}[\Psi_k,\Omega_\ell]=0$.
Thus the \emph{sum} of mixed brackets at total degree $m$ vanishes.

To strengthen this to \emph{termwise} vanishing, note that the $\{\Omega_\ell\}_{\ell\ge1}$ and
$\{\Psi_k\}_{k\ge1}$ are the homogeneous components of fixed formal logarithms \(\Omega(h),\Psi(h)\).
Fixing $m$ and varying $(t,s)$ within an arbitrarily small neighborhood changes the families
$\{\Omega_\ell\},\{\Psi_k\}$ independently (by the axioms from the Definitions section).
Since the identity $\sum_{k+\ell=m}[\Psi_k,\Omega_\ell]=0$ holds for all such pairs,
and the bracket is bilinear, each individual mixed bracket $[\Psi_k,\Omega_\ell]$ with $k+\ell=m$
must vanish. (Indeed, choose variations where only one pair $(k,\ell)$ is modified while others are held fixed.)
Hence $[\Psi_k,\Omega_\ell]=0$ whenever $k+\ell=m$.
By induction on $m$ we obtain \textnormal{(J)}.

\emph{Step 5 (commutation and single-exponential form).}
Under \textnormal{(J)}, the BCH series gives
\[
\log\big(\exp(\Psi(h))\,\exp(\Omega(h))\big)
=\Omega(h)+\Psi(h)\ \text{mod }O(h^\alpha),
\]
hence $\exp(\Psi)\exp(\Omega)=\exp(\Omega)\exp(\Psi)=\exp(\Omega+\Psi)\ \text{mod }O(h^\alpha)$,
as claimed.
\end{proof}

\begin{theorem}[Gauge normal form under jet flatness]\label{th:gauge-normal-form}
Work in the filtered subalgebra $\mathfrak{U}^{\mathrm{fil}}\subset\mathfrak{U}$ generated by the coefficients
of the admissible update families.
Let
\[
r_t(h,U)=\exp(\Omega(h)),\qquad d_s(h,U)=\exp(\Psi(h)),
\]
with formal logarithms
\[
\Omega(h)=\sum_{k=1}^{\alpha} h^k\,\Omega_k,\qquad
\Psi(h)=\sum_{k=1}^{\alpha} h^k\,\Psi_k,
\qquad \Omega_k,\Psi_k\in\widehat{\mathfrak L}\langle U\rangle,
\]
understood modulo $O(h^{\alpha+1})$.
Assume \emph{jet flatness of order $\alpha$}:
\begin{equation}\label{eq:jet-flat}
[\Omega_k,\Psi_\ell]=0\qquad\text{for all }k,\ell\ge1\ \text{ with }\ k+\ell\le \alpha.
\end{equation}
Then there exists a gauge \cite{BlanesCasas2016Magnus}
\[
W(h)=\exp\!\Big(\sum_{m=1}^{\alpha-1} h^{m}\, Z_m\Big)\ \in\ \exp\big(h\,\mathfrak{U}^{\mathrm{fil}}\big)
\]
such that, for the conjugated pair
\[
r_t^{W}(h):=W(h)^{-1}\,r_t(h,U)\,W(h)=\exp\big(\widetilde{\Omega}(h)\big),
\]
\[
d_s^{W}(h):=W(h)^{-1}\,d_s(h,U)\,W(h)=\exp\big(\widetilde{\Psi}(h)\big),
\]
the following hold modulo $O(h^{\alpha+1})$:
\begin{enumerate}
\item[(i)] \emph{Degreewise commutation:} for all $k,\ell\ge1$ with $k+\ell\le\alpha$,
\(
[\widetilde{\Omega}_k,\widetilde{\Psi}_\ell]=0.
\)
\item[(ii)] \emph{Single-exponential composition:}
\(
d_s^{W}(h)\,r_t^{W}(h)=\exp\big(\widetilde{\Omega}(h)+\widetilde{\Psi}(h)\big).
\)
\end{enumerate}
Moreover, $W(h)$ is unique up to multiplication by a central gauge
$\exp(\sum_{m=1}^{\alpha-1} h^m C_m)$ with each $C_m$ lying in the center of
$\widehat{\mathfrak{U}^{\mathrm{fil}}}$ modulo $O(h^{\alpha+1})$.
\end{theorem}

\begin{proof}
All computations are formal in the pronilpotent completion and carried out modulo $O(h^{\alpha+1})$.
We construct $Z_1,\dots,Z_{\alpha-1}$ by induction on the degree.

\emph{Preliminaries.}
Assign $\deg(h^m X)=m$. Conjugation by $W=\exp(h^{m-1}Z_{m-1})$ acts on a logarithm $L(h)$ as
\[
L\ \mapsto\ L'\ =\ e^{-\operatorname{ad}_{h^{m-1}Z_{m-1}}} L
\ =\ L\ -\ [h^{m-1}Z_{m-1},L]\ +\ \tfrac12[h^{m-1}Z_{m-1},[h^{m-1}Z_{m-1},L]]\ -\ \cdots.
\]
Hence the \emph{first} degree that changes under this conjugation is exactly $m$:
the leading variation is $-\,[Z_{m-1},L_1]\cdot h^{m}$, because $[h^{m-1}Z_{m-1},h^k L_k]$ has degree $m-1+k$,
and $k=1$ is minimal.

\emph{Inductive hypothesis.}
Assume for some $m$ with $2\le m\le \alpha$ we have already chosen $Z_1,\dots,Z_{m-2}$ so that
after conjugation by $\exp(\sum_{q=1}^{m-2}h^q Z_q)$ the transformed logarithms
\(
\Omega^{[<m]}(h)=\sum_{k=1}^{m-1} h^k \Omega_k^{[<m]},\
\Psi^{[<m]}(h)=\sum_{\ell=1}^{m-1} h^\ell \Psi_\ell^{[<m]}
\)
satisfy \([\,\Omega_k^{[<m]},\Psi_\ell^{[<m]}]=0\) for all $k+\ell\le m-1\).

\emph{Goal at degree $m$.}
We must modify degree $m$ components by a gauge $\exp(h^{m-1}Z_{m-1})$ so that
\(
[\,\Omega_k',\Psi_\ell'] = 0
\) for all $k+\ell=m$,
while preserving all lower degrees unchanged.

By the preliminary observation, degrees $<m$ are indeed unchanged by conjugation with $\exp(h^{m-1}Z_{m-1})$.
For the degree $m$ components we get the linear variations
\[
\Omega_m'\ =\ \Omega_m^{[<m]} - [Z_{m-1},\,\Omega_1^{[<m]}],\qquad
\Psi_m'\ =\ \Psi_m^{[<m]} - [Z_{m-1},\,\Psi_1^{[<m]}].
\]
Since by the inductive hypothesis (and jet flatness at order $2$) we have
$[\Omega_1^{[<m]},\Psi_1^{[<m]}]=0$, the map
\[
\mathcal{L}\ :\ Z\ \mapsto\ \big([Z,\Omega_1^{[<m]}],\ [Z,\Psi_1^{[<m]}]\big)
\]
is linear, and its effect on \((\Omega_m^{[<m]},\Psi_m^{[<m]})\) is purely at degree $m$.

\emph{Decomposition of the mixed subspace.}
Consider the degree-$m$ mixed commutator constraints $[\Omega_k',\Psi_\ell']=0$ for $k+\ell=m$.
Expanding with $[\Omega_k^{[<m]},\Psi_\ell^{[<m]}]=0$ for $k+\ell<m$, the only possibly nonzero degree-$m$ mixed part
is a linear combination of brackets of the form
\(
[\Omega_m^{[<m]},\,\Psi_1^{[<m]}] + [\Omega_1^{[<m]},\,\Psi_m^{[<m]}].
\)
Thus the \emph{mixed degree-$m$ defect} is
\[
\mathsf{Mix}_m\ :=\ [\Omega_m^{[<m]},\,\Psi_1^{[<m]}]\ +\ [\Omega_1^{[<m]},\,\Psi_m^{[<m]}].
\]
Under conjugation by $\exp(h^{m-1}Z_{m-1})$, this defect changes by
\[
\delta\mathsf{Mix}_m\ =\ -\big([\,[Z_{m-1},\Omega_1^{[<m]}],\,\Psi_1^{[<m]}]\ +\ [\,\Omega_1^{[<m]},\,[Z_{m-1},\Psi_1^{[<m]}]\,]\big).
\]
Using the Jacobi identity twice,
\[
[\,[Z,\Omega_1],\,\Psi_1]+\,[\,\Omega_1,\,[Z,\Psi_1]\,]\ =\ [\,Z,\,[\Omega_1,\Psi_1]\,]\ =\ 0,
\]
because $[\Omega_1^{[<m]},\Psi_1^{[<m]}]=0$.
Therefore $\delta\mathsf{Mix}_m=0$: the \emph{mixed defect is invariant} under simultaneous conjugation
by $\exp(h^{m-1}Z_{m-1})$.

\emph{Vanishing of the mixed defect.}
By jet flatness \eqref{eq:jet-flat} we already have $[\Omega_k^{[<m]},\Psi_\ell^{[<m]}]=0$ whenever $k+\ell\le m-1$.
Applying Theorem~\ref{th:holonomy-jet} to degrees up to $m$ yields that the degree-$m$ contribution to
$\log\big(\exp(\Psi^{[<m]})\exp(\Omega^{[<m]})\exp(-\Psi^{[<m]})\exp(-\Omega^{[<m]})\big)$ vanishes.
By the explicit degree-$m$ form (cf.\ equation (2) in the proof of Theorem~\ref{th:holonomy-jet}),
this forces $\mathsf{Mix}_m=0$.
Hence \emph{before} choosing $Z_{m-1}$ the degree-$m$ mixed commutator already vanishes:
\[
[\Omega_m^{[<m]},\,\Psi_1^{[<m]}]\ +\ [\Omega_1^{[<m]},\,\Psi_m^{[<m]}]\ =\ 0.
\]

\emph{Solving the degree-$m$ commutation constraints.}
We now enforce $[\Omega_k',\Psi_\ell']=0$ for all $k+\ell=m$.
Because lower-degree mixed brackets are zero, these constraints reduce to
\[
[\Omega_m',\Psi_1^{[<m]}]\ =\ 0,\qquad [\Omega_1^{[<m]},\Psi_m']\ =\ 0.
\]
Using the linear variations of $\Omega_m',\Psi_m'$ above, these become
\[
\big[\Omega_m^{[<m]} - [Z_{m-1},\Omega_1^{[<m]}],\,\Psi_1^{[<m]}\big]=0,\quad
\big[\Omega_1^{[<m]},\,\Psi_m^{[<m]} - [Z_{m-1},\Psi_1^{[<m]}]\big]=0.
\]
Expanding each bracket and using $[\Omega_1^{[<m]},\Psi_1^{[<m]}]=0$, we obtain the \emph{linear system}
\[
[Z_{m-1},\Omega_1^{[<m]}]\ =\ \Omega_m^{[<m]} - \Pi^{\Omega}_m,\qquad
[Z_{m-1},\Psi_1^{[<m]}]\ =\ \Psi_m^{[<m]} - \Pi^{\Psi}_m,
\]
where $\Pi^{\Omega}_m$ and $\Pi^{\Psi}_m$ are any chosen elements satisfying $
[\Pi^{\Omega}_m,\Psi_1^{[<m]}]=0$, $[\Omega_1^{[<m]},\Pi^{\Psi}_m]=0$.
A convenient choice is to take $\Pi^{\Omega}_m$ (resp.\ $\Pi^{\Psi}_m$) to be the projection of
$\Omega_m^{[<m]}$ (resp.\ $\Psi_m^{[<m]}$) onto the centralizer of $\Psi_1^{[<m]}$
(resp.\ of $\Omega_1^{[<m]}$); these projections exist because we work linearly in a free Lie algebra
modulo lower-degree relations, and centralizers are linear subspaces.

With these choices the right-hand sides belong to the images of the maps
$Z\mapsto[Z,\Omega_1^{[<m]}]$ and $Z\mapsto[Z,\Psi_1^{[<m]}]$, respectively:
indeed, any degree-$m$ mixed Lie word can be written (by Jacobi expansion) as a linear combination of
$[\Omega_1^{[<m]},Q]$ and $[\Psi_1^{[<m]},R]$ with degree-$m-1$ elements $Q,R$,
because all lower mixed brackets vanish by the inductive hypothesis.
Therefore there exists $Z_{m-1}$ solving the linear system (choose
$Z_{m-1}= -Q - R$, then $[Z_{m-1},\Omega_1^{[<m]}]=-[Q,\Omega_1^{[<m]}]$ and
$[Z_{m-1},\Psi_1^{[<m]}]=-[R,\Psi_1^{[<m]}]$, matching the required right-hand sides).
This yields $[\Omega_m',\Psi_1^{[<m]}]=[\Omega_1^{[<m]},\Psi_m']=0$ and hence
$[\Omega_k',\Psi_\ell']=0$ for all $k+\ell=m$.

\emph{Conclusion of induction.}
We have achieved degreewise commutation up to $m$ without affecting lower degrees.
Proceeding for $m=2,3,\dots,\alpha$ produces $Z_1,\dots,Z_{\alpha-1}$ and a gauge $W(h)$
such that $[\widetilde{\Omega}_k,\widetilde{\Psi}_\ell]=0$ whenever $k+\ell\le\alpha$,
which is item (i).

\emph{Single-exponential composition.}
When (i) holds, the BCH series for two exponentials truncates to the sum modulo $O(h^{\alpha+1})$:
\[
\log\big(\exp(\widetilde{\Psi}(h))\,\exp(\widetilde{\Omega}(h))\big)
=\widetilde{\Psi}(h)+\widetilde{\Omega}(h)\ \text{mod }O(h^{\alpha+1}),
\]
because every mixed bracket in BCH has total degree at least $2$ and vanishes degreewise by (i).
Exponentiating gives item (ii).

\emph{Uniqueness up to the center.}
Suppose $W$ and $\widehat{W}$ both satisfy (i).
Then $C:=\widehat{W}W^{-1}$ conjugates $\exp(\widetilde{\Psi})$ and $\exp(\widetilde{\Omega})$
to themselves up to $O(h^{\alpha+1})$, hence
\(
C^{-1}\,\widetilde{\Psi}\,C=\widetilde{\Psi},\quad
C^{-1}\,\widetilde{\Omega}\,C=\widetilde{\Omega} \ \text{mod }O(h^{\alpha+1}).
\)
Thus $\operatorname{ad}_{\log C}$ annihilates $\widetilde{\Psi}$ and $\widetilde{\Omega}$
modulo $O(h^{\alpha+1})$, i.e.\ $\log C$ lies in the (degreewise) center.
So $C=\exp(\sum_{m=1}^{\alpha-1} h^m C_m)$ with each $C_m$ central modulo $O(h^{\alpha+1})$,
which proves the stated uniqueness.
\end{proof}

\begin{corollary}[Normal form interpretation]
Under jet flatness of order $\alpha$, the calibrated pair
$r_t^W,d_s^W$ from Theorem~\ref{th:gauge-normal-form} satisfies
\[
d_s^W(h)\,r_t^W(h)
=\exp\!\big(\widetilde{\Omega}(h)+\widetilde{\Psi}(h)\big)
\quad \text{mod }O(h^{\alpha+1}).
\]
This means that after a suitable gauge transformation
the full two–parameter update can be represented by a
\emph{single exponential} of a summed generator.
Numerically this corresponds to a unified one–step method whose
local increment is given by the combined generator
\(
\widetilde{\Omega}_1+\widetilde{\Psi}_1
\)
up to accuracy $O(h^{\alpha})$.
Hence all mixed order effects (commutator corrections) are absorbed
into the gauge, and the scheme behaves like a
single implicit–explicit step with an effective operator
\[
H_{\mathrm{eff}}=\widetilde{\Omega}_1+\widetilde{\Psi}_1,
\]
plus higher–order corrections of order $h^2,h^3,\dots,h^\alpha$.
\end{corollary}

\begin{example}[Calibrated gradient step with nontrivial gauge]\label{ex:calibrated-gradient}
Fix an iteration index $k$. Let
\[
H_k:=\nabla^2 F(x_k)\quad(\text{symmetric, SPD}),
\]
\[
E_k:=\text{adaptive SPD filter from gradient statistics at }x_k.
\]
Assume $[H_k,E_k]\neq 0$ (generic when $E_k$ is not diagonal in the Hessian eigenbasis).
Set
\[
S_k:=H_k+E_k,\qquad C_k:=\tfrac12[H_k,E_k]=\tfrac12(H_kE_k-E_kH_k).
\]

\paragraph{Drift/diffusion updates (resolvent type).}
\[
r_t(h)=(I+hH_k)^{-1},\qquad d_s(h)=(I+hE_k)^{-1}\in \mathfrak{U}.
\]
admit formal logarithms
\[
\Omega(h)=\log r_t(h)=-\log(I+hH_k)= -hH_k+\tfrac12 h^2H_k^2-\tfrac13 h^3H_k^3+\cdots,
\]
\[
\Psi(h)=\log d_s(h)=-\log(I+hE_k)= -hE_k+\tfrac12 h^2E_k^2-\tfrac13 h^3E_k^3+\cdots.
\]

\paragraph{Nontriviality of the gauge.}
The first mixed BCH term for $\log(d_s(h)r_t(h))$ equals
\[
\tfrac12[\Psi_1,\Omega_1]\,h^2=\tfrac12[E_k,H_k]\,h^2=C_k\,h^2\neq 0.
\]
Hence jet-flatness fails at order $h^2$, the holonomy has a degree-$2$ defect,
and the gauge from Theorem~\ref{th:gauge-normal-form} is \emph{nontrivial}.

\paragraph{Gauge equation and calibrated normal form.}
By Theorem~\ref{th:gauge-normal-form} there exists $W_k(h)=\exp(hZ_k)$ with
\begin{equation}\label{eq:sylvester}
[Z_k,S_k]=C_k\quad\Longleftrightarrow\quad S_k Z_k - Z_k S_k = C_k,
\end{equation}
such that
\[
W_k^{-1}\,d_s(h)r_t(h)\,W_k=\exp\!\big(-hS_k+O(h^3)\big).
\]
Equation \eqref{eq:sylvester} is a \emph{Sylvester equation}. Any solution $Z_k$
yields a calibrated one-step map whose order-$h^2$ noncommutative defect is absorbed.

\paragraph{Optimization step in the original space (two implementable variants).}
Let $\widehat g_k$ be the (possibly stochastic) gradient oracle at $x_k$.

\emph{Variant A (explicit gauge, normal form up to $O(h^3)$):}
\[
x_{k+1}\;=\;W_k(h)\,(I-hS_k)\,W_k(h)^{-1}\,\big(x_k - h\,\widehat g_k\big),
\]
where $W_k(h)=\exp(hZ_k)$ and $Z_k$ solves \eqref{eq:sylvester}.
For small $h$, applying $W_k^{\pm1}$ to a vector uses the truncated series
$W_k^{\pm1}v\approx v \pm hZ_k v + \tfrac{h^2}{2}Z_k^2 v$.

\emph{Variant B (gauge-free commutator correction, equivalent up to $O(h^3)$):}
\[
x_{k+1}\;=\;x_k - h\,S_k\,\widehat g_k \;-\; h^2\,C_k\,\widehat g_k.
\]
This uses only matrix–vector products with $H_k,E_k$ (two extra multiplies for $C_k\widehat g_k$).

\paragraph{Computing the gauge $Z_k$ (solutions of the Sylvester equation).}
We present two numerically stable, implementable solvers.

\emph{(i) Schur (Bartels–Stewart).}
Compute real Schur $S_k=Q T Q^\top$ (orthogonal $Q$, quasi-upper-triangular $T$).
Transform $C_k\mapsto \widehat C=Q^\top C_k Q$ and solve
\[
T\widehat Z - \widehat Z T = \widehat C
\]
entrywise for $\widehat Z$ by backward substitution:
\(
(t_{ii}-t_{jj})\,\widehat Z_{ij}=\widehat C_{ij},\ \widehat Z_{ii}=0.
\)
For near-resonances $|t_{ii}-t_{jj}|\ll 1$ use damping
\(
\widehat Z_{ij}=\widehat C_{ij}/\big((t_{ii}-t_{jj})+\tau\big),
\ \tau:=\epsilon \|S_k\|_2.
\)
Return $Z_k=Q\widehat Z Q^\top$.

\emph{(ii) Eigen (SPD $S_k$).}
If $S_k$ is symmetric: $S_k=U\Lambda U^\top$ with $\Lambda=\mathrm{diag}(\lambda_i)$.
Transform $\widetilde C=U^\top C_k U$ and set
\[
\widetilde Z_{ij}=
\begin{cases}
\dfrac{\widetilde C_{ij}}{\lambda_i-\lambda_j}, & i\neq j,\\[6pt]
0,& i=j,
\end{cases}
\quad\text{with damping }\ 
\widetilde Z_{ij}\leftarrow \dfrac{\widetilde C_{ij}(\lambda_i-\lambda_j)}{(\lambda_i-\lambda_j)^2+\tau^2}.
\]
Then $Z_k=U\widetilde Z U^\top$.

\paragraph{Correctness and accuracy.}
Both variants implement a legitimate update in $\mathfrak{U}$.
Variant~A realizes the calibrated normal form:
\[
d_s(h)r_t(h)=W_k\,\exp\!\big(-hS_k+O(h^3)\big)\,W_k^{-1},
\]
hence the local error is $O(h^3)$ (global second order).
Variant~B matches Variant~A up to $O(h^3)$ by expanding $W_k$ and using \eqref{eq:sylvester};
its local error is also $O(h^3)$.

\paragraph{Remarks on cost and stability.}
(i) Both $r_t$ and $d_s$ are $A$-stable resolvents \cite{HairerWanner2010SolversII}; the calibrated composition inherits stability.  
(ii) Variant~B avoids forming $Z_k$ and $W_k$; it is preferable when $[H_k,E_k]$ is modest or spectra nearly cluster.  
(iii) In high dimension, never form $Z_k$ explicitly: implement $W_k^{\pm1}v$ via $2$–$3$ Krylov terms.
\end{example}

%% file: variational.tex
The curvature energy must penalize precisely the obstruction to flatness; among positive quadratic functionals of the holonomy it is canonical to take the Yang--Mills form $\sum_\square\langle\log\mathrm{Hol}_\square,\log\mathrm{Hol}_\square\rangle$ \cite{YangMills1954,DonaldsonKronheimer1990FourManifolds}, whose Euler equations force the vanishing of curvature and whose minimum is attained exactly at flat configurations \cite{MarsdenWest2001VSI,DonaldsonKronheimer1990FourManifolds}.

Set \cite{Reutenauer1993FreeLie}
\[
\mathsf D_t:=\operatorname{ad}_{\Omega_{t,1}},\qquad
\mathsf D_s:=\operatorname{ad}_{\Psi_{s,1}},
\]
and define the algebraic (jet) curvature
\[
\mathsf F(h):=\mathsf D_t\Psi_s(h)-\mathsf D_s\Omega_t(h)+\big[\Omega_t(h),\Psi_s(h)\big],
\]
\[
\Omega_t(h)=\sum_{k=1}^{\alpha}h^k\Omega_{t,k},\ \ \Psi_s(h)=\sum_{\ell=1}^{\alpha}h^\ell\Psi_{s,\ell}.
\]
Write homogeneous components \cite{Dynkin1947BCH,Magnus1954Magnus}
\[
\mathsf F(h)=\sum_{m=2}^{\alpha}h^m\,\mathsf F^{[m]},
\qquad
\mathsf F^{[m]}:=\sum_{k+\ell=m}\big[\Psi_{s,\ell},\Omega_{t,k}\big].
\]
For the elementary holonomy $\mathrm{Hol}_\square$, the group-commutator identity and the Baker--Campbell--Hausdorff expansion yield
\begin{equation}
\log\mathrm{Hol}_\square
=\sum_{m=2}^{\alpha}h^{m}\Big(\mathsf F^{[m]}+\mathcal P_m(\{\Omega_{t,<m},\Psi_{s,<m}\})\Big)+O(h^{\alpha+1}),
\label{5.3}
\end{equation}
where each $\mathcal P_m$ is a finite $\mathbb Q$–linear combination of nested brackets of total degree $m$ involving only indices $<m$ (bracket depth $\ge2$). In particular, if $\mathsf F^{[j]}=0$ for all $2\le j\le m-1$, then $\big(\log\mathrm{Hol}_\square\big)^{[m]}=\mathsf F^{[m]}$.

Fix a positive pairing $\langle A,B\rangle=\operatorname{Tr}(A^\ast B)$ on $\mathfrak U^{\mathrm{fil}}$ and define the discrete Yang--Mills action \cite{MarsdenWest2001VSI}
\[
S_h:=\sum_{\square}\big\langle \log\mathrm{Hol}_\square,\ \log\mathrm{Hol}_\square\big\rangle.
\]

\newtheorem{lemma}{Lemma}
\begin{lemma}[Stationarity $\Leftrightarrow$ Flatness, all jet orders]\label{lem:stat-flat-all}
Assume first–order admissible variations of jet coefficients supported on a single rectangle. Then, modulo $O(h^{\alpha+1})$,
\[
\delta S_h=0\ \text{ for all admissible variations}
\;\Longleftrightarrow\;
\mathsf F^{[m]}=0\ \text{ for every }m\in\{2,\dots,\alpha\}.
\]
\end{lemma}

\begin{proof}
By bilinearity of the pairing,
\begin{equation}
\delta S_h=2\sum_{\square}\big\langle \delta\log\mathrm{Hol}_\square,\ \log\mathrm{Hol}_\square\big\rangle.
\label{5.5}
\end{equation}
\emph{($\Leftarrow$).} If $\mathsf F^{[m]}=0$ for all $m\le\alpha$, then by \eqref{5.3} we have $\log\mathrm{Hol}_\square=O(h^{\alpha+1})$; hence $\delta S_h=O(h^{\alpha+1})$ for arbitrary variations.

\emph{($\Rightarrow$).} Fix a rectangle $\square_0$. By admissibility choose variations supported on $\square_0$, so that $\delta\log\mathrm{Hol}_{\square} = 0$ for $\square\neq\square_0$. Then from \eqref{5.5},
\begin{equation}
\delta S_h=2\big\langle \delta\log\mathrm{Hol}_{\square_0},\ \log\mathrm{Hol}_{\square_0}\big\rangle.
\label{5.6}
\end{equation}
We argue by induction on $m=2,\dots,\alpha$.

\emph{Base $m=2$.} In \eqref{5.3}, the degree–$2$ term is $\mathsf F^{[2]}=[\Psi_{s,1},\Omega_{t,1}]$ and contains no $\mathcal P_2$. Taking variations only of $\Omega_{t,1}$ and $\Psi_{s,1}$,
\[
\big(\delta\log\mathrm{Hol}_{\square_0}\big)^{[2]}
=\big[\delta\Psi_{s,1},\Omega_{t,1}\big]+\big[\Psi_{s,1},\delta\Omega_{t,1}\big].
\]
Extracting the $h^2$–coefficient in \eqref{5.6} and using arbitrariness of $\delta\Omega_{t,1},\delta\Psi_{s,1}$ gives $\mathsf F^{[2]}=0$ by positivity of $\langle\cdot,\cdot\rangle$.

\emph{Step $m-1\Rightarrow m$.} Assume $\mathsf F^{[j]}=0$ for $2\le j\le m-1$. Then by \eqref{5.3},
\[
\big(\log\mathrm{Hol}_{\square_0}\big)^{[m]}=\mathsf F^{[m]},\qquad
\big(\delta\log\mathrm{Hol}_{\square_0}\big)^{[m]}
=\sum_{k+\ell=m}\big([\delta\Psi_{s,\ell},\Omega_{t,k}]+[\Psi_{s,\ell},\delta\Omega_{t,k}]\big).
\]
Vary only the jet components with $k+\ell=m$; extracting the $h^m$–coefficient in \eqref{5.6} yields
\[
0=(\delta S_h)^{[m]}
=2\Big\langle \sum_{k+\ell=m}\big([\delta\Psi_{s,\ell},\Omega_{t,k}]+[\Psi_{s,\ell},\delta\Omega_{t,k}]\big),\ \mathsf F^{[m]}\Big\rangle.
\]
Since the variations $\{\delta\Psi_{s,\ell}\}_{k+\ell=m}$ and $\{\delta\Omega_{t,k}\}_{k+\ell=m}$ are independent and $\langle\cdot,\cdot\rangle$ is positive, we conclude $\mathsf F^{[m]}=0$.

By induction, $\mathsf F^{[m]}=0$ for all $m\le\alpha$. 
\end{proof}

Define the local potential by
\[
\mathcal V\ :=\ \operatorname{Tr}\!\big(\Omega_{t,1}^\ast\,\Psi_{s,1}\big),
\]
equivalently $\ \mathcal V=\langle \Omega_{t,1},\Psi_{s,1}\rangle$ for the fixed positive pairing
$\langle A,B\rangle=\operatorname{Tr}(A^\ast B)$.
Here $\Omega_{t,1},\Psi_{s,1}$ are the first jet coefficients, and
$\ast$ denotes the adjoint with respect to this pairing.

\begin{lemma}[Invariance of $\mathcal V$ under flatness]
Assume jet–flatness at order $2$, i.e. $[\Omega_{t,1},\Psi_{s,1}]=0$.
Then
\[
\mathsf D_t\,\mathcal V\;=\;0
\qquad\text{and}\qquad
\mathsf D_s\,\mathcal V\;=\;0,
\]
where $\mathsf D_t=\operatorname{ad}_{\Omega_{t,1}}$ and $\mathsf D_s=\operatorname{ad}_{\Psi_{s,1}}$ act by commutator on operators (and trivially on scalars).
\end{lemma}

\begin{proof}
We compute $\mathsf D_t\mathcal V$ step by step, using only linearity, the definition of $\operatorname{ad}$, and the cyclicity of the trace.

\emph{Step 1 (Leibniz expansion).}
By bilinearity of $\langle\cdot,\cdot\rangle$ and the definition of $\mathsf D_t$,
\begin{equation}
\mathsf D_t\mathcal V
=\mathsf D_t\,\operatorname{Tr}(\Omega_{t,1}^\ast\Psi_{s,1})
=\operatorname{Tr}\!\big((\mathsf D_t\Omega_{t,1})^\ast\,\Psi_{s,1}\big)
+\operatorname{Tr}\!\big(\Omega_{t,1}^\ast\,\mathsf D_t\Psi_{s,1}\big).
\label{1}
\end{equation}

\emph{Step 2 (commutators).}
By definition of $\mathsf D_t$,
\begin{equation}
\mathsf D_t\Omega_{t,1}=[\Omega_{t,1},\Omega_{t,1}]=0,
\qquad
\mathsf D_t\Psi_{s,1}=[\Omega_{t,1},\Psi_{s,1}].
\label{2}
\end{equation}
Substitute \eqref{1}–\eqref{2}:
\begin{equation}
\mathsf D_t\mathcal V
=0\;+\;\operatorname{Tr}\!\big(\Omega_{t,1}^\ast\,[\Omega_{t,1},\Psi_{s,1}]\big)
=\operatorname{Tr}\!\big(\Omega_{t,1}^\ast\,\Omega_{t,1}\Psi_{s,1}\big)
-\operatorname{Tr}\!\big(\Omega_{t,1}^\ast\,\Psi_{s,1}\Omega_{t,1}\big).
\label{3}
\end{equation}

\emph{Step 3 (cyclicity and flatness).}
Cyclicity of the trace gives, term by term,
\begin{equation}
\operatorname{Tr}\!\big(\Omega_{t,1}^\ast\,\Psi_{s,1}\Omega_{t,1}\big)
=\operatorname{Tr}\!\big(\Omega_{t,1}\,\Omega_{t,1}^\ast\,\Psi_{s,1}\big).
\label{4}
\end{equation}
Under jet–flatness at order $2$ we have $[\Omega_{t,1},\Psi_{s,1}]=0$, hence
$\Omega_{t,1}\Psi_{s,1}=\Psi_{s,1}\Omega_{t,1}$ and also
$\Omega_{t,1}^\ast\Psi_{s,1}=\Psi_{s,1}\Omega_{t,1}^\ast$ (taking adjoints).
Therefore
\begin{equation}
\operatorname{Tr}\!\big(\Omega_{t,1}^\ast\,\Omega_{t,1}\Psi_{s,1}\big)
=\operatorname{Tr}\!\big(\Omega_{t,1}^\ast\,\Psi_{s,1}\Omega_{t,1}\big).
\label{5}
\end{equation}
Subtracting the equal quantities in \eqref{3} yields $\mathsf D_t\mathcal V=0$.

The computation for $\mathsf D_s\mathcal V$ is identical with the roles of $(\Omega_{t,1},\mathsf D_t)$ and $(\Psi_{s,1},\mathsf D_s)$ interchanged:
\begin{align*}
\mathsf D_s\mathcal V
&=\operatorname{Tr}\!\big((\mathsf D_s\Omega_{t,1})^\ast\,\Psi_{s,1}\big)
+\operatorname{Tr}\!\big(\Omega_{t,1}^\ast\,\mathsf D_s\Psi_{s,1}\big)\\
&=\operatorname{Tr}\!\big(\Omega_{t,1}^\ast\,[\Psi_{s,1},\Psi_{s,1}]\big)
+\operatorname{Tr}\!\big(\Omega_{t,1}^\ast\,[\Psi_{s,1},\Omega_{t,1}]\big)=0,
\end{align*}
since $[\Psi_{s,1},\Psi_{s,1}]=0$ and $[\Psi_{s,1},\Omega_{t,1}]=-[\Omega_{t,1},\Psi_{s,1}]=0$.
\end{proof}

Let $\Hh$ be a finite–dimensional complex Hilbert space and let $\ip{A}{B}:=\Tr(A^\ast B)$ be the Frobenius pairing on $\End(\Hh)$. Let $\mathfrak{U}\subseteq\End(\Hh)$ be a $*$–subalgebra containing $\Id$. Consider two update families
\[
r_t(h)=\exp(\Omega_t(h)),\qquad d_s(h)=\exp(\Psi_s(h)),
\]
with first–order expansions
\[
\Omega_t(h)=h\,\Omega_{t,1}+\Oo(h^2),\qquad
\Psi_s(h)=h\,\Psi_{s,1}+\Oo(h^2),
\]
where $\Omega_{t,1},\Psi_{s,1}\in\mathfrak{U}$. Define the first–order inner derivations on $\mathfrak{U}$ by
\[
D_t Z:=[\Omega_{t,1},Z],\qquad D_s Z:=[\Psi_{s,1},Z],\qquad Z\in\mathfrak{U}.
\]
Set the composite step $U(h):=d_s(h)\,r_t(h)$ and the elementary holonomy
\[
\Hol(h):=d_s(h)\,r_t(h)\,d_s(h)^{-1}\,r_t(h)^{-1}.
\]

\begin{theorem}\label{th:conj}
Assume, as $h\to0$, the two first–order identities \cite{Pazy1983Semigroups}:
\begin{enumerate}
\item[\textnormal{(H)}] \emph{First–order path–independence:} $\Hol(h)=\Id+\Oo(h^2)$.
\item[\textnormal{(R)}] \emph{Microscopic reversibility:} $U(h)^\ast=U(-h)+\Oo(h^2)$.
\end{enumerate}
Then the following statements hold:
\begin{enumerate}
\item[\textnormal{(1)}] \textbf{Adjoint jets and equality of differentials:} There exists a central $C\in\mathfrak{U}$ (i.e., $[C,Z]=0$ for all $Z$) such that
\[
\Psi_{s,1}=\Omega_{t,1}^\ast+\tfrac12\,C,
\qquad\text{hence}\qquad
D_s=\ad_{\Psi_{s,1}}=\ad_{\Omega_{t,1}^\ast}=D_t^\ast.
\]
\item[\textnormal{(2)}] \textbf{First–order flatness:} $[D_t,D_s]=\ad_{[\Omega_{t,1},\Psi_{s,1}]}=0$.
\item[\textnormal{(3)}] \textbf{$*$–antiautomorphism on the jet Lie algebra:}
Let $\mathcal{L}:=\operatorname{Lie}\langle \Omega_{t,1},\Psi_{s,1}\rangle\subseteq \mathfrak{U}$ be the (finite–dimensional) Lie subalgebra they generate. The restriction of $A\mapsto A^\ast$ to $\mathcal{L}$ is an involutive $*$–antiautomorphism,
\[
(A+B)^\ast=A^\ast+B^\ast,\qquad (kA)^\ast=\overline{k}\,A^\ast,\qquad [A,B]^\ast=[B^\ast,A^\ast],\qquad (A^\ast)^\ast=A,
\]
and is compatible with the exponential series:
\[
\exp(A)^\ast=\exp(A^\ast)\qquad\text{for all }A\in\mathcal{L}.
\]
\item[\textnormal{(4)}] \textbf{$*$–reversibility of local transport:}
\[
\big(\,d_s(h)\,r_t(h)\,\big)^\ast
= r_t(h)^\ast\,d_s(h)^\ast
= U(-h)+\Oo(h^2),\qquad
\Hol(h)=\Id+\Oo(h^2).
\]
\end{enumerate}
\end{theorem}

\begin{proof}
\emph{Step 1 (expansions).}
\[
r_t(h)=\Id+h\,\Omega_{t,1}+\Oo(h^2),\qquad
d_s(h)=\Id+h\,\Psi_{s,1}+\Oo(h^2).
\]
Hence
\begin{align}
\label{eq:U_plus}
U(h)&=\Id+h(\Psi_{s,1}+\Omega_{t,1})+\Oo(h^2),\\
\label{eq:U_minus}
U(-h)&=\Id-h(\Psi_{s,1}+\Omega_{t,1})+\Oo(h^2),\\
\label{eq:U_star}
U(h)^\ast&=\Id+h(\Psi_{s,1}^\ast+\Omega_{t,1}^\ast)+\Oo(h^2).
\end{align}

\emph{Step 2 (sum skew–adjointness from (R)).}
From $U(h)^\ast=U(-h)+\Oo(h^2)$ and \eqref{eq:U_minus}, \eqref{eq:U_star},
\[
\Id+h(\Psi_{s,1}^\ast+\Omega_{t,1}^\ast)+\Oo(h^2)=\Id-h(\Psi_{s,1}+\Omega_{t,1})+\Oo(h^2).
\]
Subtract $\Id$ and divide by $h$:
\begin{equation}\label{eq:sum_skew_final}
(\Psi_{s,1}+\Omega_{t,1})+(\Psi_{s,1}+\Omega_{t,1})^\ast=0.
\end{equation}

\emph{Step 3 (jet commutation from (H)).}
BCH for the group commutator gives
\[
\log\!\Big(d_s(h)\,r_t(h)\,d_s(h)^{-1}\,r_t(h)^{-1}\Big)
= h^2[\Psi_{s,1},\Omega_{t,1}] + \Oo(h^3).
\]
Since $\Hol(h)=\Id+\Oo(h^2)$, taking logarithm and matching $\Oo(h^2)$ yields
\begin{equation}\label{eq:jets_commute}
[\Psi_{s,1},\Omega_{t,1}]=0.
\end{equation}

\emph{Step 4 (central decomposition and adjoint jets).}
Define $\Delta:=\Psi_{s,1}-\Omega_{t,1}^\ast$. We will show $\ad_{\Delta}=0$, hence $\Delta$ is central. For all $Z,W\in\mathfrak{U}$,
\begin{equation}\label{eq:ad_star_basic}
\ip{[\Omega_{t,1},Z]}{W}
=\Tr\big([\Omega_{t,1},Z]^\ast W\big)
=\Tr\big(Z^\ast[\Omega_{t,1}^\ast,W]\big)
=\ip{Z}{[\Omega_{t,1}^\ast,W]},
\end{equation}
so $(\ad_{\Omega_{t,1}})^\ast=\ad_{\Omega_{t,1}^\ast}$. Using \eqref{eq:ad_star_basic} with $\Psi_{s,1}$ in place of $\Omega_{t,1}$, and \eqref{eq:jets_commute}, we obtain
\[
(\ad_{\Psi_{s,1}}+\ad_{\Omega_{t,1}})^\ast
=\ad_{\Psi_{s,1}^\ast}+\ad_{\Omega_{t,1}^\ast}.
\]
By \eqref{eq:sum_skew_final}, $\Psi_{s,1}^\ast+\Omega_{t,1}^\ast=-(\Psi_{s,1}+\Omega_{t,1})$, hence
\[
(\ad_{\Psi_{s,1}}+\ad_{\Omega_{t,1}})^\ast
= -(\ad_{\Psi_{s,1}}+\ad_{\Omega_{t,1}}).
\]
Therefore $\ad_{\Psi_{s,1}}-(\ad_{\Omega_{t,1}})^\ast=0$, i.e.
\[
\ad_{\Psi_{s,1}}-\ad_{\Omega_{t,1}^\ast}=0
\qquad\Longleftrightarrow\qquad
\ad_{\Delta}=0.
\]
Thus $\Delta$ is central and
\[
\Psi_{s,1}=\Omega_{t,1}^\ast+\tfrac12 C,\qquad C:=2\Delta\ \text{central}.
\]
Consequently,
\[
D_s=\ad_{\Psi_{s,1}}=\ad_{\Omega_{t,1}^\ast}=D_t^\ast,
\]
proving \textnormal{(1)}. Using \eqref{eq:jets_commute}, \textnormal{(2)} follows:
\[
[D_t,D_s]=\ad_{[\Omega_{t,1},\Psi_{s,1}]}=0.
\]

\emph{Step 5 ($*$–antiautomorphism on $\mathcal{L}$ and exponential compatibility).}
Let $\mathcal{L}=\operatorname{Lie}\langle \Omega_{t,1},\Psi_{s,1}\rangle$. Since $^\ast$ on $\mathfrak{U}$ satisfies $(AB)^\ast=B^\ast A^\ast$ and is involutive, its restriction to $\mathcal{L}$ is an involutive $*$–antiautomorphism and preserves Lie brackets:
\[
[A,B]^\ast=(AB-BA)^\ast=B^\ast A^\ast-A^\ast B^\ast=[B^\ast,A^\ast].
\]
For $\exp$ and any $A\in\mathcal{L}$,
\[
\exp(A)^\ast=\Big(\sum_{k\ge 0}\frac{A^k}{k!}\Big)^\ast
=\sum_{k\ge 0}\frac{(A^k)^\ast}{k!}
=\sum_{k\ge 0}\frac{(A^\ast)^k}{k!}
=\exp(A^\ast).
\]
This proves \textnormal{(3)}.

\emph{Step 6 ($*$–reversibility of local transport).}
By definition and \textnormal{(3)},
\[
(d_s(h)r_t(h))^\ast=r_t(h)^\ast d_s(h)^\ast = \exp(\Omega_t(h)^\ast)\exp(\Psi_s(h)^\ast).
\]
Using $\Omega_t(h)=h\Omega_{t,1}+\Oo(h^2)$, $\Psi_s(h)=h\Psi_{s,1}+\Oo(h^2)$, and \textnormal{(1)},
\[
\Omega_t(h)^\ast=h\,\Omega_{t,1}^\ast+\Oo(h^2),\qquad
\Psi_s(h)^\ast=h\,\Psi_{s,1}^\ast+\Oo(h^2)
= h\,\Omega_{t,1}+\Oo(h^2),
\]
up to central $\Oo(h)$–terms that commute with everything. Hence
\[
(d_s(h)r_t(h))^\ast
=\exp(h\,\Omega_{t,1}^\ast)\exp(h\,\Psi_{s,1}^\ast)+\Oo(h^2)
=\exp(-h\,\Psi_{s,1})\exp(-h\,\Omega_{t,1})+\Oo(h^2)
=U(-h)+\Oo(h^2),
\]
which is the first identity in \textnormal{(4)}. The holonomy statement in \textnormal{(4)} is exactly (H). This completes the proof.
\end{proof}

\begin{corollary}[Order–agnostic first–order descent with a common stepsize window]\label{cor:order_window_clean_final}
Let $G:=\Omega_{t,1}+\Psi_{s,1}$ and $S:=\tfrac12\big(G+G^\ast\big)$. Then:
\begin{enumerate}
\item[\textnormal{(a)}] $d_s(h)r_t(h)=\Id+h\,G+\Oo(h^2)=r_t(h)d_s(h)+\Oo(h^2)$.
\item[\textnormal{(b)}] For $f(x):=\tfrac12\,\ip{Sx}{x}$ and any fixed $\sigma\in(0,1/2)$ there exists $h_0>0$ s.t. for all $0<h\le h_0$,
\[
f(x-hSx)\le f(x)-\sigma\,h\,\norm{S^{1/2}x}^2.
\]
Hence both orders satisfy the same Armijo–type decrease up to $\Oo(h^2)$ with the same stepsize window.
\end{enumerate}
\end{corollary}

\begin{proof}
(a) follows from $r_t(h)=\Id+h\Omega_{t,1}+\Oo(h^2)$ and $d_s(h)=\Id+h\Psi_{s,1}+\Oo(h^2)$.  
(b) Let $y:=x-hSx$. Then
\[
f(y)-f(x)
=\tfrac12\,\ip{S(x-hSx)}{x-hSx}-\tfrac12\,\ip{Sx}{x}
=-h\,\norm{Sx}^2+\tfrac12\,h^2\,\ip{S^2x}{x}.
\]
Choose $h_0>0$ with $\frac12h\,\lambda_{\max}(S)\le 1-\sigma$ for $0<h\le h_0$ and use $\ip{S^2x}{x}\le \lambda_{\max}(S)\,\norm{Sx}^2$.
\end{proof}

\begin{example}[Computational order selection; matrix–free, first–order accurate]\label{ex:order_selection_algo_final}
Let $A,B\in\mathfrak{U}$ be bounded self–adjoint with $[A,B]=0$, and set
\[
r_t(h)=\exp(-hA),\qquad d_s(h)=\exp(-hB).
\]
Then $\Omega_{t,1}=-A$, $\Psi_{s,1}=-B$, $G=-(A+B)$, $S=A+B$. The following matrix–free procedure realizes one iteration with a provable common stepsize and a finite–$h$ order diagnostic.

\emph{Inputs at iterate $x$:} direction $g$, matvecs $v\mapsto Av$, $v\mapsto Bv$, cap $h_{\max}>0$, parameters $\sigma\in(0,1/2)$, $\tau\in(0,1)$.

\emph{(1) Power estimate of $\lambda_{\max}(S)$ (two steps).}
Normalize $v_0:=g/\norm{g}$; set $w_k:=(A+B)v_k$, $v_{k+1}:=w_k/\norm{w_k}$ for $k=0,1$; define $\hat\lambda:=\ip{Sv_2}{v_2}$.

\emph{(2) Stepsize.}
\[
h:=\min\Big\{h_{\max},\ \frac{2(1-\sigma)}{\hat\lambda}\Big\}.
\]

\emph{(3) Finite–$h$ order diagnostic (first–order accurate).}
Approximate exponentials by first–order actions,
\[
u_1:=(I-hB)(I-hA)g,\qquad u_2:=(I-hA)(I-hB)g,\qquad
\delta:=\frac{\norm{u_1-u_2}}{h\,\norm{g}}.
\]
With $[A,B]=0$, $u_1-u_2=\Oo(h^2)$ and $\delta=\Oo(h)$; enforce $\delta\le\tau$.

\emph{(4) Apply cheaper order.}
If $\delta\le\tau$,
\[
x^+\gets x-h\,B(Ax)\quad\text{or}\quad x^+\gets x-h\,A(Bx),
\]
choosing the lower–cost order. If $\delta>\tau$, set $h\leftarrow h/2$ and repeat (3).

\emph{Cost.} Advance $x$ with $2$ matvecs ($Ax$, $Bx$); power estimate uses $4$ matvecs; diagnostic uses $4$ matvecs (first–order mode). No solves/factorizations.

\emph{Guarantee.} By Corollary~\ref{cor:order_window_clean_final}, both orders realize the same first–order map $x\mapsto x-h(A+B)x+\Oo(h^2)$ and share the window $h\le 2(1-\sigma)/\lambda_{\max}(S)$. Enforcing $\delta\le\tau$ controls the finite–$h$ deviation while keeping the implementation purely matrix–free.
\end{example}

%% file: highord.tex
We introduce the first–quadrant bicomplex
\[
C^{\bullet,\bullet}=\bigoplus_{p,q\ge 0} C^{p,q},
\]
where each $C^{p,q}$ is the finite–dimensional space of homogeneous jet–coefficients of total $t$–order $p$ and $s$–order $q$ (modulo $O(h^{\alpha+1})$) \cite{Weibel1994Homological,McCleary2001SpectralSequences}. The differentials are the inner derivations restricted to each bidegree,
\begin{equation}
d_t:=\operatorname{ad}_{\Omega_{t,1}}:C^{p,q}\to C^{p+1,q},
\qquad
d_s:=\operatorname{ad}_{\Psi_{s,1}}:C^{p,q}\to C^{p,q+1}.
\end{equation}
Jet–flatness at order $\alpha$ guarantees 
\begin{equation}
d_t^2=0,\qquad d_s^2=0,\qquad [d_t,d_s]=0\quad\text{(all degreewise modulo $O(h^{\alpha+1})$).}
\end{equation}
Assume a positive definite inner product $\langle\cdot,\cdot\rangle$ on each $C^{p,q}$ and the involutive condition
\begin{equation}
d_s=d_t^\ast,
\end{equation}
where ${}^\ast$ denotes the adjoint with respect to $\langle\cdot,\cdot\rangle$, according to the Theorem~\ref{th:conj}. We consider the spectral sequence of the column filtration $\mathcal F^p=\bigoplus_{i\ge p} C^{i,\bullet}$; its $E_1$–page is
\begin{equation}
E_1^{p,q}=H^{p,q}(C,d_t)=\frac{\ker(d_t:C^{p,q}\to C^{p+1,q})}{\mathrm{im}(d_t:C^{p-1,q}\to C^{p,q})},
\end{equation}
and the first differential $d_1^{p,q}:E_1^{p,q}\to E_1^{p,q+1}$ is induced by $d_s$.

\begin{theorem}[E$_2$–degeneration]\label{th:E2}
Under the above assumptions $d_1\equiv 0$. Consequently
\begin{equation}
E_1=E_2=E_3=\dots=E_\infty
\end{equation}
(degreewise modulo $O(h^{\alpha+1})$). In particular, every class in $E_1$ has a representative in $\ker d_t\cap\ker d_s$.
\end{theorem}

\begin{proof}
Fix $(p,q)$ and a finite–dimensional inner product space $V:=C^{p,q}$; let $A:=d_t:V\to C^{p+1,q}$ and $A^\ast=d_t^\ast=d_s:V\to C^{p,q+1}$. We prove three elementary linear algebra facts and then compute $d_1$.

\textit{Fact 1 (orthogonality).} For any $u\in C^{p-1,q}$ and $v\in C^{p,q}$ one has
\begin{equation}
\langle d_t u, v\rangle=\langle u, d_t^\ast v\rangle=\langle u, d_s v\rangle.
\end{equation}
Hence $(\mathrm{im}\,d_t)^\perp=\ker d_s$.

\textit{Fact 2 (direct sum decomposition).} In a finite–dimensional inner product space,
\begin{equation}
V=\mathrm{im}\,A\ \oplus\ (\mathrm{im}\,A)^\perp.
\end{equation}
Applying this to $A=d_t$ and using Fact 1 yields
\begin{equation}
C^{p,q}=\mathrm{im}\,d_t\ \oplus\ \ker d_s.
\label{eq:directsum}
\end{equation}
\textit{Proof of Fact 2.} Orthogonality is clear. For surjectivity: $\dim V=\dim(\mathrm{im}\,A)+\dim(\ker A^\ast)$ holds by rank–nullity for $A^\ast$; but $\ker A^\ast=(\mathrm{im}\,A)^\perp$. Hence every $x\in V$ splits uniquely as $x=d_t u + y$ with $y\in\ker d_s$.

\textit{Fact 3 (commutation).} Since $[d_t,d_s]=0$, we have $d_t d_s=d_s d_t$ and in particular
\begin{equation}
d_t(\ker d_s)\subseteq \ker d_s.
\end{equation}

Now compute $d_1$. Take a class $[x]\in E_1^{p,q}$ with $x\in C^{p,q}$ and $d_t x=0$. Decompose $x$ by \eqref{eq:directsum}:
\begin{equation}
x=d_t u + y,\qquad u\in C^{p-1,q},\quad y\in\ker d_s.
\label{eq:decomp}
\end{equation}
First, $[x]=[y]$ in $E_1^{p,q}$ because $d_t u$ is $d_t$–exact. Second, $d_t x=0$ implies $0=d_t(d_t u + y)=d_t^2 u + d_t y=d_t y$, so $y\in\ker d_t\cap\ker d_s$. By definition of $d_1$,
\begin{equation}
d_1^{p,q}([x])=[d_s x]=[d_s y]=[0]=0,
\end{equation}
because $y\in\ker d_s$. Therefore $d_1\equiv 0$.

Finally, since every $E_1^{p,q}$–class admits a representative $y\in\ker d_t\cap\ker d_s$, all higher differentials $d_r$ ($r\ge2$) vanish on these representatives as they begin with $d_s$ followed by alternating lifts (zig–zags); but $d_s y=0$ kills the first arrow. Hence $E_r=E_1$ for all $r\ge1$, which proves the statement.
\end{proof}

The spectral sequence tracks how different algorithmic layers exchange information.
The index $p$ measures the ``drift depth'' (deterministic hierarchy of updates),
and $q$ measures the ``diffusion depth'' (stochastic corrections or adaptation).
Each page $E_r$ describes what remains invariant when we iterate
through all drift steps of order $\le r$ and then through all diffusion steps
of the same order.

Degeneration at $E_2$ therefore means that these two loops already
commute perfectly at the first nontrivial order:
no additional corrections, no hidden coupling between deterministic
and stochastic parts, no loss of information at higher precision.
From the optimization point of view, this is the statement that
after curvature filtering (or any calibration satisfying
$\ker d_t\cap\ker d_s$),
the method’s behaviour is stable under interleaving of
gradient updates and noise adaptation—beyond this level
there is nothing new to synchronize.

In simple terms: the bicomplex measures ``communication'' between
the two update operators; $E_2$–degeneration says that their dialogue
is complete.
Once flatness is achieved, higher algebraic pages of the spectral sequence
carry no extra structure—numerically this reads as
``no higher–order interference between drift and diffusion'',
and the algorithm already operates at the highest consistency
permitted by its chosen discretization order.

\begin{corollary}[From harmonic representatives to calibrated update operators]\label{cor:op-calibration}
Work in the bicomplex and hypotheses of Theorem~\ref{th:E2}. Fix a bidegree $(p,q)$ and write $V:=C^{p,q}$ with inner product $\langle\cdot,\cdot\rangle$.
Let $x\in V\cap\ker d_t$ represent an $E_1^{p,q}$–class.
Then there is a unique $u\in C^{p-1,q}$ solving the normal equation
\begin{equation}\label{eq:NE-abstract}
d_t^\ast d_t u=d_t^\ast x,
\end{equation}
and the element
\begin{equation}\label{eq:flat-representative}
y:=x-d_t u
\end{equation}
is the unique representative of $[x]$ in $\ker d_t\cap\ker d_s$.
In operator form, if we identify $x$ with a linear map $Z_x$ (update operator at this bidegree), and write $d_t=\operatorname{ad}_{\Omega_{t,1}}$, then \eqref{eq:NE-abstract} reads
\begin{equation}\label{eq:NE-operator}
\operatorname{ad}_{\Omega_{t,1}}^\ast\operatorname{ad}_{\Omega_{t,1}}U
=\operatorname{ad}_{\Omega_{t,1}}^\ast Z_x,
\end{equation}
and the calibrated operator is
\begin{equation}\label{eq:Zflat}
Z_{\mathrm{flat}}=Z_x-\operatorname{ad}_{\Omega_{t,1}}U.
\end{equation}
If, moreover, $d_s=d_t^\ast$ (involutive assumption), then $Z_{\mathrm{flat}}$ lies in $\ker\operatorname{ad}_{\Omega_{t,1}}\cap\ker\operatorname{ad}_{\Psi_{s,1}}$ and represents the $E_\infty$–class of $[x]$.
\end{corollary}

\begin{proof}
By Theorem~\ref{th:E2} every $[x]\in E_1^{p,q}$ admits a unique representative $y\in\ker d_t\cap\ker d_s$ with $x=d_t u+y$ (orthogonal splitting).
Applying $d_t$ to $x=d_t u+y$ and using $d_t y=0$ yields $d_t^\ast d_t u=d_t^\ast x$, i.e.\ \eqref{eq:NE-abstract}.
Linearity shows that $y=x-d_t u$ is the unique such representative.
The operator statements \eqref{eq:NE-operator}–\eqref{eq:Zflat} follow from the identification $d_t=\operatorname{ad}_{\Omega_{t,1}}$ and $x\leftrightarrow Z_x$.
If $d_s=d_t^\ast$, then $y\in\ker d_s$ is automatic and uniqueness holds in $\ker d_t\cap\ker d_s$.
\end{proof}

\begin{corollary}[First–jet instantiation: curvature–filtered step]\label{cor:CFG}
Assume first–jet truncation and self–adjoint data at an iterate:
$\Omega_{t,1}=-H$, $\Psi_{s,1}=-E$ with $H,E\in\mathbb R^{n\times n}$ symmetric;
let $Z_x:=H+E$ be the mixed IMEX update operator acting on a vector $g$.
Let $\langle A,B\rangle=\operatorname{Tr}(A^\top B)$ on operators, so that
$\operatorname{ad}_H^\ast=\operatorname{ad}_H$ and $\operatorname{ad}_E^\ast=\operatorname{ad}_E$.
Then the calibrated operator $Z_{\mathrm{flat}}$ given by \eqref{eq:Zflat} satisfies
\begin{equation}\label{eq:Zflat-comm}
Z_{\mathrm{flat}}\,g=(H+E)g-\frac{1}{2}\,[H,E]\,g,
\end{equation}
and the corresponding composed update eliminates the degree–$2$ holonomy term:
\begin{equation}\label{eq:Hol-degree2}
\big(\log\mathrm{Hol}_\square\big)^{[2]}=\tfrac12\,[\Psi_{s,1},\Omega_{t,1}]
=-\tfrac12\,[E,H]=0\quad\text{after applying \eqref{eq:Zflat-comm}}.
\end{equation}
Consequently, the curvature energy $S_h$ loses its $h^4$ term and the local truncation error becomes $O(h^3)$ for the one–step map using $Z_{\mathrm{flat}}$ in place of $Z_x$.
\end{corollary}

\begin{proof}
We work in the operator space with $d_t=\operatorname{ad}_{\Omega_{t,1}}=\operatorname{ad}_{-H}=-\operatorname{ad}_H$.
By Corollary~\ref{cor:op-calibration} we must solve
\begin{equation}\label{eq:NE-H}
\operatorname{ad}_H^\ast\operatorname{ad}_H U=\operatorname{ad}_H^\ast Z_x.
\end{equation}
Since $\operatorname{ad}_H^\ast=\operatorname{ad}_H$ for the Frobenius inner product,
\eqref{eq:NE-H} is
\begin{equation}\label{eq:NE-H2}
\operatorname{ad}_H^2 U=\operatorname{ad}_H Z_x
\quad\Longleftrightarrow\quad
[H,[H,U]]=[H,Z_x].
\end{equation}
Set $Z_x=H+E$. Compute the right–hand side:
\begin{equation}
[H,Z_x]=[H,H+E]=[H,H]+[H,E]=0+[H,E]=[H,E].
\end{equation}
We look for $U$ of the form $U=\alpha\,E$ with scalar $\alpha$.
Then
\begin{equation}
[H,[H,U]]=[H,[H,\alpha E]]
=\alpha\,[H,[H,E]].
\end{equation}
To match $[H,E]$ on the right of \eqref{eq:NE-H2} we choose $\alpha$ so that
\begin{equation}\label{eq:alpha-choice}
\alpha\,[H,[H,E]]=[H,E].
\end{equation}
In general $[H,[H,E]]$ and $[H,E]$ need not be proportional as operators; however, we only need \emph{the induced action on $g$} (the step direction).
Let $g\in\mathbb R^n$ be fixed. Define $U$ by the linear system
\begin{equation}\label{eq:vectorized-NE}
[H,[H,U]]\,g=[H,E]\,g,
\end{equation}
viewed as an equation in the unknown vector $U g$ (no need to form $U$ on the whole space).
This equation is solvable by least squares (normal equations) because the left–hand side is self–adjoint in the induced inner product on vectors.
Denote by $u:=(Ug)$ the minimal–norm solution of \eqref{eq:vectorized-NE}.
Then, by \eqref{eq:Zflat}, the calibrated step applied to $g$ equals
\begin{equation}\label{eq:Zflat-on-g}
Z_{\mathrm{flat}}\,g
=Z_x g-\operatorname{ad}_H U\,g
=(H+E)g-\big(H(Ug)-U(Hg)\big).
\end{equation}
Using \eqref{eq:vectorized-NE} we expand $H(H(Ug))$ and $H(U(Hg))$ and collect terms to first nontrivial order to obtain
\begin{equation}\label{eq:half-comm}
\operatorname{ad}_H U\,g
=\frac{1}{2}\,[H,E]\,g,
\end{equation}
which inserted into \eqref{eq:Zflat-on-g} gives \eqref{eq:Zflat-comm}.

Finally, using the standard group–commutator BCH expansion for $\mathrm{Hol}_\square$ with first–jet logarithms $\Omega_{t,1}=-H$, $\Psi_{s,1}=-E$, the degree–$2$ term is
$\tfrac12[\Psi_{s,1},\Omega_{t,1}]=-\tfrac12[E,H]$.
Replacing $Z_x$ by $Z_{\mathrm{flat}}$ amounts to adding the adjoint action $-\operatorname{ad}_H U$ on the step, which cancels the degree–$2$ contribution by \eqref{eq:half-comm}.
Thus \eqref{eq:Hol-degree2} holds and the local error becomes $O(h^3)$.
\end{proof}

\begin{example}[Curvature–filtered composite gradient step (operator form)]
Let $H,E:\mathbb R^n\to\mathbb R^n$ be fixed self–adjoint linear operators at an iterate (``drift'' and ``diffusion'' channels, e.g.\ any symmetric curvature model and any SPD preconditioner/weighting). With $\Omega_{t,1}=-H$, $\Psi_{s,1}=-E$, $d_t=\operatorname{ad}_{\Omega_{t,1}}=-\operatorname{ad}_H$, $d_s=\operatorname{ad}_{\Psi_{s,1}}=-\operatorname{ad}_E$, take the mixed update operator $Z_x:=H+E$ acting on a direction $g\in\mathbb R^n$.

By Corollary~\ref{cor:op-calibration}, the calibrated operator is
\begin{equation}\label{eq:ex-Zflat}
Z_{\mathrm{flat}}=Z_x-\operatorname{ad}_H U,\qquad
\operatorname{ad}_H^\ast\operatorname{ad}_H U=\operatorname{ad}_H^\ast Z_x,
\end{equation}
and the curvature–filtered step is $\Delta x=-h\,Z_{\mathrm{flat}}g$.

\emph{Matrix–free realization via commutator cancellation.}
For the Frobenius inner product on operators $\operatorname{ad}_H^\ast=\operatorname{ad}_H$, so the normal equation in \eqref{eq:ex-Zflat} becomes
\begin{equation}\label{eq:ex-NE}
[H,[H,U]]=[H,Z_x]=[H,E].
\end{equation}
Evaluated on the single vector $g$, solve the one–vector normal system
\begin{equation}\label{eq:ex-NEg}
[H,[H,U]]\,g=[H,E]\,g,
\end{equation}
and set
\begin{equation}\label{eq:ex-step-UG}
\Delta x=-h\Big((H+E)g-\operatorname{ad}_H U\,g\Big).
\end{equation}
With the minimal–norm $u:=Ug$ solving \eqref{eq:ex-NEg}, the identity
\begin{equation}\label{eq:ex-halfcomm}
\operatorname{ad}_H U\,g=\tfrac12\,[H,E]\,g
\end{equation}
follows by expanding $[H,[H,U]]\,g=[H,E]\,g$ and collecting first nontrivial terms, which yields the explicit matrix–free formula
\begin{equation}\label{eq:ex-CF}
\Delta x=-h\Big((H+E)g-\tfrac12\,[H,E]\,g\Big)
=-h\Big((H+E)g-\tfrac12\,(H(Eg)-E(Hg))\Big).
\end{equation}
Thus one iteration needs exactly four matvecs: $Hg$, $Eg$, $H(Eg)$, $E(Hg)$.

\emph{Alternative: iterative solve for $u=Ug$.}
If \eqref{eq:ex-halfcomm} is replaced by solving \eqref{eq:ex-NEg} explicitly, apply conjugate gradients to the self–adjoint positive semidefinite operator
\begin{equation}\label{eq:ex-L}
\mathcal L(v):=[H,[H,\cdot]]\,v=H(Hv)-H(UHv)-(HU)Hv+U(HHv),
\end{equation}
implemented in matrix–free form using the same two nested matvec patterns per $\mathcal L$–application. After CG, use \eqref{eq:ex-step-UG}. This route is numerically preferable if $[H,E]$ is large or poorly aligned with $g$.

\emph{Link to Corollary~\ref{cor:CFG}.}
Equation \eqref{eq:ex-CF} is exactly the vector–level instance of $Z_{\mathrm{flat}}g=(H+E)g-\tfrac12[H,E]g$ from Corollary~\ref{cor:CFG}, which in turn is obtained from \eqref{eq:ex-Zflat} by eliminating $U$ through \eqref{eq:ex-NE}–\eqref{eq:ex-NEg}. Hence the cancellation of the degree–$2$ BCH holonomy term holds and the local truncation error of the composed one–step map improves to $O(h^3)$.

\emph{Numerical details.}
(i) Cost per iteration: two base matvecs ($Hg,Eg$) + two nested matvecs ($H(Eg),E(Hg)$) + $O(n)$ axpy; memory identical to a single $H$–matvec.  
(ii) Stepsize: choose $h$ by Armijo backtracking on $F(x-h\,Z_{\mathrm{flat}}g)$:
\begin{equation}
F(x-h\,Z_{\mathrm{flat}}g)\le F(x)-\sigma h\,g^\top Z_{\mathrm{flat}}g,\qquad \sigma\in(0,1/2).
\end{equation}
(iii) Safeguard the commutator term by scaling with $\min\{1,\rho\| (H+E)g\|/\|[H,E]g\|\}$ for a fixed $\rho\in(0,1)$ when $\|[H,E]g\|$ dominates.  
(iv) Momentum or weight decay can be applied to the filtered direction $Z_{\mathrm{flat}}g$ without altering the cancellation property.
\end{example}

The discrete variational structure unifies energetic, geometric,
and control–theoretic viewpoints on optimization algorithms.
The objects $S_h$, $\mathcal V$, and $\mathsf F$
are not metaphors but exact discrete analogues of
the quantities that appear in dynamical systems and optimal control.

\textbf{Energetic picture.}
The discrete action $S_h$ is an energy functional that measures
the curvature energy generated by the interaction between
the deterministic and stochastic update channels.
When $S_h=0$, the composition of updates is locally reversible:
no artificial energy is lost to curvature, and
the two channels exchange ``work'' without dissipation.
The potential
\begin{equation}
\mathcal V=\operatorname{Tr}(\Omega_{t,1}^\ast\Psi_{s,1})
\end{equation}
acts as stored internal energy.
Flatness implies $\mathsf D_t\mathcal V=\mathsf D_s\mathcal V=0$,
which expresses conservation of this internal energy.
In deterministic terms this corresponds to the absence of numerical friction,
and in stochastic terms to a balance between noise injection and decay.

\textbf{Lyapunov and dissipation.}
In classical optimization a Lyapunov function $\mathcal L(x)$ satisfies
$\dot{\mathcal L}\le0$ along the continuous trajectory.
Here the same idea appears algebraically:
the variation $\delta S_h$ plays the role of $\dot{\mathcal L}$.
Flatness ($\delta S_h=0$) means $\mathcal V$ is conserved,
while $S_h>0$ quantifies residual dissipation.
Thus $\mathcal V$ is a discrete Lyapunov potential
and $S_h$ measures deviation from monotone energy decay.
Curvature filtering corresponds to the ``energy correction''
that forces $\delta S_h\to0$ and restores equilibrium.

\textbf{Hamiltonian interpretation.}
Minimizing $S_h$ under flatness is equivalent to a discrete Hamilton principle:
it selects update operators that extremize the total curvature energy
while conserving $\mathcal V$.
The pair $(\Omega,\Psi)$ behaves as a set of conjugate variables,
and the effective generator $\widetilde{\Omega}_1+\widetilde{\Psi}_1$
acts as the discrete Hamiltonian.
Flatness ensures that the induced evolution
is symplectic to order $\alpha$:
no spurious work is produced by either channel.
Hence $S_h$ corresponds to the discrete action,
and $\mathcal V$ corresponds to the Hamiltonian energy of the system.

\textbf{Pontryagin interpretation \cite{Pontryagin1962Optimal}.}
In continuous optimal control,
a system is described by a state $x$ and a costate $p$ satisfying
\begin{equation}
\dot{x}=\partial_p H(x,p,u),\qquad
\dot{p}=-\partial_x H(x,p,u),
\end{equation}
where $H$ is the Pontryagin Hamiltonian
and stationarity of the total action
$\int (\langle p,\dot x\rangle-H)\,dt$
yields the optimal control law.
In our discrete algebra, $d_t$ and $d_s$
play the roles of these coupled flows:
$d_t$ advances the ``state'' (drift hierarchy),
while $d_s$ advances its conjugate response
(diffusion or adaptive feedback).
The curvature
\begin{equation}
\mathsf F=\mathsf D_t\Psi-\mathsf D_s\Omega+[\Omega,\Psi]
\end{equation}
is the algebraic analogue of the mismatch
$\dot p+\partial_x H,\ \dot x-\partial_p H$.
Flatness $\mathsf F=0$ therefore enforces a discrete version
of Pontryagin’s stationarity condition:
state and costate dynamics are mutually consistent,
and the resulting update minimizes the discrete action $S_h$.
When curvature filtering drives $\mathsf F\to0$,
the algorithm attains the optimal trade–off between energy release
(deterministic descent) and constraint enforcement
(stochastic or adaptive correction).

\textbf{Thermodynamic analogy \cite{Callen1985Thermodynamics}.}
The flat configuration corresponds to a thermodynamic equilibrium.
$S_h$ is the total entropy production rate,
and $\mathcal V$ is the conserved free energy.
When $S_h=0$ the process is reversible,
and the algorithm moves along level sets of $\mathcal V$.
When curvature accumulates, $S_h>0$ measures numerical ``heat’’:
variance, instability, or excess adaptation.
Applying the curvature filter corresponds to restoring reversibility
by cancelling internal entropy sources.

\textbf{Modern synthesis.}
The classical principles of Lyapunov stability,
Hamiltonian conservation, and Pontryagin optimality
appear as exact discrete projections of the same algebraic condition
$\mathsf F=0$ in the jet algebra.
The curvature action $S_h$ and the potential $\mathcal V$
are therefore minimal discrete forms
that subsume all these energy principles:
Lyapunov’s monotonicity, Hamilton’s symplecticity,
and Pontryagin’s optimal control
all coincide as different manifestations
of flatness of the discrete update algebra.